\documentclass[10pt,reqno]{amsart}

\title[A calculus for flow categories]{A calculus for flow categories}
\author[Andrew Lobb]{Andrew Lobb}
\address{Department of Mathematical Sciences\\ Durham University, UK}
\email{andrew.lobb@durham.ac.uk}

\author[Patrick Orson]{Patrick Orson}
\address{Department of Mathematics,
Boston College, USA}
\email{patrick.orson@bc.edu}

\author[Dirk Sch\"utz]{Dirk Sch\"utz}
\address{Department of Mathematical Sciences\\ Durham University, UK}
\email{dirk.schuetz@durham.ac.uk}

\newlength{\myhmargin}  
\usepackage{amsmath,amssymb,amsthm,graphicx,mathrsfs,tikz,color}
\usetikzlibrary{decorations.pathmorphing,graphs}
\usepackage[colorlinks,
    linkcolor={red!50!black},
    citecolor={black!50!black},
    urlcolor={blue!80!black}]{hyperref}
\usepackage[all]{hypcap}

\newcommand{\score}[2]{
\draw[lightgray] (0,0) -- (#2,0);
\draw[lightgray] (0,#1) -- (#2,#1);
\draw[lightgray] (0,#1+#1) -- (#2,#1+#1);
\draw[lightgray] (0,#1+#1+#1) -- (#2,#1+#1+#1);}

\newcommand{\scorex}[3]{
\draw[lightgray] (#3,0) -- (#2+#3,0);
\draw[lightgray] (#3,#1) -- (#2+#3,#1);
\draw[lightgray] (#3,#1+#1) -- (#2+#3,#1+#1);
\draw[lightgray] (#3,#1+#1+#1) -- (#2+#3,#1+#1+#1);}

\newcommand{\scoret}[2]{
\draw[lightgray] (0,0) -- (#2,0);
\draw[lightgray] (0,#1) -- (#2,#1);
\draw[lightgray] (0,#1+#1) -- (#2,#1+#1);}

\newcommand{\scoretx}[3]{
	\draw[lightgray] (#3,0) -- (#2+#3,0);
	\draw[lightgray] (#3,#1) -- (#2+#3,#1);
	\draw[lightgray] (#3,#1+#1) -- (#2+#3,#1+#1);}

\newcommand{\dota}[1]{
\draw
    (#1, 1.8) node[circle, inner sep=0.04cm, fill=black, draw] {}
}

\newcommand{\hookt}[2]{
\draw
    (#1, 1.8) node[circle, inner sep=0.04cm, fill=black, draw] {} --
    (#1, 1.5) node[left] {#2} -- 
    (#1, 1.2) node[circle, inner sep=0.04cm, fill=black, draw] {}
}

\newcommand{\hookmit}[2]{
\draw
    (#1, 1.8) node[circle, inner sep=0.04cm, fill=black, draw] {} --
    (#1, 1.5) node[right] {#2} -- 
    (#1, 1.2) node[circle, inner sep=0.04cm, fill=black, draw] {}
}

\newcommand{\hooket}[2]{
\draw
    (#1, 0) node[circle, inner sep=0.04cm, fill=black, draw] {} -- 
    (#1 + 0.8, 1.2) node[circle, inner sep=0.04cm, fill=black, draw] {} -- 
    (#1 + 0.8, 1.5) node[right] {#2} -- 
    (#1 + 0.8, 1.8) node[circle, inner sep=0.04cm, fill=black, draw] {}
}

\newcommand{\hookmiet}[2]{
\draw
    (#1 ,1.8) node[circle, inner sep=0.04cm, fill=black, draw] {} --
    (#1, 1.5) node[left] {#2} -- 
    (#1 ,1.2) node[circle, inner sep=0.04cm, fill=black, draw] {} -- 
    (#1 + 0.8, 0) node[circle, inner sep=0.04cm, fill=black, draw] {} 
}

\newcommand{\hookbet}[3]{
\draw
    (#1, 0.6) node[circle, inner sep=0.04cm, fill=black, draw] {} -- 
    (#1, 0.3) node[left] {#3} -- 
    (#1, 0) node[circle, inner sep=0.04cm, fill=black, draw] {} -- 
    (#1 + 0.8, 1.2) node[circle, inner sep=0.04cm, fill=black, draw] {} -- 
    (#1 + 0.8, 1.5) node[left] {#2} -- 
    (#1 + 0.8, 1.8) node[circle, inner sep=0.04cm, fill=black, draw] {}
}

\newcommand{\hookteb}[3]{
\draw
    (#1 ,1.8) node[circle, inner sep=0.04cm, fill=black, draw] {} --
    (#1, 1.5) node[right] {#2} -- 
    (#1 ,1.2) node[circle, inner sep=0.04cm, fill=black, draw] {} -- 
    (#1 + 0.8, 0) node[circle, inner sep=0.04cm, fill=black, draw] {} -- 
    (#1 + 0.8, 0.3) node[right] {#3} -- 
    (#1 + 0.8, 0.6) node[circle, inner sep=0.04cm, fill=black, draw] {}
}

\newcommand{\hookbe}[2]{
\draw
    (#1, 0.6) node[circle, inner sep=0.04cm, fill=black, draw] {} -- 
    (#1, 0.3) node[left] {#2} -- 
    (#1, 0) node[circle, inner sep=0.04cm, fill=black, draw] {} -- 
    (#1 + 0.8 ,1.2) node[circle, inner sep=0.04cm, fill=black, draw] {} 
}

\newcommand{\hookmibe}[2]{
\draw
    (#1, 1.2) node[circle, inner sep=0.04cm, fill=black, draw] {} -- 
    (#1 + 0.8, 0) node[circle, inner sep=0.04cm, fill=black, draw] {} -- 
    (#1 + 0.8, 0.3) node[right] {#2} -- 
    (#1 + 0.8 ,0.6) node[circle, inner sep=0.04cm, fill=black, draw] {} 
}

\newcommand{\hookbem}[3]{
	\draw
	(#1, 0.6) node[circle, inner sep=0.04cm, fill=black, draw] {} -- 
	(#1, 0.3) node[left] {#2} -- 
	(#1, 0) node[circle, inner sep=0.04cm, fill=black, draw] {} -- 
	(#1 + 0.8 ,1.2) node[circle, inner sep=0.04cm, fill=black, draw] {} --
	(#1 + 0.8, 0.9) node[right]{#3} --
	(#1 + 0.8, 0.6) node[circle, inner sep=0.04cm, fill=black, draw] {}
}

\newcommand{\dotc}[1]{
\draw
    (#1, 0.6) node[circle, inner sep=0.04cm, fill=black, draw] {}
}

\newcommand{\hookb}[2]{
\draw
    (#1, 0.6) node[circle, inner sep=0.04cm, fill=black, draw] {} -- 
    (#1, 0.3) node[left] {#2} -- 
    (#1, 0) node[circle, inner sep=0.04cm, fill=black, draw] {} 
}

\newcommand{\hookmib}[2]{
\draw
    (#1, 0.6) node[circle, inner sep=0.04cm, fill=black, draw] {} -- 
    (#1, 0.3) node[right] {#2} -- 
    (#1, 0) node[circle, inner sep=0.04cm, fill=black, draw] {} 
}

\newcommand{\hookm}[2]{
\draw
    (#1, 1.2) node[circle, inner sep=0.04cm, fill=black, draw] {} -- 
    (#1, 0.9) node[left] {#2} -- 
    (#1, 0.6) node[circle, inner sep=0.04cm, fill=black, draw] {} 
}

\newcommand{\hookmim}[2]{
	\draw
	(#1, 1.2) node[circle, inner sep=0.04cm, fill=black, draw] {} -- 
	(#1, 0.9) node[right] {#2} -- 
	(#1, 0.6) node[circle, inner sep=0.04cm, fill=black, draw] {} 
}

\newcommand{\hookx}[1]{
\draw
	(#1, 0.6) node[circle, inner sep=0.04cm, fill=black, draw] {} -- 
	(#1 + 0.8, 1.8) node[circle, inner sep=0.04cm, fill=black, draw] {} 
}

\newcommand{\hookmix}[1]{
\draw
(#1 + 0.8, 0.6) node[circle, inner sep=0.04cm, fill=black, draw] {} -- 
(#1, 1.8) node[circle, inner sep=0.04cm, fill=black, draw] {} 
}

\newcommand{\hooke}[1]{
\draw
    (#1, 0) node[circle, inner sep=0.04cm, fill=black, draw] {} -- 
    (#1 + 0.8, 1.2) node[circle, inner sep=0.04cm, fill=black, draw] {} 
}

\newcommand{\hookmie}[1]{
\draw
    (#1 + 0.8, 0) node[circle, inner sep=0.04cm, fill=black, draw] {} -- 
    (#1, 1.2) node[circle, inner sep=0.04cm, fill=black, draw] {} 
}

\newcommand{\hookem}[2]{
	\draw
	(#1, 0) node[circle, inner sep=0.04cm, fill=black, draw] {} -- 
	(#1 + 0.8, 1.2) node[circle, inner sep=0.04cm, fill=black, draw] {} --
	(#1 + 0.8, 0.9) node[right]{#2} --
	(#1 + 0.8, 0.6) node[circle, inner sep=0.04cm, fill=black, draw] {}
}

\newcommand{\hookmiem}[2]{
	\draw
	(#1, 0.6) node[circle, inner sep=0.04cm, fill=black, draw] {} --
	(#1, 0.9) node[left]{#2} --
	(#1, 1.2) node[circle, inner sep=0.04cm, fill=black, draw] {} --
	(#1 + 0.8, 0) node[circle, inner sep=0.04cm, fill=black, draw] {}

}

\newcommand{\dotb}[1]{
\draw
    (#1 ,1.2) node[circle, inner sep=0.04cm, fill=black, draw] {} 
}

\newcommand{\dotd}[1]{
\draw
    (#1, 0) node[circle, inner sep=0.04cm, fill=black, draw] {} 
}

\newcommand{\shookl}[2] {
\draw[blue]
    (#1,#2) -- (#1 - 0.3, #2 + 0.2)
}

\newcommand{\shookr}[2] {
\draw[blue]
    (#1,#2) -- (#1 + 0.3, #2 + 0.2)
}

\newtheorem{theorem}{Theorem}[section]
\newtheorem{lemma}[theorem]{Lemma}

\newtheorem{proposition}[theorem]{Proposition}
\newtheorem{corollary}[theorem]{Corollary}
\newtheorem{conjecture}[theorem]{Conjecture}

\theoremstyle{definition}
\newtheorem{definition}[theorem]{Definition}
\newtheorem{construction}[theorem]{Construction}
\newtheorem{remark}[theorem]{Remark}
\newtheorem{example}[theorem]{Example}

\newcommand{\gr}[1]{|{#1}|}

\DeclareMathAlphabet{\mathpzc}{OT1}{pzc}{m}{it}
\newcommand{\E}{\mathbb{E}}

\newcommand{\Z}{\mathbb{Z}}

\newcommand{\Q}{\mathbb{Q}}
\newcommand{\R}{\mathbb{R}}
\newcommand{\F}{\mathbb{F}}
\renewcommand{\phi}{\varphi}
\newcommand{\SO}{\mathbf{SO}}
\newcommand{\Hom}{\text{Hom}}
\newcommand{\Eta}{\text{H}}
\DeclareMathOperator{\Ob}{Ob}
\DeclareMathOperator{\PL}{PL}

\newcommand{\M}{\mathcal{M}}
\newcommand{\Mo}{\mathscr{M}}
\newcommand{\Sp}{\mathscr{S}}
\newcommand{\B}{\mathscr{B}}

\newcommand{\omax}{\overline{\max}}
\newcommand{\new}{\text{new}}

\newcommand{\cC}{\mathscr{C}}
\newcommand{\cP}{\mathscr{P}}
\newcommand{\X}{\mathcal{X}}

\DeclareMathOperator{\Sq}{Sq}
\DeclareMathOperator{\id}{id}

\begin{document}
\parindent0em
\thispagestyle{empty}

\begin{abstract}

We describe a calculus of moves for modifying a framed flow category without changing the associated stable homotopy type.  We use this calculus to show that if two framed flow categories give rise to the same stable homotopy type of homological width at most three, then the flow categories are move equivalent. The process we describe is essentially algorithmic and can often be performed by hand, without the aid of a computer program.
\end{abstract}

\maketitle


\section{Introduction}
\label{sec:introduction}

Framed flow categories were introduced by Cohen-Jones-Segal \cite{CJS} as a way of encoding flow data associated with a Morse function or a Floer functional.  To a framed flow category they associated a finite CW spectrum up to stable homotopy: a \emph{stable homotopy type}.  The idea was that Floer cohomology should be recovered as the singular cohomology of a stronger invariant taking values in stable homotopy types.

This formalism was exploited in seminal work of Lipshitz and Sarkar \cite{LipSarKhov}, who associated a framed flow category to a link diagram.  The associated stable homotopy type is a link invariant that recovers Khovanov cohomology \cite{kh1} as its reduced singular cohomology.

In this paper we collect a set of four \emph{flow category moves} - perturbation, stabilization, handle cancellation, and the extended Whitney trick (Definition \ref{defn:moves}) - which change a framed flow category without changing the associated stable homotopy type. Handle cancellation and the extended Whitney trick are named for the analogous operations in the space of Morse functions. The analogy between flow category moves and moves in Morse Theory suggests the following conjecture:

\begin{conjecture}\label{conj:main} If two framed flow categories determine the same stable homotopy type then they are move equivalent -- i.e.\ related by a finite sequence of flow category moves.
\end{conjecture}

We say that a space or spectrum has \emph{homological width $r$} if the reduced homology and cohomology is only supported in degrees $n, n+1,\dots,n+r$, for some $n\in\Z$. The \emph{homological width of a framed flow category} is the width of the associated stable homotopy type. The main result of this paper is a confirmation of Conjecture \ref{conj:main} when the framed flow categories have homological width $3$:

\begin{theorem}
	\label{thm:main}
	Suppose $\cC_1$ and $\cC_2$ are framed flow categories such that there is a stable homotopy equivalence between their associated stable homotopy types. Suppose further that the reduced homology $\widetilde{H}_*(\cC_1;\Z)$ and cohomology $\widetilde{H}^*(\cC_2;\Z)$ are each supported only in degrees $n$, $n+1$, $n+2$, $n+3$, for some $n\in\Z$. Then $\cC_1$ and $\cC_2$ are move equivalent.
\end{theorem}

To prove Theorem \ref{thm:main}, we describe an algorithm for using the flow category moves to reduce a homological width $3$ flow category to one made up from a list of certain standard flow categories. The proof is then completed by an appeal to the ``uniqueness of decomposition'' component of a homotopy classification result due to Baues and Hennes \cite{BauHen}.

Baues and Hennes considered, for fixed $n\geq 4$, the homotopy types of $(n-1)$-connected $(n+3)$-dimensional polyhedra $X$. They proved that for any such $X$ there is a homotopy equivalence
\begin{equation}\label{eq:BH}
X\simeq X_1\vee X_2\vee\dots\vee X_N
\end{equation}
where the $X_i$ come from a given list of spaces that are in one-to-one correspondence with isomorphism classes of collections of algebraic data called \emph{stable $A^3_n$-systems} \cite[Def.~4.7]{BauHen}. They proved that moreover this decomposition is unique up to reordering. We call these $X_i$ the \emph{Baues-Hennes spaces}, and in Section \ref{sec:baues_hennes} we will describe a list of standard width 3 flow categories, each recovering a particular Baues-Hennes space up to stable homotopy. Our flow category reduction algorithm will take an arbitrary width 3 flow category and obtain a decomposition into these standard flow categories. The uniqueness part of the Baues-Hennes theorem can then be applied to show that if two such decomposed flow categories give homotopic spacial realisations, then these spacial realisations are the same Baues-Hennes decomposition, up to reordering. We prove in Lemma \ref{lem:def_BH} that flow categories whose spacial realisations give the same Baues-Hennes decomposition are moreover flow category move equivalent. This is enough to complete the proof of Theorem \ref{thm:main}.

On the way to constructing our flow category reduction algorithm, we obtain Corollary~\ref{cor:changclassification}, which is the homological width $2$ version of Theorem \ref{thm:main}. For $n\geq 3$, the homotopy types of $(n-1)$-connected $(n+2)$-dimensional polyhedra were first classified by Chang~\cite{Chang}, following Whitehead \cite{MR0030758}, and using the so-called Chang spaces. Though reproducing the Chang and Baues-Hennes results is not the main aim of this article, we show in Corollary~\ref{cor:changrecover} that our flow category method does give an alternative proof of Chang's classification, but only up to stable homotopy. This argument uses the result (Proposition \ref{prop:CWisflow}), which is of independent interest, that given a finite CW complex there exists a framed flow category that recovers it up to stable homotopy. Similarly our method recovers the existence of the Baues-Hennes decomposition of Equation (\ref{eq:BH}), up to stable homotopy, though we do not independently prove the uniqueness of this decomposition.

One of the attractive features of our approach to proving Theorem \ref{thm:main} is that it is algorithmic and largely based on graphical moves and calculations. Thus it often lends itself to effective practical computations by hand. At the end of this paper we apply our calculus to an example case of the Lipshitz-Sarkar framed flow category. This more computational work was continued in \cite{ALPODSn}, where we used the calculus of moves to produce a new combinatorial algorithm for computing the first two Steenrod squares in a framed flow category. Furthermore, a version of the calculus of moves has been implemented in a computer program due to the third author \cite{SchuetzKJ}.

\subsection{A rough idea of framed flow categories}
\label{subsec:rough_idea}

Let us start with just flow categories before coming to the modifier \emph{framed}. A flow category $\cC$ should be thought of as a way of keeping track of information about flowlines of a Morse function.  The object set $\Ob(\cC)$ of the category is finite and comes together with an integer grading $\gr{\cdot} : \Ob(\cC) \rightarrow \Z$.  The objects should be thought of as playing the role of the critical points of a Morse function, while the grading is just a relative Morse index.

If $x,y \in \Ob(\cC)$ are objects with $x \not= y$ then the set of morphisms from $x$ to $y$ is written $\M(x,y)$ and has the structure of a compact manifold-with-corners of dimension $\gr{x}- \gr{y} - 1$ (in particular when $\gr{y} \geq \gr{x}$, there are no morphisms).  This should be thought of as the unparametrized moduli space of flowlines from $x$ to $y$.

Flow categories satisfy certain axioms arising from the analogy with Morse theory.  For example the boundary of a moduli space of flowlines should be given by the broken flowlines
\[ \partial \M(x,y) = \bigcup_{z \in \Ob{\cC}} \M(z,y) \times \M(x,z) {\rm .} \]
Breaking at multiple intermediate objects is what gives rise to the cornered structure of the moduli spaces.

A flow category is called \emph{framed} when its moduli spaces come with framed embeddings into cornered Euclidean space, in a way compatible with the cornered structure of the moduli spaces.  The easiest case to visualize is when $\M(x,y)$ is a closed manifold.  In this case one should just think of $\M(x,y)$ as being embedded in some Euclidean space with a framing of the normal bundle.

\subsection{Framed flow category moves}
\label{subsec:fcc_conjecture}
Framed flow categories can be modified by perturbing the framed embeddings of the moduli spaces, or by increasing the dimension of the ambient cornered Euclidean spaces for these embeddings, while leaving the maps essentially unchanged. These operations, called \emph{perturbation} and \emph{stabilization}, do not alter the associated stable homotopy type \cite[Lemmas 3.25, 3.26]{LipSarKhov}.

In addition to these basic modifications, previous work of the authors (together with Jones) gives two further moves which do not alter the associated stable homotopy type. Together, we have the following four moves, whose definitions are recalled in Sections \ref{sec:stabpert}, \ref{sec:handlecancellation}, and \ref{sec:extendedwhitney}.

\begin{definition}
	\label{defn:moves}
The \emph{flow category moves} are:
\begin{enumerate}
	\item\label{item:perturbation} \textbf{Perturbation.}
	\item\label{item:stabilization} \textbf{Stabilization}.
	\item\label{item:handle} \textbf{Handle cancellation} (see \cite[Theorem 2.17]{JLS2}).
	\item\label{item:whitney} \textbf{Extended Whitney trick} (see \cite[\textsection 3.3]{ALPODS2}).
\end{enumerate}

The equivalence relation on the set of framed flow categories given by the transitive closure of the flow category moves is called \emph{move equivalence}.\footnote{In \cite[\textsection 2.2]{ALPODS2}, we also defined a move called a \textbf{handle slide} (see Section \ref{sec:handleslide} for definition), but this move can be derived from handle cancellation so does not appear as an additional flow category move.}
\end{definition}

The stable homotopy type is an invariant of the move equivalence class of a framed flow category.

\subsection{Graph notation}
\label{subsec:graph_notation}
Framed flow categories can be represented graphically by drawing the graph of the underlying category and decorating the edges with the corresponding framed moduli spaces. Many of the tools developed in this paper take the form of a graphical calculus based on the flow category moves.

\begin{definition}
	The graph $\Gamma=\Gamma(\cC,\iota,\phi)$ of a framed flow category is the graph with a vertex $v_a$ for each object $a$ of $\cC$ and an edge $e_{a,b}$ between $v_a$ and $v_b$ whenever the moduli space $\M(a,b)$ is non-empty. The vertices and edges of the graph are decorated with the corresponding objects and framed moduli spaces respectively.
	
	The length of an edge $e_{a,b}$ is defined by $||a|-|b||$. Given the graph $\Gamma$ of a framed flow category, the subgraph obtained by removing all edges of length greater than $n$ is written $\Gamma_n=\Gamma_n(\cC,\iota,\phi)$.
\end{definition}

This graphical description enables us to give an example for the reader new to flow categories.

\begin{example}\label{exam:Adem_relation}
	Consider the graph in Figure \ref{fig:Adem_relation}.  In this figure the number $2$ represents two positively framed points, while $\eta$ represents a circle with a certain framing. The objects are in five adjacent gradings.  Is it possible that this is the graph $\Gamma_2$ of a framed flow category $\cC$?
	\begin{figure}[ht]
		\begin{center}
			\begin{tikzpicture}
			\draw
			(0, 1) node[circle, inner sep=0.04cm, fill=black, draw] {} -- 
			(0, 0.5) node[left] {$2$} -- 
			(0, 0) node[circle, inner sep=0.04cm, fill=black, draw] {} -- 
			(1,2) node[circle, inner sep=0.04cm, fill=black, draw] {};
			\draw
			(0,1) --
			(-1,3) node[circle, inner sep=0.04cm, fill=black, draw] {} -- 
			(-1,3.5) node[left] {$2$} --
			(-1,4) node[circle, inner sep=0.04cm, fill=black, draw] {} --
			(1,2);
			\draw (-0.3,0) node{$e$};
			\draw (-0.3,1) node{$d$};
			\draw (1.3,2) node{$c$};
			\draw (-1.3,3) node{$b$};
			\draw (-1.3,4) node{$a$};
			\draw (0.8,1) node{$\eta$};
			\draw (-0.8,2) node{$\eta$};
			\draw (0.3,3) node{$\eta$};
			\end{tikzpicture}
			\caption{The graph $\Gamma_2$ of a potential flow category $\cC$.}
			\label{fig:Adem_relation} 
		\end{center}
	\end{figure}
	
	Indeed there is a way to add in the missing edges of length greater than $2$ to form a framed flow category $\cC$, but this is not obvious.  The general problem is that if one has framed all moduli spaces of dimension less than $n$, then this determines the framing of the \emph{boundaries} of all $n$-dimensional moduli spaces. To  extend the framings to the interior, these boundary framings would have to be null-cobordant!  We shall return to this example later when we are more precise about $\eta$.
\end{example}

\subsection{Plan of attack}
\label{subsec:attack_plan}
We conclude the introduction by giving a user's guide to the rest of the paper, the majority of which is aimed at proving Theorem \ref{thm:main}.

In Section \ref{sec:flow_framed_categories}, we give the definitions required for working with framed flow categories, and discuss stably framed smooth manifolds in low dimensions (since these make up the moduli spaces with which we shall later be working).  We revisit Example \ref{exam:Adem_relation} in the light of this discussion.

Section \ref{sec:calculus} gives us the tools which form the heart of many of the arguments which follow.  We use our four flow category moves to build a small library of operations on flow categories, that can be understood graphically.  For an example of this in action, jump forward to Lemma \ref{lem:trick1}.  The operation which we wish to add to our library is the graphical move depicted in the statement of the lemma.  The proof that this can be decomposed as a product of flow category moves follows the sequence of diagrams at the end of the proof of the lemma.

Section \ref{sec:musical_scores} further refines the notion of a diagram depicting a framed flow category, in a way particularly well suited to those categories whose objects appear only in~4 adjacent gradings. 

Sections \ref{sec:chang} and \ref{sec:baues_hennes} contain those flow categories that we are aiming for.  More specifically, for each stable homotopy type of homological width at most $3$ we give a specific framed flow category realizing it. For a nice example of such a specific framed flow category see the diagram immediately following Definition \ref{defn:bh_flow_category}.

The idea of the paper is then, given any framed flow category of the homological widths we consider, to apply our moves and operations to it in order to arrive at one of the specific framed flow categories that we have given.  In the case of homological width~$2$ this is relatively straightforward and is achieved already in Theorem \ref{thm:chang}. We dedicate Sections \ref{sec:partitions} and \ref{sec:main_theorem} to the homological width $3$ case.

Finally, in Section \ref{sec:algorithmic}, we discuss how the ideas in the paper give rise to an algorithmic approach for computation in framed flow categories, something that the third author has implemented in a computer program \cite{SchuetzKJ}.  An example is given of the computation of the Lipshitz-Sarkar stable homotopy type of a certain knot.

\subsection*{Acknowledgements}
We thank the referees for their careful reading and helpful suggestions. We also thank Taketo Sano for alerting us to a problem in an earlier version of Lemma~\ref{lem:trick1}. The authors were partially supported by the EPSRC Grant EP/K00591X/1. PO was partially supported by a CIRGET postdoctoral fellowship.

\section{Flow categories and framed manifolds}
\label{sec:flow_framed_categories}
We recall the necessary definitions to talk about framed flow categories. We then collect some standard facts about framed bordism groups that will be used throughout the article.

\subsection{Framed flow categories} Framed flow categories first appeared in \cite{CJS} and were further developed in \cite{MR2597734} and \cite{LipSarKhov}. We will work with the description given by Lipshitz and Sarkar in \cite{LipSarKhov}.

A smooth manifold with corners is defined in the same way as an ordinary smooth manifold, except that the differentiable structure is now modelled on the open subsets of the $k$-fold product $[0,\infty)^k$. To describe framed flow categories, first we need a restricted class of manifolds with corners originally defined by J\"{a}nich \cite{ich}, and further developed by Laures \cite{Laures}. 

If $X$ is a smooth manifold with corners and $x\in X$ is represented by $(x_1,\ldots,x_k)\in [0,\infty)^k$, let $c(x)$ be the number of coordinates in this $k$-tuple which are $0$. Denote by
\begin{eqnarray*}
	\partial^iX&=&\{x\in X\,|\,c(x)=i\}
\end{eqnarray*}
the codimension-$i$-boundary. Note that $x$ belongs to at most $c(x)$ different connected components of $\partial^1 X$. We call $X$ a \em smooth manifold with faces \em if every $x\in X$ is contained in the closure of exactly $c(x)$ components of $\partial^1X$. A \em connected face \em is the closure of a component of $\partial^1 X$, and a \em face \em is any union of pairwise disjoint connected faces (including the empty face). Note that every face is itself a manifold with faces. We define the boundary of $X$, $\partial X$, as the closure of $\partial^1 X$.

\begin{definition}
	Let $k,n$ be non-negative integers and $X$ a smooth manifold with faces. An \em $n$-face structure \em for $X$ is an ordered $n$-tuple $(\partial_1 X,\ldots, \partial_n X)$ of faces of $X$ such that
	\begin{enumerate}
		\item $\partial_1 X\cup \cdots \cup \partial_nX = \partial X$.
		\item $\partial_i X\cap \partial_j X$ is a face of both $\partial_i X$ and $\partial_j X$ for $i\not=j$.
	\end{enumerate}
	A smooth manifold with faces $X$ together with an $n$-face structure is called a \em smooth $\langle n \rangle$-manifold\em.
	
	If $a=(a_1,\ldots,a_n)\in \{0,1\}^n$, we define
	\begin{eqnarray*}
		X(a)&=&\bigcap_{i\in \{j\,|\,a_j=0\}} \partial_i X
	\end{eqnarray*}
	and note that this is an $\langle |a| \rangle$-manifold, where $|a|=a_1+\cdots+ a_n$. If $a=(1,\ldots,1)$ we interpret the empty intersection as $X$.
\end{definition}

There is an obvious partial order $\leq$ on $\{0,1\}^n$ such that $X(a)\subset X(b)$ for $a\leq b$.

\begin{definition} \label{euclidcorners}
	Given an $(n+1)$-tuple $\mathbf{d}=(d_0,\ldots,d_n)$ of non-negative integers, let
	\begin{eqnarray*}
		\E^\mathbf{d}&=&\R^{d_0}\times [0,\infty) \times \R^{d_1}\times [0,\infty) \times \cdots \times [0,\infty) \times \R^{d_n}.
	\end{eqnarray*}
	Furthermore, if $0\leq a < b \leq n+1$, we denote $\E^\mathbf{d}[a:b]=\E^{(d_a,\ldots,d_{b-1})}$.
\end{definition}

We can turn $\E^\mathbf{d}$ into an $\langle n\rangle$-manifold by setting
\begin{eqnarray*}
	\partial_i \E^\mathbf{d}&=& \R^{d_0}\times \cdots \times \R^{d_{i-1}} \times \{0\} \times \R^{d_i} \times \cdots \times \R^{d_n}.
\end{eqnarray*}
We will refer to this boundary part as the \em $i$-boundary\em. In the case of $\E^\mathbf{d}[a:b]$ we also refer to the set
\begin{eqnarray*}
	\partial_{i-a} \E^\mathbf{d}[a:b]&=& \R^{d_a}\times \cdots \times \R^{d_{i-1}} \times \{0\} \times \R^{d_i} \times \cdots \times \R^{d_{b-1}}
\end{eqnarray*}
as the $i$-boundary, although strictly speaking this should be the $(i-a)$-boundary.

\begin{definition}
	A \em neat immersion $\imath$ \em of an $\langle n \rangle$-manifold is a smooth immersion $\imath\colon X \looparrowright \E^\mathbf{d}$ for some $\mathbf{d}\in \Z^{n+1}$ such that
	\begin{enumerate}
		\item For all $i$ we have $\imath^{-1}(\partial_i\E^\mathbf{d})=\partial_i X$.
		\item The intersection of $X(a)$ and $\E^\mathbf{d}(b)$ is perpendicular for all $b<a$ in $\{0,1\}^n$.
	\end{enumerate}
	A \em neat embedding \em is a neat immersion that is also an embedding.
	
	Given a neat immersion $\imath\colon X \looparrowright \E^\mathbf{d}$ we have a normal bundle $\nu_{\imath(a)}$ for each immersion $\imath(a)\colon X(a) \looparrowright \E^\mathbf{d}(a)$ as the orthogonal complement of the tangent bundle of $X(a)$ in $T\E^\mathbf{d}(a)$.
\end{definition}

\begin{definition}
	A \em flow category \em is a pair $(\cC,\gr{\cdot})$ where $\cC$ is a category with finitely many objects $\Ob=\Ob(\cC)$ and $\gr{\cdot}\colon \Ob \to \Z$ is a function, called the \em grading\em, satisfying the following:
	\begin{enumerate}
		\item $\Hom(x,x)=\{\id\}$ for all $x\in \Ob$, and for $x\not=y \in \Ob$, $\Hom(x,y)$ is a smooth, compact $(\gr{x}-\gr{y}-1)$-dimensional $\langle \gr{x}-\gr{y}-1\rangle$-manifold which we denote by $\mathcal{M}(x,y)$.
		\item For $x,y,z\in \Ob$ with $\gr{z}-\gr{y}=m$, the composition map
		$$\circ\colon \mathcal{M}(z,y) \times \mathcal{M}(x,z) \to \mathcal{M}(x,y)$$
		is an embedding into $\partial_m\mathcal{M}(x,y)$. Furthermore,
		\begin{eqnarray*}
			\circ^{-1}(\partial_i \mathcal{M}(x,y))&=&\left\{ \begin{array}{lr}
				\partial_i \mathcal{M}(z,y)\times \mathcal{M}(x,z) & \mbox{for }i<m \\
				\mathcal{M}(z,y)\times \partial_{i-m}\mathcal{M}(x,z) & \mbox{for }i>m
			\end{array}
			\right.
		\end{eqnarray*}
		\item For $x\not= y\in \Ob$, $\circ$ induces a diffeomorphism
		\begin{eqnarray*}
			\partial_i\mathcal{M}(x,y)&\cong & \coprod_{z,\,\gr{z}=\gr{y}+i} \mathcal{M}(z,y) \times \mathcal{M}(x,z).
		\end{eqnarray*}
		
	\end{enumerate}
	We also write $\mathcal{M}_\cC(x,y)$ if we want to emphasize the flow category. The manifold $\mathcal{M}(x,y)$ is called the \em moduli space from $x$ to $y$\em, and we also set $\mathcal{M}(x,x)=\emptyset$.
\end{definition}

\begin{definition} \label{def:neat_immersion}
	Let $\cC$ be a flow category and $\mathbf{d}=(d_A,\ldots,d_{B-1})\in \Z^{B-A}$ a sequence of non-negative integers with $A\leq \gr{x} \leq B$ for all $x\in \Ob(\cC)$. A \em neat immersion $\imath$ \em of the flow category $\cC$ relative $\mathbf{d}$ is a collection of neat immersions $\imath_{x,y}\colon \mathcal{M}(x,y)\looparrowright \E^\mathbf{d}[\gr{y}:\gr{x}]$ for all objects $x,y$ such that for all objects $x,y,z$ and all points $(p,q)\in \mathcal{M}(z,y)\times \mathcal{M}(x,z)$ we have
	\begin{eqnarray*}
		\imath_{x,y}(p\circ q)&=&(\imath_{z,y}(p),0,\imath_{x,z}(q)).
	\end{eqnarray*}
	The neat immersion $\imath$ is called a \em neat embedding\em, if for all $i,j$ with $A\leq j<i\leq B$ the induced map
	$$
	\imath_{i,j}\colon \coprod_{(x,y),\gr{x}=i,\gr{y}=j}\mathcal{M}(x,y) \to \E^\mathbf{d}[j:i]
	$$
	is an embedding.
\end{definition}

\begin{definition}
	Let $\imath$ be a neat immersion of a flow category $\cC$ relative $\mathbf{d}$. A \em coherent framing $\varphi$ \em of $\imath$ is a framing for the normal bundle $\nu_{\imath_{x,y}}$ for all objects $x,y$, such that the product framing of $\nu_{\imath_{z,y}}\times \nu_{\imath_{x,z}}$ equals the pullback framing of $\circ^\ast \nu_{\imath_{x,y}}$ for all objects $x,y,z$.
	
	A \em framed flow category \em is a triple $(\cC,\imath,\varphi)$, where $\cC$ is a flow category, $\imath$ a neat immersion and $\varphi$ a coherent framing of $\imath$. We will omit the neat immersion and the coherent framing from the notation if it does not cause confusion.
\end{definition}

Given a framed flow category $\cC$, the \emph{Cohen--Jones--Segal construction} can be used to build a stable homotopy type $\mathcal{X}(\cC)$. This is some precise number of desuspensions of a finite CW complex with cells corresponding to the objects of $\cC$ and glueing maps determined in some precise way by the framed embedded moduli spaces. We will not repeat this construction here, but refer the reader to \cite[Definition 3.23]{LipSarKhov}  for the definition and basic properties.

\subsection{Framed manifolds}
\label{sec:framed_manifolds}

The abelian group $\Omega^{fr}_n$ of framed cobordism classes of closed, framed, smooth $n$-manifolds coincides via Pontryagin-Thom construction with the $n$-th stable homotopy group $\pi_n^{st}$. For some small values of $n$ we need to understand the generators of these groups. Clearly, $\Omega^{fr}_0\cong \Z$ is generated by a framed point.

For $n=1$ we get that $\Omega_1^{fr}\cong \Z/2\Z$ is generated by a non-trivially framed circle. This non-trivial framing can be obtained by pulling back a framed point in $S^2$ via the Hopf map $h\colon S^3\to S^2$ and stabilizing. We shall call this non-trivially framed circle $\eta$, in line with the standard notation for the non-trivial element $\eta\in \pi_1^{st}$.

For $n=2$ we get $\Omega_2^{fr}\cong \Z/2\Z$ is generated by the torus $\eta^2$ which we also call $\varepsilon$, to avoid confusion later on with words used for Baues-Hennes spaces. A framed surface $S$ is orientable, so if we have an embedded circle $C$ in $S$, its normal bundle in $S$ is trivial and can therefore be used to get a framing of $C$. Circles determining a symplectic basis for $H_1(S;\Z/2\Z)$, together with these induced framings, can be used to calculate the Arf invariant of $S$, and in particular whether the surface as a whole is framed trivially or not.

We will also need a relative version. Assume we have two non-trivially framed circles embedded in some $\R^m$ with $m\geq 4$. Since together they represent the trivial element in $\Omega^{fr}_1\cong\Z/2\Z$, we can find a framed cobordism $W$ which bounds them in $\R^{m}\times [0,\infty)$. Using a relative version of the Pontryagin-Thom construction, we see that between two non-trivially framed circles there are exactly two framed cobordisms $W$, up to framed cobordisms which are fixed near the boundary. As we can do framed surgeries along trivially framed circles in $W$, we see that both of these framed cobordisms can be represented by a cylinder. Furthermore, adding the non-trivially framed torus $\varepsilon$ to one gives the other.

\begin{remark}\label{rem:detect_triv_circle}
Given a framed circle $C$ in some $\R^m$ and a point on $C$, there is exactly one tangential direction which together with the framing gives an element of $\SO(m)$. Varying the point in $C$ gives an element $H_1(\SO(m))\cong \Z/2\Z$ with the property that $C$ is trivially framed if and only if this element is non-trivial in $H_1(\SO(m))$, compare \cite[\textsection 3.2]{LipSarSq}.
\end{remark}

\begin{remark}\label{rem:extend_framing}
Given a framed flow category $\cC$ and two objects $a,b\in \Ob(\cC)$, the boundary of $\M(a,b)$ is embedded into some cornered Euclidean space $\R^m\times \partial ([0,\infty)^k)$, where $k = \gr{a}-\gr{b}-1$ and $\partial ([0,\infty)^k)$ consists of those points in $[0,\infty)^k$ which do not have a neighborhood homeomorphic to $\R^k$. Because of the transversality conditions associated to the framed embedding, projection gives an immersion $\partial\M(a,b)\to \partial ([0,\infty)^k)$. Since we can smoothen $\partial ([0,\infty)^k)$ to $\R^{k-1}$, we can also smoothen $\partial\M(a,b)$ to a closed $(k-1)$-dimensional manifold.
Also, since the framing is orthogonal to $\partial ([0,\infty)^k)$, $\partial\M(a,b)$ is in fact a framed manifold, with the framing of $\M(a,b)$ being a framed null-cobordism.

Given a flow category, one may try to frame it by induction on the dimension of the moduli spaces based on whether the framing of $\partial \M(a,b)$ is trivial.
\end{remark}

Let us revisit Example \ref{exam:Adem_relation} and consider it more deeply.

\begin{example}\label{exam:Adem_relation2}
Consider again the graph in Figure \ref{fig:Adem_relation}.  We shall now see that this is the graph $\Gamma_2$ of a framed flow category $\cC$ as follows.

The moduli spaces $\M(a,d)$ and $\M(b,e)$ have as boundary two copies of $\eta$ each (where now by $\eta$ we understand the non-trivially framed circle).  We can therefore choose both of these moduli spaces to be a framed cylinder. It is easy to see that $\M(d,e)\times \M(a,d)\cup \M(b,e)\times \M(a,b)$ consists of four cylinders which form a single torus. 

Furthermore, if we form a circle in this torus by traversing in the cylindrical direction of the moduli spaces $\M(a,d)$ and $\M(b,e)$, we see that this circle is non-trivially framed. This follows from Remark \ref{rem:detect_triv_circle}, as we go twice through each cylinder. It follows that this torus represents $\varepsilon\in \Omega_2^{fr}$.

Note that $\M(c,e)\times \M(a,c)$ gives another non-trivially framed torus, and therefore we can extend $\partial\M(a,e)$ to a framed cobordism between these two tori, leading to a framed flow category.

It is worth pointing out that $\M(d,e)\times \M(a,d)\cup \M(b,e)\times \M(a,b)$ and $\M(c,e)\times \M(a,c)$ are embedded quite differently into the cornered Euclidean space, so it is not clear whether there should be a natural null-cobordism. In fact, by adding a generator of $\Omega_3^{fr}$ via disjoint union to $\M(a,e)$, we can change the value of $\Sq^4$ in this flow category. So the stable homotopy type of $\cC$ is not uniquely determined by $\Gamma_2$ alone.
\end{example}

\section{A calculus of flow category moves}
\label{sec:calculus}

We now describe in detail the equivalence relation on framed flow categories, which we call \emph{move equivalence}. Move equivalent framed flow categories have the same stable homotopy type. This equivalence relation is the transitive closure of several moves that have been defined in various papers. For the convenience of the reader we now recall these moves.

\subsection{Stabilization and Perturbation}\label{sec:stabpert}

Let $(\cC,\imath,\varphi)$ be a framed flow category. Recall that $\imath$ is a collection of immersions $\imath_{x,y}\colon \M(x,y) \looparrowright \E^\mathbf{d}[\gr{y}:\gr{x}]$ into cornered Euclidean space, and $\varphi$ is a collections of framings of these immersions. 

A \emph{stabilization} of $(\cC,\imath,\varphi)$ is a framed flow category $(\cC,\imath',\varphi')$ where $\imath'_{x,y}\colon \M(x,y) \looparrowright \E^\mathbf{d'}[\gr{y}:\gr{x}]$ immerses into a higher-dimensional Euclidean space with $\mathbf{d'} = (d_A+e_A,\ldots,d_{B-1}+e_{B-1})\in \Z^{B-A}$ with all $e_i\geq 0$, and projection to the new dimensions is constant $0$. The framing $\varphi'$ agrees with $\varphi$ on the old dimensions, and is the standard framing in the new dimensions.

A $1$-parameter family of framings of a flow category $\cC$ is a collection $(\imath(t),\varphi(t))$ for every $t\in [0,1]$, such that the $\imath(t)$ are neat immersions into some $\E^\mathbf{d}$, smoothly varying in $t$, and the $\varphi(t)$ are smoothly varying coherent framings of the $\imath(t)$. We say that the framed flow category $(\cC,\imath',\varphi')$ is a \emph{perturbation} of $(\cC,\imath,\varphi)$, if there is a $1$-parameter family of framings between stabilizations of $(\imath',\varphi')$ and $(\imath,\varphi)$.  

The stable homotopy type associated to a framed flow category $(\cC,\imath,\phi)$ is denoted by~$\X(\cC)$. Its definition uses the fact that every framing $(\imath,\varphi)$ can be perturbed to a framing $(\imath',\varphi')$ with $\imath'$ a neat embedding, and that the stable homotopy type does not depend on the perturbation, see \cite[Lms.\ 3.25, 3.26]{LipSarKhov}. 

\subsection{Handle Cancellation}\label{sec:handlecancellation}

Let $\cC$ be a framed flow category and assume we have two objects $x,y\in \Ob(\cC)$ such that $|x|=|y|+1$ and $\M(x,y)$ consists of a single point.

\begin{definition} \label{cancelledcat}
	Denote by $\cC_H$ the flow category whose object set is given by
	\[ 
	\Ob(\cC_H) =  \Ob(\cC) \setminus \{ x,y \} 
	\]
	and whose moduli spaces are given by
	\[
	\mathcal{M}_{\cC_H}(a,b) = \mathcal{M}_\cC(a,b)\cup_f \big( \mathcal{M}_\cC(x,b)\times \mathcal{M}_\cC(a,y) \big)
	\]
	where $f$ identifies the subsets 
	\[
	\mathcal{M}_\cC(x,b)\times \mathcal{M}_\cC(a,x)\cup \mathcal{M}_\cC(y,b)\times \mathcal{M}_\cC(a,y) \subset \mathcal{M}_\cC(a,b)
	\]
	and
	\begin{multline*}
	\mathcal{M}_\cC(x,b)\times (\mathcal{M}_\cC(x,y)\times \mathcal{M}_\cC(a,x)) \cup (\mathcal{M}_\cC(y,b)\times \mathcal{M}_\cC(x,y))\times \mathcal{M}_\cC(a,y) \\ \subset \mathcal{M}_\cC(x,b)\times \mathcal{M}_\cC(a,y).
	\end{multline*}
	We call $\cC_H$ the \emph{cancelled category} (relative to $x$ and $y$) of $\cC$.
\end{definition}

We refer the reader to \cite[\textsection 2]{JLS2} for the embedding and framing of $\cC_H$, so that $\X(\cC_H)\simeq \X(\cC)$.

\subsection{Extended Whitney Trick}\label{sec:extendedwhitney}

In \cite[\S 2.2]{JLS2} a Whitney trick was introduced, which removes two oppositely framed points in a $0$-dimensional moduli space. This was generalized in \cite[\S 4]{ALPODS2} to changing moduli spaces up to framed cobordism relative to the boundary. We will now describe this extended Whitney trick, but refer the reader to \cite[\textsection 4]{ALPODS2} for the definition of the framings.

\begin{definition}
Let $\cC$ be a framed flow category and assume we have objects $x,y\in \Ob(\cC)$ such that $M = \M(x,y)$ is framed cobordant relative to the boundary to a framed manifold $M'$ via a framed cobordism $W$.

Then there is a framed flow category $\cC_W$ with $\Ob(\cC_W)= \Ob(\cC)$ and the moduli spaces are given as follows.
\begin{enumerate}
	\item $\M_{\cC_W}(x,y) = M'$.
	\item For $a\in \Ob(\cC)$ with $\M_\cC(a,x)\not=\emptyset$ we get
	\[
		\M_{\cC_W}(a,y) = W \times \M_\cC(a,x) \cup_g \M(a,y)
	\]
	where $g$ identifies the two copies of $M\times \M_\cC(a,x)$ in $W\times \M_\cC(a,x)$ and $\partial \M(a,y)$.
	\item For $b\in \Ob(\cC)$ with $\M_\cC(y,b)\not=\emptyset$ we get
	\[
		\M_{\cC_W}(x,b) = \M_\cC(y,b) \times W \cup_g \M_\cC(x,b)
	\]
	where $g$ identifies the two copies of $\M_\cC(y,b)\times M$ in $\M_\cC(y,b) \times W$ and $\partial \M_\cC(x,b)$.
	\item For $a,b\in \Ob(\cC)$ with both $\M_\cC(a,x)\not=\emptyset \not= \M_\cC(y,b)$ the moduli space $\M_{\cC_W}(a,b)$ is obtained by using an appropriate cobordism between
	\[
		\M_{\cC_W}(x,b) \times \M_\cC(a,x) \cup \M_\cC(y,b)\times \M_{\cC_W}(a,y),
	\]
	see \cite[Def.4.5]{ALPODS2} for details.
	\item In all other cases $\M_{\cC_W}(a,b)=\M_\cC(a,b)$.
\end{enumerate}
\end{definition}

We refer the reader to \cite[\S 4]{ALPODS2} for the framings of $\cC_W$ and the fact that $\X(\cC_W)\simeq \X(\cC)$.

\subsection{Handle Sliding}\label{sec:handleslide}

The handle slide was introduced in \cite[\S 3]{ALPODS2} as a consequence of two handle cancellations. It is however a very useful move, and in practice one of the most commonly used moves. We therefore list it here as well.

\begin{definition}
	Let $(\cC,\imath,\varphi)$ be a framed flow category and $x\not=y$ be objects with $|x|=|y|$. Then the \emph{$(\pm)$-handle slide of $x$ over $y$} is the framed flow category $(\cC_S,\imath_S,\varphi_S(\pm))$ defined as follows. We have $\Ob(\cC_S) = \Ob(\cC)$ and the moduli spaces are given by
	\begin{align*}
	\M_{\cC_S}(x,b)&= \M_\cC(x,b)\sqcup\M_\cC(y,b),\\
	\M_{\cC_S}(a,y)&= \M_\cC(a,x)\sqcup\M_\cC(a,y),
	\end{align*}
	In all other cases 
	\[
	\M_{\cC_S}(a,b) = \M_\cC(a,b) \sqcup \left(\M_\cC(y,b) \times [0,1] \times \M_\cC(a,x)\right).
	\]
	In the case of a $(+)$-handle slide the new copy of $\M_\cC(y,b)$ in $\M_{\cC_S}(x,b)$ keeps its framing, while the new copy of $\M_\cC(a,x)$ in $\M_{\cC_S}(a,y)$ gets the opposite framing. In the case of a $(-)$-handle slide these roles are reversed.
\end{definition}

See \cite[\S3]{ALPODS2} for details on the effects of the framings, and that $\X(\cC)\simeq \X(\cC_S)$. In terms of the graph notation we can interpret a handle slide as follows.
\begin{center}
\begin{tikzpicture}
\node at (0.3,0) {$\cC\colon$};
\dotd{1};
\dotd{3};
\draw (1,1) -- (1,-1);
\draw (3,1) -- (3,-1);
\node at (0.8,0.5) {$\alpha$};
\node at (0.8,-0.5) {$\gamma$};
\node at (3.2,0.5) {$\beta$};
\node at (3.2,-0.5) {$\delta$};
\node at (4.3,0) {$\cC_S\colon$};
\node at (1.3,0) {$x$};
\node at (2.7,0) {$y$};
\dotd{5};
\dotd{7};
\draw (5,1) -- (5,-1);
\draw (7,1) -- (7,-1);
\draw (5.1,1) -- (7,0);
\draw (5,0) -- (6.9,-1);
\node at (5.3,0) {$x$};
\node at (6.7,0) {$y$};
\node at (4.8,0.5) {$\alpha$};
\node at (4.8,-0.5) {$\gamma$};
\node at (7.2,0.5) {$\beta$};
\node at (7.2,-0.5) {$\delta$};
\node at (6.3,0.6) {$\mp \alpha$};
\node at (5.5,-0.6) {$\pm \delta$};
\end{tikzpicture}
\end{center}

\subsection{Move equivalence and basic consequences}

\begin{definition}
	Two framed flow categories $(\cC,\imath,\varphi)$ and $(\cC',\imath',\varphi')$ are called \emph{move equivalent}, if there exists a finite sequence of framed flow categories $(\cC_1,\imath_1,\varphi_1),\ldots,$ $(\cC_k,\imath_k,\varphi_k)$ with $(\cC_1,\imath_1,\varphi_1) = (\cC,\imath,\varphi)$, $(\cC_k,\imath_k,\varphi_k) = (\cC',\imath',\varphi')$, and for $i=1,\ldots,k-1$ we have that $(\cC_i,\imath_i,\varphi_i)$ and $(\cC_{i+1},\imath_{i+1},\varphi_{i+1})$ are related by either a perturbation, a handle cancellation, or an extended Whitney trick.
	
	We write $(\cC,\imath,\varphi)\sim (\cC',\imath',\varphi')$ or simply $\cC \sim \cC'$ for move equivalent framed flow categories.
\end{definition} 

\begin{definition}
	A framed flow category $\cC$ is of \emph{homological width $r$} if the reduced homology $\widetilde{H}_*(\cC;\Z)$ and cohomology $\widetilde{H}^\ast(\cC;\Z)$ are only supported in degrees $n$, $n+1$,\dots, $n+r$, for some $n\in\Z$. We call the framed flow category \emph{trivial}, if its reduced integral homology vanishes.

	The framed flow category is of \emph{width $r$} if for all objects $x\in \Ob(\cC)$ we have $\gr{x}\in \{n,n+1,\ldots,n+r\}$ for some $n\in \Z$.
\end{definition}

Recall that given a finitely generated free chain complex over $\Z$ there is a basis inducing direct sum decompositions $C_r\cong U_r\oplus V_r$, such that the differentials have matrices of the form\[d_r=\left(\begin{matrix}0&0\\D&0\end{matrix}\right)\colon U_r\oplus V_r\to U_{r-1}\oplus V_{r-1},\]where $D$ is injective and diagonal with $D_{ii}|D_{i+1\,i+1}$. If we are moreover allowed to add or remove cancelling generators in adjacent homological degrees these matrices can be changed so that the diagonal entries of $D$ are, up to sign, all prime powers different from $\pm1$. Call such a basis the \emph{primary Smith normal form} for the chain complex.

Observe that in this form the basis elements of $V_r$ are a basis for $H_r(C)$ and the images of the basis elements of $U_{r+1}$ under $d_{r+1}$ are exactly the relators presenting the primary decomposition of the torsion part of the abelian group $H_r(C)$.

A framed flow category $\cC$ determines a based chain complex $(C_*,d)$. The basis of $C_r$ is given by the objects $a$ of $\cC$ with $|a|=r$. If $|a|=r$ and $|b|=r-1$ then the $(a,b)$ entry in the matrix of $d_{r-1}$ is given by the signed count of the points in $\M(a,b)$.

\begin{definition}
	A framed flow category is in \emph{primary Smith normal form} if its based chain complex $(C_*,d)$ is in primary Smith normal form and the number of points in any 0-dimensional moduli space is exactly the corresponding entry in the matrix of the differential $d$.
\end{definition}

Notice that the graph $\Gamma_1$ of a framed flow category in primary Smith normal form is a disjoint union of vertices, and edges of the form 
\begin{tikzpicture} 
\draw
    (10, 1.2) node[circle, inner sep=0.04cm, fill=black, draw] {} -- 
    (10, 1.5) node[left] {$p^k$} -- 
    (10, 1.8) node[circle, inner sep=0.04cm, fill=black, draw] {};
\end{tikzpicture}
.

In a previous paper, we showed the following.

\begin{theorem}[{\cite[Theorem 6.2]{ALPODS2}}]\label{thm:smith} 
	Any framed flow category $\cC$ is move equivalent to a framed flow category in primary Smith normal form.
\end{theorem}

\begin{definition}
Let $k\geq 2$ and $n$ be integers. Then let $\Mo(\Z/k\Z,n)$ be the framed flow category with two objects $a,b$ satisfying $|a|=|b|+1 = n+1$ such that $\M(a,b)$ consists of $k$ positively framed points. We call $\Mo(\Z/k\Z,n)$ the \emph{Moore flow category for $(\Z/k\Z,n)$}, and we also write $\Mo(\Z/k\Z)$ if we do not want to specify $n$.

Similarly, define $\Sp^n$ to be the framed flow category with one object $a$ satisfying $|a|=n$, called the \emph{sphere flow category}.
\end{definition}

Given two framed flow categories $\cC_1$ and $\cC_2$ we can form a new framed flow category, their \emph{disjoint union}, denoted $\cC=\cC_1\sqcup \cC_2$. This flow category contains both $\cC_1$ and $\cC_2$ as full subcategories, and any moduli spaces $\M(a,b)$ with $a\in \Ob(\cC_i)$ and $b\in \Ob(\cC_j)$, $\{i,j\}=\{1,2\}$, are empty.

It follows immediately from the Cohen--Jones--Segal construction that 
\[
\mathcal{X}(\cC)=\mathcal{X}(\cC_1)\vee \mathcal{X}(\cC_2).
 \]

A framed flow category is called \emph{indecomposable} if it is not move equivalent to a disjoint union of two non-trivial framed flow categories. We call the sphere flow categories $\Sp^n$ and the Moore flow categories $\Mo(\Z/p^k\Z,n)$, where $p$ is a prime number and $k\geq 1$, the \emph{elementary Moore flow categories}.  

\begin{lemma}\label{lem:trick1}
Let $\cC$ be a framed flow category in primary Smith normal form, and let $b,c,d\in \Ob(\cC)$ be objects such that $|c|=|d|+1$ and $|b|=|c|+m$ for some $m>0$. Assume that $\M_\cC(c,d)$ consists of $p>0$ positively framed points, and $\tau = \M_\cC(b,d)$ is a closed framed manifold such that the order of $\tau\in \Omega_m^{fr}$ is $q>0$ with $\gcd(p,q)=1$. Then $\cC$ is move equivalent to a framed flow category $\cC'$ which has the same objects as $\cC$, and such that
\[
 \M_{\cC'}(b,d) = \emptyset
\]
and
\[
 \M_{\cC'}(x,y) = \M_{\cC}(x,y)
\]
for all objects $x,y\in \Ob(\cC)$ with $|x|-|y| \leq m$. Furthermore, if $|x|-|y| = m+1$, we have
\[
 \M_{\cC'}(x,y) = \M_{\cC}(x,y)
\]
unless $(x,y)=(b,d)$ or $y=c$ and $\M(x,b)\not=\emptyset$.

There is an obvious dual statement where $|b|=|d|-m$ and $\tau = \M_\cC(c,b)$.

\begin{center}
\begin{tikzpicture}

\draw
    (0, 0.6) node[circle, inner sep=0.04cm, fill=black, draw] {} -- 
    (0, 0.3) node[left] {p} -- 
    (0, 0) node[circle, inner sep=0.04cm, fill=black, draw] {} -- 
    (0.8,1.8) node[circle, inner sep=0.04cm, fill=black, draw] {} ;
    
\draw [->] [decorate, decoration={snake}] (1.2,0.9) -- (2.4,0.9);
    
\draw
    (3, 0.6) node[circle, inner sep=0.04cm, fill=black, draw] {} -- 
    (3, 0.3) node[left] {p} -- 
    (3, 0) node[circle, inner sep=0.04cm, fill=black, draw] {};
\draw
    (3.8,1.8) node[circle, inner sep=0.04cm, fill=black, draw] {} ;

\node at (0.7,0.9) {$\tau$};

\node at (5.2,0.9) {and};

\draw
    (7, 1.2) node[circle, inner sep=0.04cm, fill=black, draw] {} -- 
    (7, 1.5) node[left] {p} -- 
    (7, 1.8) node[circle, inner sep=0.04cm, fill=black, draw] {} -- 
    (7.8,0) node[circle, inner sep=0.04cm, fill=black, draw] {} ;
    
\draw [->] [decorate, decoration={snake}] (8.2,0.9) -- (9.4,0.9);
    
\draw
    (10, 1.2) node[circle, inner sep=0.04cm, fill=black, draw] {} -- 
    (10, 1.5) node[left] {p} -- 
    (10, 1.8) node[circle, inner sep=0.04cm, fill=black, draw] {};
\draw
    (10.8,0) node[circle, inner sep=0.04cm, fill=black, draw] {} ;

\node at (7.7,0.9) {$\tau$};

\end{tikzpicture}
\end{center}
\end{lemma}

\begin{proof}
First observe that $\M(b,c)=\emptyset$ for otherwise $\M(b,d)$ would not be a closed manifold.

By Bezout's Lemma there exist integers $r,s$ so that $rp+sq=1$. Hence $rp-1$ copies of $\tau$ are frame nullcobordant. We can now define a framed flow category $\cC''$ containing $\cC$ as a full subcategory as follows. Starting with $\cC$, we add two objects $a',b'$ with $|b'|=|b|=|a'|-1$, setting $\M_{\cC''}(a',b')$ a negatively framed point, $\M_{\cC''}(a',c)$ consists of $r$ copies of $\tau$.

For any object $x\in \Ob(\cC)$ with $|x|\leq |d|$ we set 
\[
 \M_{\cC''}(b',x) = \M(c,x) \times r\tau
\]
so that $\partial \M_{\cC''}(a',x)$ consists of two copies of $\M(c,x) \times r\tau$ with opposite framings. We can therefore obtain a framed moduli space $\M_{\cC''}(a',x)$.

Notice that $\M_{\cC''}(b',d) = rp \cdot \tau$ which is framed cobordant to $\tau$. After performing the extended Whitney trick we get $\M_{\cC''}(b',d) = \tau$. As handle cancellation of $a',b'$ leads to $\cC$, these categories are move equivalent. We now perform a handle slide, the extended Whitney trick and cancellation of $a'$ and $b'$ on $\cC''$ as in the moves below to get the required flow category $\cC'$.

\begin{center}
\begin{tikzpicture}

\draw
    (0, 0.6) node[circle, inner sep=0.04cm, fill=black, draw] {} -- 
    (0, 0.3) node[left] {$p$} -- 
    (0, 0) node[circle, inner sep=0.04cm, fill=black, draw] {} -- 
    (0.8,1.8) node[circle, inner sep=0.04cm, fill=black, draw] {} ;
    
    \node at (0.7,0.9) {$\tau$};
    
\draw [->] [decorate, decoration={snake}] (1.2,0.9) -- (2.4,0.9);
    
\draw
    (4.4, 0.6) node[circle, inner sep=0.04cm, fill=black, draw] {} -- 
    (4.4, 0) node[circle, inner sep=0.04cm, fill=black, draw] {} -- 
    (5.2,1.8) node[circle, inner sep=0.04cm, fill=black, draw] {} ;
    
\draw
    (4.4, 0.6) node[circle, inner sep=0.04cm, fill=black, draw] {} -- 
    (2.6,2.4) node[circle, inner sep=0.04cm, fill=black, draw] {} --
    (2.6, 2.1) node[left] {-1} -- 
    (2.6,1.8) node[circle, inner sep=0.04cm, fill=black, draw] {} --
    (4.4, 0) node[circle, inner sep=0.04cm, fill=black, draw] {} ;

\node at (3.1,0.9) {$\tau$};
\node at (3.8,1.6) {$r\tau$};
\node at (4.3,0.3) {$p$};

\draw [->] [decorate, decoration={snake}] (5.5,0.9) -- (6.7,0.9);

\draw
    (8.6, 0.6) node[circle, inner sep=0.04cm, fill=black, draw] {} -- 
    (8.6, 0) node[circle, inner sep=0.04cm, fill=black, draw] {} ;
\draw
    (9.4,1.8) node[circle, inner sep=0.04cm, fill=black, draw] {} ;

\draw
    (8.6, 0.6) node[circle, inner sep=0.04cm, fill=black, draw] {} -- 
    (6.8,2.4) node[circle, inner sep=0.04cm, fill=black, draw] {} --
    (6.8, 2.1) node[left] {-1} -- 
    (6.8,1.8) node[circle, inner sep=0.04cm, fill=black, draw] {} --
    (8.6, 0) node[circle, inner sep=0.04cm, fill=black, draw] {} ;

\node at (7.4,0.9) {$\tau$};
\node at (8,1.6) {$r\tau$};
\node at (8.5,0.3) {$p$};

\draw [->] [decorate, decoration={snake}] (9.8,0.9) -- (11,0.9);

\draw
    (11.3, 0.6) node[circle, inner sep=0.04cm, fill=black, draw] {} -- 
    (11.3, 0.3) node[left] {$p$} -- 
    (11.3, 0) node[circle, inner sep=0.04cm, fill=black, draw] {} ;
    \draw
    (12.1,1.8) node[circle, inner sep=0.04cm, fill=black, draw] {} ;

\end{tikzpicture}
\end{center}
The dual move equivalence follows similarly and is left to the reader.
\end{proof}

Some of the higher dimensional moduli spaces can indeed change. To see this, consider the following example.

\begin{example}\label{exam:get_epsilon}
Let $\cC$ be the framed flow category which has objects $a,b,c,d$ with $|a|=|b|+1=|c|+2=|d|+3 = 3$. Let $\M(b,c)$ consist of a prime number $p>2$ of positively framed points (in fact, any odd number will work), and let $\M(a,c)$ and $\M(b,d)$ be $\eta$ with all other moduli spaces empty. As in the proof above we now get
\begin{center}
\begin{tikzpicture}
\draw (0,0) node[circle, inner sep=0.04cm, fill=black, draw] {} --
      (1,1.2) node[circle, inner sep=0.04cm, fill=black, draw] {} -- 
      (1,0.6) node[circle, inner sep=0.04cm, fill=black, draw] {} --
      (2,1.8) node[circle, inner sep=0.04cm, fill=black, draw] {};

\draw (1.13,0.9) node{$p$};
\draw (0.4,0.7) node{$\eta$};
\draw (1.4,1.3) node{$\eta$};

\draw (2.4,1.2) node{$\sim$};

\draw (3.2,0) node[circle, inner sep=0.04cm, fill=black, draw] {} --
      (4.2,1.2) node[circle, inner sep=0.04cm, fill=black, draw] {} -- 
      (4.2,0.6) node[circle, inner sep=0.04cm, fill=black, draw] {} --
      (5.2,1.8) node[circle, inner sep=0.04cm, fill=black, draw] {};

\draw (4.2,0.6) node{} --
      (3.2,1.8) node[circle, inner sep=0.04cm, fill=black, draw] {} --
      (3.2,2.4) node[circle, inner sep=0.04cm, fill=black, draw] {} --
      (4.2,1.2) node{};

\draw (3.2,1.8) to[out=-110,in=110] (3.2,0) ;

\draw (4.33,0.9) node{$p$};
\draw (3.6,0.7) node{$\eta$};
\draw (4.6,1.3) node{$\eta$};
\draw (3.5,1.2) node{$\eta$};
\draw (3.8,1.9) node{$\eta$};
\draw (2.8,0.9) node{$\varepsilon$};
\draw (3.33,2.06) node{$1$};

\draw (5.6,1.2) node{$\sim$};

\draw (6.4,0) node[circle, inner sep=0.04cm, fill=black, draw] {} --
      (7.4,1.2) node[circle, inner sep=0.04cm, fill=black, draw] {} -- 
      (7.4,0.6) node[circle, inner sep=0.04cm, fill=black, draw] {} ;

\draw (7.4,0.6) node{} --
      (6.4,1.8) node[circle, inner sep=0.04cm, fill=black, draw] {} --
      (6.4,2.4) node[circle, inner sep=0.04cm, fill=black, draw] {} --
      (7.4,1.2) node{};

\draw (6.4,1.8) to[out=-110,in=110] (6.4,0) ;
\draw (6.4,0) to[out=0, in=-90] (8.4,1.8) node[circle, inner sep=0.04cm, fill=black, draw] {};

\draw (7.53,0.9) node{$p$};
\draw (6.8,0.7) node{$\eta$};
\draw (7,1.9) node{$\eta$};
\draw (6.7,1.2) node{$\eta$};
\draw (8,0.4) node{$\varepsilon$};
\draw (6,0.9) node{$\varepsilon$};
\draw (6.53,2.06) node{$1$};

\draw (8.8,1.2) node{$\sim$};

\draw (9.4,0.6) node[circle, inner sep=0.04cm, fill=black, draw] {} --
      (9.4,1.2) node[circle, inner sep=0.04cm, fill=black, draw] {} --
      (10.4,0) node[circle, inner sep=0.04cm, fill=black, draw] {} --
      (10.4,1.8) node[circle, inner sep=0.04cm, fill=black, draw] {};

\draw (9.27,0.9) node{$p$};
\draw (10,0.7) node{$\eta$};
\draw (10.6,0.9) node{$\varepsilon$};
\end{tikzpicture}

\begin{tikzpicture}
\draw (0,1.2) node{$\sim$};

\draw (0.4,0) node[circle, inner sep=0.04cm, fill=black, draw] {} --
      (0.4,0.6) node[circle, inner sep=0.04cm, fill=black, draw] {} --
      (1.4,1.8) node[circle, inner sep=0.04cm, fill=black, draw] {} --
      (1.4,1.2) node[circle, inner sep=0.04cm, fill=black, draw] {} --
      (0.4,0);

\draw (1.4,1.8) -- 
      (2.4,0.6) node[circle, inner sep=0.04cm, fill=black, draw] {} --
      (2.4,2.4) node[circle, inner sep=0.04cm, fill=black, draw] {};

\draw (0.5,0.35) node{$1$};
\draw (1.29,1.4) node{$p$};
\draw (0.7,1.2) node{$\eta$};
\draw (1.1,0.5) node{$\eta$};
\draw (2,1.3) node{$\eta$};
\draw (2.6,1.5) node{$\varepsilon$};

\draw(3,1.2) node{$\sim$};

\draw (3.4,0) node[circle, inner sep=0.04cm, fill=black, draw] {} --
      (3.4,0.6) node[circle, inner sep=0.04cm, fill=black, draw] {} --
      (4.4,1.8) node[circle, inner sep=0.04cm, fill=black, draw] {} --
      (4.4,1.2) node[circle, inner sep=0.04cm, fill=black, draw] {} --
      (3.4,0);

\draw (5.4,0.6) node[circle, inner sep=0.04cm, fill=black, draw] {} --
      (5.4,2.4) node[circle, inner sep=0.04cm, fill=black, draw] {};

\draw (3.5,0.35) node{$1$};
\draw (4.29,1.4) node{$p$};
\draw (3.7,1.2) node{$\eta$};
\draw (4.1,0.5) node{$\eta$};
\draw (5.6,1.5) node{$\varepsilon$};

\draw (6,1.2) node{$\sim$};

\draw (6.6,1.8) node[circle, inner sep=0.04cm, fill=black, draw] {} --
      (6.6,1.2) node[circle, inner sep=0.04cm, fill=black, draw] {};

\draw (7.6,2.4) node[circle, inner sep=0.04cm, fill=black, draw] {} --
      (7.6,0.6) node[circle, inner sep=0.04cm, fill=black, draw] {};

\draw (6.8,1.5) node{$p$};
\draw (7.8,1.5) node{$\varepsilon$};

\end{tikzpicture}
\end{center}
So we see that $\cC$ is move equivalent to the disjoint union of the Moore category $\Mo(\Z/p\Z)$ and what we will later call the Baues--Hennes flow category $\B(\varepsilon)$. 
\end{example}

\begin{corollary}\label{cor:odd_Moore}
Let $\cC$ be a flow category in primary Smith normal form of homological width $r \leq 2p-3$ for some prime number $p$. Assume that $\widetilde{H}_\ast(\cC;\Z)$ contains a $\Z/q^k\Z$-summand, where $q$ is a prime number $q\geq p$ and $k\geq 1$. Then $\cC$ is move equivalent to $\cC'\sqcup \Mo(\Z/q^k\Z)$ for some framed flow category $\cC'$.
\end{corollary}

\begin{proof}
For $p=2$ this follows from Theorem \ref{thm:smith}, so assume $p\geq 3$. By a theorem of Serre \cite[IV, Prop.11]{Serre} we get that $\pi^{st}_m$ does not contain $q$-torsion for $m < 2p-3$. Hence whenever $\M(u,v)$ is a closed framed manifold of dimension bigger than $0$ in $\cC$, its order $s$ will satisfy $\gcd(s,q)=1$.

We can assume that $\cC$ is in primary Smith normal form, and so we get a full subcategory $\Mo(\Z/q^k\Z)$ of $\cC$ by the homology assumption. Denote by $a$ and $b$ the two objects in this subcategory, with $|a|=|b|+1$.

Now let $u$ be an object with minimal $|u|$ such that $\M_{\cC}(u,b) \not= \emptyset$. If $\M_{\cC}(u,a)\not=\emptyset$, then $\partial \M_{\cC}(u,b)$ consists of $q^k$ copies of $\M_{\cC}(u,a)$ by the minimality of $u$. But then $\M_{\cC}(u,a)$ is a closed framed manifold which is framed null-cobordant, as $\pi^{st}_{|u|-|a|-1}$ does not contain $q$-torsion. Using an extended Whitney trick, we can assume that $\M_{\cC}(u,a)=\emptyset$.

Then $\M_{\cC}(u,b)$ is a framed closed manifold, and by Lemma \ref{lem:trick1} we get a move equivalent framed flow category $\cC'$, with the same object set, such that $\M_{\cC'}(u,b) = \emptyset$, and such that for all objects $v\not=a$ with $|v|\leq |u|$ we have $\M_{\cC'}(v,b) = \M_{\cC}(v,b)=\emptyset$. We can repeat this argument finitely many times to get a move equivalent framed flow category $\cC''$ with $\M_{\cC''}(u,b) = \emptyset$ for all objects $u\not=a$ in $\cC''$.

The dual argument gives us a move equivalent framed flow category $\cC'''$ with $\M_{\cC'''}(a,v)=\emptyset$ for all objects $v\not=b$.
\end{proof}

\begin{lemma}\label{lem:trick2}
Let $\cC$ be a framed flow category in primary Smith normal form and $b,c\in \Ob(\cC)$ with $|b|=|c|+1$. Then there is a framed flow category $\cC'$ move equivalent to $\cC$ such that the following holds:
\begin{itemize}
	\item For all objects $x,y$ such that $\M_{\cC}(x,b)$ and $\M_{\cC}(c,y)$ are empty we have
	\[
		\M_{\cC'}(x,y) = \M_{\cC}(x,y).
	\]
	\item If $d$ is an object such that $\M_\cC(c,d)$ is a closed framed manifold, then
	\[
		\M_{\cC'}(b,d) = \M_\cC(b,d) \sqcup \M_\cC(c,d) \times \eta.
	\]
	\item If $a$ is an object such that $\M_\cC(a,b)$ is a closed framed manifold, then
	\[
		\M_{\cC'}(a,c) = \M_\cC(a,c) \sqcup \eta \times \M_\cC(a,b).
	\]
\end{itemize}

\end{lemma}

\begin{proof}
	Perform the extended Whitney trick in $\M_\cC(b,c)$ using the cobordism from $\M_\cC(b,c)$ to itself that consists of the disjoint union of the product  cobordism $\M_\cC(b,c)\times [0,1]$ and a circle (here, the circle is considered as a cobordism from the empty set to the empty set). The product cobordism is given the framing extended trivially from $\M_\cC(b,c)$ and the circle is given the non-trivial framing. This has the desired effect.
\end{proof}

\begin{example}\label{exam:remove_epsilon}
	Let $\cC$ be the framed flow category consisting of four objects $a,b,c,d$ with $|a|=|b|+1=|c|+2=|d|+3$. Assume that $\M(b,c)$ consists of $p>1$ positively framed points, $\M(a,d)=\varepsilon$, $\M(b,d)=\eta$ and all other moduli spaces empty. Using Lemma \ref{lem:trick2} on $a,b$ adds an extra $\varepsilon$ to $\M(a,d)$, and $p$ copies of $\eta$ in $\M(a,c)$. We can then use extended Whitney tricks to get $\M(a,d)$ empty and have one $\eta$ in $\M(a,c)$ (for $p$ odd) or $\M(a,c)$ empty (for $p$ even).
	\begin{center}
		\begin{tikzpicture}
			\draw (0.4,0.6) node[circle, inner sep=0.04cm, fill=black, draw] {} --
			(0.4,1.2) node[circle, inner sep=0.04cm, fill=black, draw] {} --
			(1.4,0) node[circle, inner sep=0.04cm, fill=black, draw] {} --
			(1.4,1.8) node[circle, inner sep=0.04cm, fill=black, draw] {};
			
			\draw (0.27,0.9) node{$p$};
			\draw (1,0.7) node{$\eta$};
			\draw (1.6,0.9) node{$\varepsilon$};
			
			\draw (2,0.9) node{$\sim$};
			
			\draw (2.4,0) node[circle, inner sep=0.04cm, fill=black, draw] {} --
			(3.4,1.2) node[circle, inner sep=0.04cm, fill=black, draw] {} -- 
			(3.4,0.6) node[circle, inner sep=0.04cm, fill=black, draw] {} --
			(4.4,1.8) node[circle, inner sep=0.04cm, fill=black, draw] {};
			
			\draw (3.53,0.9) node{$p$};
			\draw (2.8,0.7) node{$\eta$};
			\draw (3.8,1.3) node{$\eta$};
			
			\draw (4.6,0.1) node {$p$ odd,};
			
			\draw (5.5,0.9) node{$\sim$};
			
			\draw (6,0.6) node[circle, inner sep=0.04cm, fill=black, draw] {} --
			(6,1.2) node[circle, inner sep=0.04cm, fill=black, draw] {} --
			(7,0) node[circle, inner sep=0.04cm, fill=black, draw] {};
			
			\draw (7,1.8) node[circle, inner sep=0.04cm, fill=black, draw] {};
			
			\draw (5.87,0.9) node{$p$};
			\draw (6.7,0.7) node{$\eta$};
			
			\draw (8,0.1) node{$p$ even.};
		\end{tikzpicture}
	\end{center}
	Note that for $p$ odd, $\cC$ is the middle flow category in Example \ref{exam:get_epsilon}.
\end{example}

\section{Musical scores}
\label{sec:musical_scores}

\begin{definition}A \emph{stave} is a diagram consisting of a section of the $(x,y)$-plane in which are drawn 4 line segments, which we call \emph{levels}, as follows:
\begin{center}
\begin{tikzpicture}

\score{0.6cm}{10cm}

\node at (-0.8,0.05) {$y=0$};
\node at (-0.8,0.65) {$y=1$};
\node at (-0.8,1.25) {$y=2$};
\node at (-0.8,1.85) {$y=3$};

\end{tikzpicture}
\end{center}
(In general the $y$-values of the levels are suppressed from the notation as understood.) A \emph{score diagram} is a graph drawn on a stave, where vertices are only permitted to be drawn on the line segments and edges are drawn as straight lines.
\end{definition}

We want to draw our flow categories as score diagrams from this point onwards, essentially by using the graph $\Gamma_2$ of the category. Certainly this is not a sensible idea for all flow categories due to the potentially enormous loss of information, but for flow categories in primary Smith normal form of width $3$ this is rather useful.

A framed flow category in primary Smith normal form has the property that all $1$-dimensional moduli spaces are closed, and, after using extended Whitney tricks, we can assume they only contain at most one non-trivially framed circle $\eta$. A flow category with this property will be called \emph{reduced}. In a reduced flow category we will drop the labelling $\eta$ from from the corresponding edges.

\begin{definition}
Let $\cC$ be a reduced framed flow category of width 3. The \emph{score of the category} $\Sigma=\Sigma(\cC)$ is defined to be the score diagram obtained by taking the graph $\Gamma_2(\cC)$ and drawing the vertices on a stave so that a vertex $v_a$ with $|a|=r+n$ is drawn on level $y=r$, and so that edges are drawn as straight lines between vertices so as to form a score diagram.
\end{definition}

The main purpose of the score is to tell us which handle slides should be used to move the flow category into an even simpler form. For this reason we will be mainly concerned with local pictures of the score which only contain those objects involved in handle slides, and with indications only for their neighboring objects. For example, the score
\begin{center}
\begin{tikzpicture}
\score{0.6}{2.8}
\hookbe{0.6}{$q$};
\hookmibe{1.4}{$p$};
\shookl{0.6}{0};
\end{tikzpicture}
\end{center}
describes part of a flow category with at least five objects $b,c_1,c_2,d_1,d_2$ such that $\M(b,d_1)=\eta=\M(b,d_2)$, $\M(c_1,d_1)$ consists of $q$ positively framed points, $\M(c_2,d_2)$ consists of $p$ positively framed points. Furthermore, the short edge emanating from $d_1$ indicates that there may be non-empty moduli spaces $\M(b',d_1)$ for other objects $b'$. If we were to perform a handle slide from $d_1$ over $d_2$, these moduli spaces would lead to a change in the moduli spaces $\M(b',d_2)$. So the short edge mainly serves as a reminder that handle slides can affect other moduli spaces. 
The fact that there is no short edge emanating from $d_2$ is not supposed to indicate that all other moduli spaces $\M(b',d_2)$ are empty, but that they will not play a role in the upcoming moves. 

Handle slides on reduced framed flow categories tend to result in flow categories which are not even in primary Smith normal form. We shall nevertheless draw the graph $\Gamma_2(\cC)$ of such a flow category on a stave and refer to it as the score. This should not cause any confusion.

\begin{lemma}
\label{lem:hook_be}
Let $q\geq p\geq 2$ be powers of $2$ and $\cC$ a reduced framed flow category with score $\Sigma(\cC)$ as given in one of the cases below. Then $\cC$ is move equivalent to a reduced framed flow category $\cC'$ with score $\Sigma(\cC')$ differing from $\Sigma(\cC)$ only as given in the local pictures below.
\begin{center}
\begin{tikzpicture}
\draw (-1,0.6) node{$(1)$};
\draw (0,0.6) node{$\Sigma(\cC)\colon$};
\scoretx{0.6}{2.4}{0.7}
\hookbe{1.1}{$q$};
\hookmibe{1.9}{$p$};
\shookl{1.1}{0};
\draw [->] [decorate, decoration={snake}] (3.3,0.6) -- (4.5,0.6);
\draw (5.3,0.6) node{$\Sigma(\cC')\colon$};
\scoretx{0.6}{2.4}{6}
\hookbe{6.4}{$q$};
\hookmib{8}{$p$};
\shookl{6.4}{0};
\shookl{8}{0};
\end{tikzpicture}
\end{center}
\vspace{0.3cm}
\begin{center}
\begin{tikzpicture}
\draw (-1,0.6) node{$(2)$};
\draw (0,0.6) node{$\Sigma(\cC)\colon$};
\scoretx{0.6}{2.4}{0.7}
\hookbe{1.1}{$q$};
\hookmibe{1.9}{$q$};
\shookl{1.1}{0};
\shookl{1.1}{0.6};
\shookl{3}{0.4};
\draw [->] [decorate, decoration={snake}] (3.3,0.6) -- (4.5,0.6);
\draw (5.3,0.6) node{$\Sigma(\cC')\colon$};
\scoretx{0.6}{2.4}{6}
\hookbe{6.4}{$q$};
\hookmib{8}{$q$};
\shookl{6.4}{0};
\shookl{6.4}{0.6};
\shookl{8.3}{0.4};
\shookl{8}{0};
\shookl{8}{0.6};
\shookl{6.7}{0.4};
\end{tikzpicture}
\end{center}
We also get dual versions as follows.
\begin{center}
\begin{tikzpicture}
\draw (-1,0.6) node{$(3)$};
\draw (0,0.6) node{$\Sigma(\cC)\colon$};
\scoretx{0.6}{2.4}{0.7}
\hookmiem{1.1}{$p$};
\hookem{1.9}{$q$};
\shookl{3}{1};
\draw [->] [decorate, decoration={snake}] (3.3,0.6) -- (4.5,0.6);
\draw (5.3,0.6) node{$\Sigma(\cC')\colon$};
\scoretx{0.6}{2.4}{6}
\hookm{6.4}{$p$};
\hookem{7.2}{$q$};
\shookl{6.7}{1};
\shookl{8.3}{1};
\end{tikzpicture}
\end{center}
\vspace{0.3cm}
\begin{center}
\begin{tikzpicture}
\draw (-1,0.6) node{$(4)$};
\draw (0,0.6) node{$\Sigma(\cC)\colon$};
\scoretx{0.6}{2.4}{0.7}
\hookmiem{1.1}{$p$};
\hookem{1.9}{$p$};
\shookl{3}{1};
\shookl{1.1}{0.6};
\shookl{3}{0.4};
\draw [->] [decorate, decoration={snake}] (3.3,0.6) -- (4.5,0.6);
\draw (5.3,0.6) node{$\Sigma(\cC')\colon$};
\scoretx{0.6}{2.4}{6}
\hookm{6.4}{$p$};
\hookem{7.2}{$p$};
\shookl{6.7}{1};
\shookl{6.4}{0.6};
\shookl{8.3}{0.4};
\shookl{8.3}{1};
\shookl{8}{0.6};
\shookl{6.7}{0.4};
\end{tikzpicture}
\end{center}
If one moduli space is empty, we get the following.
\begin{center}
\begin{tikzpicture}
\draw (-1,0.6) node{$(5)$};
\draw (0,0.6) node{$\Sigma(\cC)\colon$};
\scoretx{0.6}{2.4}{0.7}
\hooke{1.1};
\hookmibe{1.9}{$q$};
\shookl{1.1}{0};
\draw [->] [decorate, decoration={snake}] (3.3,0.6) -- (4.5,0.6);
\draw (5.3,0.6) node{$\Sigma(\cC')\colon$};
\scoretx{0.6}{2.4}{6}
\hooke{6.4};
\hookmib{8}{$q$};
\shookl{6.4}{0};
\shookl{8}{0};
\end{tikzpicture}
\end{center}
\vspace{0.3cm}
\begin{center}
\begin{tikzpicture}
\draw (-1,0.6) node{$(6)$};
\draw (0,0.6) node{$\Sigma(\cC)\colon$};
\scoretx{0.6}{2.4}{0.7}
\hookmiem{1.1}{$p$};
\hooke{1.9};
\shookl{3}{1};
\draw [->] [decorate, decoration={snake}] (3.3,0.6) -- (4.5,0.6);
\draw (5.3,0.6) node{$\Sigma(\cC')\colon$};
\scoretx{0.6}{2.4}{6}
\hookm{6.4}{$p$};
\hooke{7.2};
\shookl{6.7}{1};
\shookl{8.3}{1};
\end{tikzpicture}
\end{center}
\end{lemma}

\begin{proof}
Call the objects $b,c_1,d_1,c_2,d_2$ where $\M(b,d_1)=\eta=\M(b,d_2)$, $\M(c_1,d_1)$ has $q$ points and $\M(c_2,d_2)$ has $p$ points. In the following score diagrams the arrows indicate the level of the handle slide and its direction.
\begin{center}
\begin{tikzpicture}
\scoretx{0.6}{2.4}{0.7}
\hookbe{1.1}{$q$};
\hookmibe{1.9}{$p$};
\shookl{1.1}{0};
\draw [->] (3.3,0) -- (4.5,0);
\node at (3.9,0.2) {$\times 1$};
\scoretx{0.6}{2.4}{4.7}
\hookbe{5.1}{$q$};
\hookmib{6.7}{$p$};
\draw [-] (5.1,0.6) -- (6.7,0);
\node at (5.9,0.4) {$-q$};
\shookl{5.1}{0};
\shookl{6.7}{0};
\draw [->] (7.3,0.6) -- (8.5,0.6);
\node at (7.9,0.8) {$\times q/p$};
\scoretx{0.6}{2.4}{8.7}
\hookbe{9.1}{$q$};
\hookmib{10.7}{$p$};
\shookl{9.1}{0};
\shookl{10.7}{0};
\end{tikzpicture}
\end{center}
Note that if $q>p$, the second handle slide is done an even number of times, so that any $\eta$ landing in $c_1$ has no effect. However, if $q=p$, each $\eta$ in $\M(a,c_1)$ for some object $a$ with $|a|=|c_1|+2$ leads to an $\eta$ in $\M(a,c_2)$ (although it may cancel with a pre-existing $\eta$ using a Whitney trick, not affecting other parts of the score). Similarly, any $\eta$ in $\M(c_2,e)$ leads to an $\eta$ in $\M(c_1,e)$. This shows the cases (1) and (2). If the object $c_1$ is not present, we can stop after the first slide and we get case (5). The dual cases are done by turning the score diagrams and the handle slides upside down.
\end{proof}

A similar result holds if a non-empty $0$-dimensional moduli space sits above or below the $\eta$:

\begin{lemma}
 \label{lem:hook_te}
Let $q>p\geq 2$ be powers of $2$ and $\cC$ a reduced framed flow category with score $\Sigma(\cC)$ as given in one of the cases below. Then $\cC$ is move equivalent to a reduced framed flow category $\cC'$ with score $\Sigma(\cC')$ differing from $\Sigma(\cC)$ only as given in the local pictures below.
\begin{center}
\begin{tikzpicture}
\draw (-1,0.9) node{$(1)$};
\draw (0,0.9) node{$\Sigma(\cC)\colon$};
\scorex{0.6}{2.4}{0.7}
\hookmiet{1.1}{$q$};
\hooket{1.9}{$p$};
\shookl{3}{1};
\draw [->] [decorate, decoration={snake}] (3.3,0.9) -- (4.5,0.9);
\draw (5.3,0.9) node{$\Sigma(\cC')\colon$};
\scorex{0.6}{2.4}{6}
\hookt{6.4}{$q$};
\hooket{7.2}{$p$};
\shookl{6.7}{1};
\shookl{8.3}{1};
\end{tikzpicture}
\end{center}
\vspace{0.3cm}
\begin{center}
\begin{tikzpicture}
\draw (-1,0.9) node{$(2)$};
\draw (0,0.9) node{$\Sigma(\cC)\colon$};
\scorex{0.6}{2.4}{0.7}
\hookmiet{1.1}{$p$};
\hooket{1.9}{$p$};
\shookl{3}{1};
\shookl{3}{1.6};
\draw [->] [decorate, decoration={snake}] (3.3,0.9) -- (4.5,0.9);
\draw (5.3,0.9) node{$\Sigma(\cC')\colon$};
\scorex{0.6}{2.4}{6}
\hookt{6.4}{$p$};
\hooket{7.2}{$p$};
\shookl{6.7}{1.6};
\shookl{8.3}{1};
\shookl{8.3}{1.6};
\shookl{6.7}{1};
\end{tikzpicture}
\end{center}
\vspace{0.3cm}
\begin{center}
\begin{tikzpicture}
\draw (-1,0.9) node{$(3)$};
\draw (0,0.9) node{$\Sigma(\cC)\colon$};
\scorex{0.6}{2.4}{0.7}
\hookmie{1.1};
\hooket{1.9}{$p$};
\shookl{3}{1};
\draw [->] [decorate, decoration={snake}] (3.3,0.9) -- (4.5,0.9);
\draw (5.3,0.9) node{$\Sigma(\cC')\colon$};
\scorex{0.6}{2.4}{6}
\dotb{6.4};
\hooket{7.2}{$p$};
\shookl{8.3}{1};
\shookl{6.7}{1};
\end{tikzpicture}
\end{center}
\vspace{0.3cm}
\begin{center}
\begin{tikzpicture}
\draw (-1,0.9) node{$(4)$};
\draw (0,0.9) node{$\Sigma(\cC)\colon$};
\scorex{0.6}{2.4}{0.7}
\hookmie{1.1};
\hooke{1.9};
\shookl{3}{1};
\draw [->] [decorate, decoration={snake}] (3.3,0.9) -- (4.5,0.9);
\draw (5.3,0.9) node{$\Sigma(\cC')\colon$};
\scorex{0.6}{2.4}{6}
\dotb{6.4};
\hooke{7.2};
\shookl{8.3}{1};
\shookl{6.7}{1};
\end{tikzpicture}
\end{center}
Again we have the dual cases obtained by turning the scores upside down.
\end{lemma}

\begin{proof}
Call the objects $a_1,b_1,a_2,b_2,d$ where $\M(b_1,d)=\eta=\M(b_2,d)$, $\M(a_1,b_1)$ has $q$ points and $\M(a_2,b_2)$ has $p$ points. We do handle slides as follows.
\begin{center}
\begin{tikzpicture}
\scorex{0.6}{2.4}{0.7}
\hookmiet{1.1}{$q$};
\hooket{1.9}{$p$};
\shookl{3}{1};
\draw [->] (3.3,1.2) -- (4.5,1.2);
\node at (3.9,0.95) {$\times 1$};
\scorex{0.6}{2.4}{4.7}
\hookt{5.1}{$q$};
\hooket{5.9}{$p$};
\draw [-] (5.1,1.8) -- (6.7,1.2);
\node at (5.9,1.6) {$-q$};
\shookl{5.4}{1};
\shookl{7}{1};
\draw [->] (7.3,1.8) -- (8.5,1.8);
\node at (7.9,1.55) {$\times q/p$};
\scorex{0.6}{2.4}{8.7}
\hookt{9.1}{$q$};
\hooket{9.9}{$p$};
\shookl{9.4}{1};
\shookl{11}{1};
\end{tikzpicture}
\end{center}
This shows case (1). The other cases follow as in the proof of Lemma \ref{lem:hook_be}.
\end{proof}

\section{The Chang spaces}
\label{sec:chang}
Equipped with our new score notation and the Smith normal form, we move on to the next simplification of a general width 3 flow category, which we will call the \emph{Chang form}. The terminology is a reference to the stable homotopy classification of CW complexes of homological width $2$, which was first obtained by Chang~\cite{Chang}.

\begin{definition}
	Let $p,q\geq 2$ be powers of $2$ and $n\in \Z$. Let $\cC(_q\eta p,n)$ be the framed flow category with objects $a,b_1,b_2,c$, where $|a|-2 = |b_1|-1=|b_2|-1=|c|=n$, $\M(a,c)=\eta$, $\M(a,b_1)$ consists of $p$ positively framed points, $\M(b_2,c)$ consists of $q$ positively framed points, and all other moduli spaces are empty.
	
	Let $\cC(\eta p,n)$ be the full subcategory with objects $a,b_1,c$, $\cC(_q\eta,n)$ the full subcategory with objects $a,b_2,c$, and $\cC(\eta,n)$ the full subcategory with objects $a,c$.
	
	For $w\in \{_q\eta p,\, _q\eta, \eta p,\eta\}$ we call $\cC(w,n)$ the \emph{Chang flow category of the word $w$}.
\end{definition}

We shall often treat $n$ as unspecified, or simply as $n=0$, and shorten the name of the flow category to $\cC(w)$.

The scores of these four types of Chang flow categories are given by 
\begin{center}
\begin{tikzpicture}
\scoret{0.6cm}{8.8}
\hookbem{1}{$q$}{$p$};
\hookbe{3}{$q$};
\hookem{5}{$p$};
\hooke{7};
\end{tikzpicture}
\end{center}
\begin{definition}A framed flow category of width 3 is in \emph{Chang form} if the following conditions are satisfied:
\begin{enumerate}
\item It is in primary Smith normal form.
\item Every non-empty 1-dimensional moduli space is a single framed circle $\eta$.
\item In the score of the category, the maximum number of length 2 edges connecting to any given vertex in level 0, or in level 2, is one.
\end{enumerate}
\end{definition}

So if a framed flow category of width 3 is in Chang form, removing the objects of degree 3 leads to a disjoint union of elementary Moore and Chang flow categories. Notice that a flow category is reduced if only (1) and (2) are satisfied.

\begin{theorem}\label{thm:chang}Any framed flow category with homological width 3 is move equivalent to a framed flow category in Chang form.
\end{theorem}

\begin{proof}By Theorem \ref{thm:smith} we may perform moves so that the category is in primary Smith normal form. In this form, all 1-dimensional moduli spaces must now be \emph{closed} 1-manifolds. Using the extended Whitney trick we assume that $\cC$ is reduced.

By Corollary \ref{cor:odd_Moore} we can assume that any pair of objects $x$ and $y$ with $\M(x,y)$ an odd number of points, is isolated from the remaining flow category.

We now prove the statement for all width 3 framed flow categories which are reduced and whose homology does not contain odd torsion, using induction on the number of objects. If there is only one object, the flow category is clearly in Chang form.

Let $X$ be the set of objects of minimal degree. If for some $x\in X$ all $1$-dimensional moduli spaces $\M(b,x)$ are empty, we can ignore $x$ and deduce the result by induction. Hence we can assume that for all $x\in X$ there is an object $b$ with $\M(b,x)=\eta$.

As $\cC$ is in primary Smith normal form, for each $x\in X$ there is at most one non-empty $0$-dimensional moduli space $\M(c,x)$. Let $S(x)$ be the number of points in this $0$-dimensional moduli space, or set $S(x)=\infty$ if there is none. We define a pre-order on $X$ by $x \preceq x'$ if $S(x)\leq S(x')$.

Pick an $x\in X$ which is maximal in this pre-order. Let 
\[
B(x)=\{b\in \Ob(\cC)\,|\, \M(b,x)=\eta\}. 
\]
For $b\in B(x)$ define $S(b)\in \Q$ to be
\[
 S(b)= \left\{\begin{array}{cl}
              1/|\M(a,b)| & \mbox{if }\M(a,b)\not=\emptyset \mbox{ for some }a\mbox{ with }|a|-1=|b| \\
              -1/|\M(b,c)| & \mbox{if }\M(b,c)\not=\emptyset \mbox{ for some }c\mbox{ with }|b|=|c|+1 \\
              0 &\mbox{otherwise.}
             \end{array}
\right.
\]
Define a pre-order on $B(x)$ by $b \preccurlyeq b'$ if $S(b)\leq S(b')$, and let $b\in B(x)$ be maximal in this pre-order. We can visualize this pre-order as
\begin{center}
\begin{tikzpicture}
\scorex{0.6}{1.4}{0}
\hookem{0.3}{2};
\node at (1.8,0.9) {$\preccurlyeq$};
\scorex{0.6}{1.4}{2.2}
\hookem{2.4}{4};
\node at (4.4,0.9) {$\preccurlyeq \cdots \preccurlyeq$};
\scorex{0.6}{1.4}{5.2};
\hooke{5.5};
\node at (7.4,0.9) {$\preccurlyeq \cdots \preccurlyeq$};
\scorex{0.6}{1.4}{8.2}
\hooket{8.5}{4};
\node at (10,0.9) {$\preccurlyeq$};
\scorex{0.6}{1.4}{10.4}
\hooket{10.7}{2};
\end{tikzpicture}
\end{center}

By Lemmas \ref{lem:hook_be} and \ref{lem:hook_te} we get that $\cC$ is move equivalent to $\cC'$ in primary Smith normal form such that $\M_{\cC'}(b,x)=\eta$, and $\M_{\cC'}(b',x)=\emptyset$ for all objects $b'\not=b$ with $|b'|=|b|$. Using extended Whitney tricks, we can assume $\cC'$ to be reduced.

We may have $\M(b,x')=\eta$ for other $x'\in X\setminus\{x\}$, but since $x$ was chosen to be maximal with respect to $\preceq$, we can use Lemma \ref{lem:hook_be} to get $\M(b,x')=\emptyset$ for all $x'\in X\setminus\{x\}$. We now satisfy condition (3) for $x$ and $b$, and we can proceed by induction to get our flow category into Chang form.\end{proof}

\begin{remark}\label{rem:algo_Chang}
The proof describes an algorithm to turn a framed flow category in primary Smith normal form into a framed flow category in Chang form: identify a pair of objects $(b,x)$ maximal in the pre-orders, and perform handle slides as in Lemma \ref{lem:hook_te} and \ref{lem:hook_be}. Repeat until the flow category is in Chang form. This algorithm has been implemented in a computer program by the third author \cite{SchuetzKJ}.
\end{remark}

At this stage we can prove a weaker version of Theorem \ref{thm:main}.

\begin{corollary}\label{cor:changclassification}Suppose $\cC_1$ and $\cC_2$ are framed flow categories such that there is a homotopy equivalence of spectra $\mathcal{X}(\cC_1)\simeq \mathcal{X}(\cC_2)$, and $\cC_1$ is width $2$. Then $\cC_1$ and~$\cC_2$ are move equivalent.
\end{corollary}

\begin{proof}
We can use Theorem \ref{thm:chang} to get both $\cC_1$ and $\cC_2$ into Chang form. Because their homological width is $2$, the Chang forms are disjoint unions of Chang flow categories and elementary Moore flow categories. By the uniqueness of the Chang homotopy types, see \cite{Chang}, the two Chang forms are equal up to stabilization and isotopy.
\end{proof}

The reader may wonder to what extent we require Chang's classification \cite{Chang} and to what extent we recover it. We argue that in fact we can recover Chang's classification by our methods, but only stably. To show this, we first record a proposition that may be known to experts but, to our knowledge, has not appeared in the literature before.

\begin{proposition}\label{prop:CWisflow} Let $X$ be a finite CW complex. Then there exists a framed flow category $\cC$ such that there is a stable homotopy equivalence $\mathcal{X}(\cC)\simeq \Sigma^\infty X$.
\end{proposition}

\begin{proof} By suspending $X$ once, we may assume that $X$ is simply connected. We first show that any simply connected CW complex is homotopy equivalent to a compact smooth manifold of some dimension. We employ some arguments from piecewise linear ($\PL$) topology and we use the definitions and results from \cite{MR0248844}.

There exists a finite polyhedron $K$ that is homotopy equivalent to $X$. The polyhedron $K$ is a subset of $\R^n$ for some $n$, and we assume for later that $n\geq 6$. The polyhedron $K\subset \R^n$ has a closed neighbourhood $N$, such that $N\simeq K$, and such that $N$ is a $\PL$ manifold with boundary \cite[Theorem II.2.11]{MR0248844}. This is also called a \emph{$\PL$ regular neighbourhood} of $K$.

Consider that $\operatorname{int}(N):=N\setminus \partial N$ is an open subset of $\R^n$, and so is endowed with the structure of a smooth manifold. We wish to show that the smooth manifold $\operatorname{int}(N)$ is moreover the interior of a smooth manifold with boundary by applying \cite[Theorem 1]{MR189046}. As $n\geq 6$, all that we must check is that $\operatorname{int}(N)$ is \emph{simply connected at infinity}, that is to say for any compact $C\subset \operatorname{int}(N)$, there is a compact $D$ with $C\subset D\subset \operatorname{int}(N)$ such that $\operatorname{int}(N)\setminus D$ is simply connected. For this, we will furthermore assume that we increased $n$ so that $n>k+2$, where $k$ denotes the dimension of the polyhedron $K$.

Now consider that as $X$ was simply connected, so is $N$. Hence any loop $\gamma\subset \partial N$ is nullhomotopic in $\operatorname{int}(N)$. We may assume that the loop is a $\PL$ submanifold of $\partial N$ and that the nullhomotopy is represented by a $\PL$ disc in $N$. Moreover, as $n>k+2$,  we may assume by general position (see \cite[\textsection 4]{MR0248844}) that the disc does not intersect $K$. Now take a sufficiently small \emph{derived neighbourhood} $N'$ of $K$ (see \cite[\textsection I.2]{MR0248844}) so that the disc lies entirely in $N\setminus N'$. The derived neighbourhood $N'$ is again a regular neighbourhood \cite[Theorem II.2.11]{MR0248844}, so by the generalised $\PL$ annulus property \cite[Corollary II.2.16.2]{MR0248844}, there is a $\PL$ homeomorphism between $\operatorname{cl}(N\setminus N')$ and $\partial N\times[0,1]$. Thus the loop $\gamma$ is contractible in $\partial N$, and it follows that $\operatorname{int}(N)$ is simply connected at infinity, as required.

We write $(Z,\partial Z)$ for a smooth manifold with boundary such that $Z\setminus \partial Z$ is diffeomorphic to $N\setminus \partial N$. Note that $Z$ is homotopy equivalent to the original CW complex $X$. Pick a generic Morse function on $Z$. One may then obtain a framed flow category $\cC$ by the method described by Cohen-Jones-Segal \cite{CJS} (see also \cite{cohenrevisit}). There are then stable homotopy equivalences $\Sigma^\infty X\simeq\Sigma^\infty Z\simeq \mathcal{X}(\cC)$.
\end{proof}

With this we can recover the main classification result of \cite{Chang}, but only up to stable homotopy.

\begin{corollary}\label{cor:changrecover} Every width 2 finite CW complex whose reduced homology is supported in degree $\geq 4$ is stably homotopy equivalent to one of the form $X_1\vee X_2\vee \dots\vee X_N$, where each $X_i$ is a Chang space, and this decomposition is unique up to permutation.
\end{corollary}

\begin{proof} Given such a CW complex $X$, obtain a framed flow category $\cC$ as in Proposition \ref{prop:CWisflow}. Now apply Theorem \ref{thm:chang} so that $\cC$ is in Chang form. This provides a stable homotopy equivalence from $X$ to $X_1\vee X_2\vee \dots\vee X_N$, where each $X_i$ is a Chang space.

Finally, we do not need to reference \cite{Chang} to obtain the uniqueness of the various stable homotopy types which arise from the collection of Chang flow categories. This can alternatively be derived from kernels and ranks of the second Steenrod squares and various Bockstein homomorphisms, as described in \cite[\S 4]{LipSarSq}.
\end{proof}

\section{The Baues-Hennes spaces}
\label{sec:baues_hennes}

In \cite{BauHen}, Baues and Hennes classify $(n+3)$-dimensional $(n-1)$-connected finite CW complexes for $n\geq 4$; see also \cite{Baues}. In this section we describe their list in terms of framed flow categories. It is worth pointing out that \cite{Baues} and~\cite{BauHen} use different conventions regarding the torsion coefficients. We will follow the convention of \cite{Baues}. 

Let
\begin{align*}
R &= \{ (r,-1)\in \Z^2\,|\, r = 2^i, i\geq 1\}\\
T &= \{ (t,0)\in \Z^2\,|\, t = 2^i, i\geq 1\}\\
S &= \{ (s,1)\in \Z^2\,|\,s = 2^i, i\geq 1\}
\end{align*}
and let
\[
\mathcal{A}= R \cup T \cup S \cup \{\xi, \eta, \varepsilon\}
\]
be an alphabet. We denote the empty word in $\mathcal{A}$ as $\emptyset$.
\begin{definition}
A finite word $w=w_1\cdots w_k$ with $k\geq 1$ in the alphabet $\mathcal{A}$ is called a \emph{basic word}, if
\begin{enumerate}
	\item $w_1\in R\cup S \cup \{\xi, \eta\}$.
	\item For $1 \leq i \leq k-1$,
	\begin{itemize}
		\item if $w_i\in S$, then $w_{i+1}= \eta$.
		\item if $w_i=\eta$, then $w_{i+1} \in R$.
		\item if $w_i\in R$, then $w_{i+1}= \xi$.
		\item if $w_i=\xi$, then $w_{i+1}\in S$.
	\end{itemize}
\end{enumerate}
\end{definition}
To streamline notation, we write $_r=(r,-1)$, $t=(t,0)$ and $^s=(s,1)$. A basic word is then a finite, non-empty subword of
\[
\cdots \xi^{s_{i}}\eta_{r_i}\xi^{s_{i+1}}\eta_{r_{i+1}} \cdots
\]
with $r_i,s_i$ powers of $2$. 

Recall that we use $\eta$ to be mean a non-trivially framed circle. In a basic word the letter $\eta$ is indeed meant to symbolize such a framed circle. In fact, the letter $\xi$ also stands for a non-trivially framed circle, but on a different level. As we shall see below, the letter $\varepsilon$ is representing the non-trivially framed torus, which we also call $\varepsilon$.

\begin{definition}
	\label{defn:bh_flow_category}
	Let $w=w_1\cdots w_k$ be a basic word, and $n\in \Z$. A reduced framed flow category $\cC$ is called a \emph{Baues--Hennes flow category of $w$}, if the following are satisfied.
	\begin{enumerate}
		\item The object set is $\Ob(\cC) = \{x_1,\ldots,x_{k+1}\}$.
		\item For $1 \leq i \leq k$,
		\begin{itemize}
			\item if $w_i\in S$, then $|x_i|=n+3$, $|x_{i+1}|=n+2$ and $\M(x_i,x_{i+1})$ consists of $w_i$ positively framed points.
			\item if $w_i=\eta$, then $|x_i|=n+2$, $|x_{i+1}|=n$ and $\M(x_i,x_{i+1})=\eta$.
			\item if $w_i\in R$, then $|x_i|=n$, $|x_{i+1}|=n+1$ and $\M(x_{i+1},x_i)$ consists of $w_i$ positively framed points.
			\item if $w_i=\xi$, then $|x_i|=n+1$, $|x_{i+1}|= n+3$ and $\M(x_{i+1},x_i)=\eta$. 
		\end{itemize}
		\item No other $0$- or $1$-dimensional moduli spaces are non-empty.
	\end{enumerate}
	We refer to $x_1$ as the \emph{start object}, and to $x_{k+1}$ as the \emph{end object} of the Baues--Hennes flow category.
\end{definition}

It is easy to see that for every basic word $w$ there exists a Baues--Hennes flow category of $w$, and its score is a sub-tree of
\begin{center}
\begin{tikzpicture}
\score{0.6}{5.5}
\hookx{1};
\hookteb{1.8}{$s_i$}{$r_i$};
\hookx{2.6};
\hookteb{3.4}{$s_{i+1}$}{$r_{i+1}$};
\draw (1,0.6) -- (1,0.3);
\draw (4.2,0.6) -- (4.4,0.9);
\node at (0.5,0.9) {$\cdots$};
\node at (5,0.9) {$\cdots$};
\end{tikzpicture}
\end{center}

\begin{definition} Given a word $w=w_1\cdots w_k$ with all $w_i\in\mathcal{A}$, we write $\bar{w} = w_k\cdots w_1$.
\end{definition}

\begin{definition}
	A finite word $w=u_k\cdots u_1 t v_1\cdots v_l$ in $\mathcal{A}$ is called a \emph{central word}, if $u = u_1\cdots u_k$ and $v = v_1\cdots v_l$ are either $\emptyset$ or basic words, and $t\in T$. Furthermore, if $u$ is non-empty, then $u_1=\eta$, and if $v$ is non-empty, then $v_1=\xi$.
	
	A reduced framed flow category $\cC$ is called a \emph{central Baues--Hennes flow category of the word $w$}, if the following are satisfied.
	\begin{enumerate}
		\item The object set is $\Ob(\cC) = \{x_{-k},\ldots,x_0,x_1,\ldots,x_{l+1} \}$.
		\item The full subcategory with objects $\{x_0,\ldots,x_{-k} \}$ is a Baues--Hennes category of $u$ with start object $x_0$ if $u$ is non-empty, or $\Sp^{n+2}$ otherwise.
		\item The full subcategory with objects $\{x_1,\ldots,x_{l+1}\}$ is a Baues--Hennes category of $v$ with start object $x_1$ if $v$ is non-empty, or $\Sp^{n+1}$ otherwise.
		\item The $0$-dimensional moduli space $\M(x_0,x_1)$ consists of $t$ positively framed points.
		\item No other $0$- or $1$-dimensional moduli spaces are non-empty.
	\end{enumerate}
\end{definition}

The score of a central Baues--Hennes flow category is a sub-tree of
\[
\begin{tikzpicture}
\score{0.6}{6.3}
\hookbet{1}{$s_0$}{$r_{-1}$};
\hookmix{1.8};
\hookbem{2.6}{$r_0$}{ };
\hookx{3.4};
\hookteb{4.2}{$s_{1}$}{$r_{1}$};
\draw (1,0.6) -- (0.8,0.9);
\draw (5,0.6) -- (5.2,0.9);
\node at (0.5,0.9) {$\cdots$};
\node at (5.8,0.9) {$\cdots$};
\node at (3.5,0.9) {$t$};
\end{tikzpicture}
\]
containing the $t$-edge.

\begin{definition}
	A finite word $w=u_k\cdots u_1\varepsilon v_1\cdots v_l$ in $\mathcal{A}$ is called an $\varepsilon$-word, if $u = u_1\cdots u_k$ and $v = v_1\cdots v_l$ are either $\emptyset$ or basic words. Furthermore, if $u$ is non-empty, then $u_1\in S$, and if $v$ is non-empty, then $v_1\in R$.
	
	A reduced framed flow category $\cC$ is called an \emph{$\varepsilon$-Baues--Hennes flow category of the word $w$}, if the following are satisfied.
	\begin{enumerate}
		\item The object set is $\Ob(\cC) = \{x_{-k},\ldots,x_0,x_1,\ldots,x_{l+1} \}$.
		\item The full subcategory with objects $\{x_0,\ldots,x_{-k} \}$ is a Baues--Hennes category of $u$ with start object $x_0$ if $u$ is non-empty, or $\Sp^{n}$ otherwise.
		\item The full subcategory with objects $\{x_1,\ldots,x_{l+1}\}$ is a Baues--Hennes category of $v$ with start object $x_1$ if $v$ is non-empty, or $\Sp^{n+3}$ otherwise.
		\item The $2$-dimensional moduli space $\M(x_1,x_0)=\varepsilon$.
		\item No other $0$- or $1$-dimensional moduli spaces are non-empty.
	\end{enumerate}
\end{definition}

The score of an $\varepsilon$-Baues--Hennes flow category is the sub-tree of
\begin{center}
\begin{tikzpicture}
\score{0.6}{6.3}
\hookbet{1}{$s_0$}{$r_{-1}$};
\hookmix{1.8};
\hookb{2.6}{$r_0$};
\hookteb{4.2}{$s_{1}$}{$r_{1}$};
\draw (1,0.6) -- (0.8,0.9);
\draw (5,0.6) -- (5.2,0.9);
\draw (2.6,0) -- (4.2,1.8);
\node at (0.5,0.9) {$\cdots$};
\node at (5.8,0.9) {$\cdots$};
\node at (3.6,0.8) {$\varepsilon$};
\end{tikzpicture}
\end{center}
obtained by removing the $\varepsilon$-edge.

\begin{remark}\label{rem:confusion}
We address a potential source of confusion. Consider a central word $w=\bar{u}tv=u_k\cdots u_1 t v_1\cdots v_l$. Above, we assigned a specific score diagram to the flow category of $w$ and a specific score diagram to the flow category of $u$ and of $v$ as basic words. Note that (unless $u$ is one letter long) $\bar{u}$ is not basic. We stress that, while the specified diagram for $v$ is precisely a subdiagram of the specified diagram for $w$, the specified diagram for $u$ only appears as a subdiagram after reflecting left-to-right. In general, the effect of left-right reflection on a score diagram does give an alternative score diagram of a flow category. So some score diagram for $u$ does appear as a subdiagram of the score diagram of $w$, just not precisely the standard one for a basic word. Similar considerations apply for $\varepsilon$-words.
\end{remark}

\begin{definition}
	A \emph{cyclic word} is a pair $(w,A)$, where $w=w_1\cdots w_{4k}$ is a basic word with $k\geq 1$ such that $w_1 = \xi$, and $A$ is an invertible $m\times m$ matrix over $\F_2$, the field with two elements, for some $m\geq 1$.
	
	A reduced framed flow category $\cC$ is called a \emph{cyclic Baues--Hennes flow category of the word $(w,A)$} if the following are satisfied.
	\begin{enumerate}
		\item The object set is given by $\Ob(\cC) = \{x_i^j\,|\, i=1,\ldots,4k, j=1,\ldots, m \}$.
		\item There exist $m$ disjoint subcategories $\cC_1,\ldots, \cC_m$ such that each $\cC_j$ is a Baues--Hennes flow category of $w_2\cdots w_{4k}$ with object set 
		\[
		\Ob(\cC_j)= \{ x_1^j,\ldots, x_{4k}^j\} 
		\]
		such that $x_1^j$  is the start object and $x_{4k}^j$ is the end object.
		\item For $i,j\in \{1,\ldots, m\}$ the $1$-dimensional moduli space 
		\[
			\M(x^i_1, x^j_{4k}) = \left\{ \begin{array}{cc}
				\eta & \mbox{if } A_{ij} = 1 \\
				\emptyset & \mbox{if } A_{ij} = 0
			\end{array}
			\right.
		\]
		\item No other $0$- or $1$-dimensional moduli spaces are non-empty.
	\end{enumerate}	
\end{definition}

\begin{example}
	Let $w= \xi ^2\eta_4$, and 
	\[
	A = \left( \begin{matrix}1\end{matrix} \right) \hspace{1cm} B = \left( \begin{matrix} 1 & 1 \\ 1 & 0 \end{matrix} \right)
	\]
	The scores of the Baues--Hennes flow categories of $(w,A)$ and $(w,B)$ are given by
	\begin{center}
	\begin{tikzpicture}
		\score{0.6}{7}
		\hookteb{1}{}{$4$};
		\node at (0.8,1.5) {$2$};
		\hookmix{1};
		\node at (2.7,0.9) {and};
		\hookteb{3.6}{}{$4$};
		\node at (3.4,1.5) {$2$};
		\hookteb{5.2}{$2$}{$4$};
		\hookmix{3.6};
		\hookx{4.4};
		\draw (3.6,1.8) -- (6,0.6);
	\end{tikzpicture}
	\end{center}
\end{example}


\begin{lemma}\label{lem:def_BH}
	Let $\cC, \cC'$ be Baues--Hennes flow categories of the same word $w$. Then $\cC$ is move equivalent to a (de-)suspension of $\cC'$.
\end{lemma}

\begin{proof}
	All $0$- and $1$-dimensional moduli spaces are determined by the word, and the only flexibility left is in the $2$-dimensional moduli spaces. Assume we have two objects $x_i,x_j$ with $|x_i|=|x_j|+3$. Then $\partial \M(x_i,x_j)$ is either empty or an even number of non-trivially framed circles.
	
	If $\partial \M(x_i,x_j)$ is empty, we either are in the case of an $\varepsilon$-word and $\M(x_i,x_j)$ is a pre-described $\varepsilon$, or there exists an object $x_k$ with either $\M(x_i,x_k)=\eta$ or $\M(x_k,x_j)=\eta$.
	In the latter case we can use Lemma \ref{lem:trick2} to create an extra $\varepsilon$ in $\M(x_i,x_j)$. Note that if $x_k$ were a starting object in one of the two subcategories of a central Baues--Hennes flow category, we also create an even number of $\eta$'s in another moduli space $\M(x_0,x_j)$ or $\M(x_i,x_1)$, but these can be removed with an extended Whitney trick.
	If the word is cyclic, we choose $x_k$ with $|x_k|=|x_i|-1$ to avoid creating several $\varepsilon$ in different moduli spaces.
	
	The case where $\partial \M(x_i,x_j)$ is non-empty is treated in the same way as the latter case above. Recall from Subsection \ref{sec:framed_manifolds} that we can assume each component of $\partial \M(x_i,x_j)$ to be a framed cylinder, and two framed cylinders can be made framed cobordant by adding a non-trivially framed torus.
	
	By the discussion in Section \ref{sec:flow_framed_categories}, the two flow categories agree after finitely many moves.
\end{proof}

\begin{definition}
	Let $w$ be a basic word, a central word or an $\varepsilon$-word, and $n\in \Z$. We write $\B(w,n)$ for a Baues--Hennes flow category of $w$ whose objects have degrees between $n$ and $n+3$. If $(w,A)$ is a cyclic word, we write $\B((w,A),n)$ for a Baues--Hennes flow category of $(w,A)$  whose objects have degrees between $n$ and $n+3$. We may omit the $n$ from the notation.
\end{definition}

Some Baues-Hennes categories of $\varepsilon$-words can be further reduced to combinations of simpler categories with the homotopy type of Moore spaces, as the next two examples show.

\begin{example}\label{exam:no2eps2}
	Let $\cC$ be the framed flow category with objects $a,b,c,d$ such that $|a|=|b|+1=|c|+2=|d|+3$. Let $\M(a,b)=\{A_1,A_2\}$ and $\M(c,d)=\{C_1,C_2\}$ consist of two positively framed points each, and let $\M(b,c)=\{P,M\}$ contain a positively and a negatively framed point. We assume each moduli space $\M(a,c)$ and $\M(b,d)$ consist of two intervals (i.e.~has no closed components). For each of $i=1,2$ we specify that $(P,A_i)$, $(M,A_i)$ lie in the same connected component of $\M(a,c)$, and that $(C_i,P)$, $(C_i,M)$ lie in the same connected component of $\M(b,d)$. It follows that
	\begin{center}
		\begin{tikzpicture}
			\node at (-0.2,0) {$\partial\M(a,d) =$};
			\dotd{1};
			\dotd{2};
			\dotd{3};
			\dotd{4};
			\dotd{5};
			\dotd{6};
			\dotd{7};
			\dotd{8};
			\draw (1,0) to[out=45,in=135] (2,0);
			\draw (1,0) to[out=-45,in=-135] (2,0);
			\draw (3,0) to[out=45,in=135] (4,0);
			\draw (3,0) to[out=-45,in=-135] (4,0);
			\draw (5,0) to[out=45,in=135] (6,0);
			\draw (5,0) to[out=-45,in=-135] (6,0);
			\draw (7,0) to[out=45,in=135] (8,0);
			\draw (7,0) to[out=-45,in=-135] (8,0);
		\end{tikzpicture}
	\end{center}
	with the endpoints of the $2$-gons given by $(C_i,P,A_j)$ and $(C_i,M,A_j)$ for $i,j\in \{1,2\}$. If we choose the framing of the intervals in $\M(a,c)$ and $\M(b,d)$ identically, we can choose $\M(a,d)$ to be four discs, all trivially framed.
	
	Now perform a Whitney trick on $\M(b,c)$. This turns $\M(a,c)$ and $\M(b,d)$ into two circles each, and $\M(a,d)$ into four cylinders, where each disc is turned into a cylinder.
	\begin{center}
		\begin{tikzpicture}
			\node at (0,0) {$\M(a,d) =$};
			\draw (1.5,-1) ellipse (0.5cm and 0.2cm);
			\draw (1.5,1) ellipse (0.5cm and 0.2cm);
			\draw (1,-1) -- (1,1);
			\draw (2,-1) -- (2,1);
			\draw (3.5,-1) ellipse (0.5cm and 0.2cm);
			\draw (3.5,1) ellipse (0.5cm and 0.2cm);
			\draw (3,-1) -- (3,1);
			\draw (4,-1) -- (4,1);
			\draw (5.5,-1) ellipse (0.5cm and 0.2cm);
			\draw (5.5,1) ellipse (0.5cm and 0.2cm);
			\draw (5,-1) -- (5,1);
			\draw (6,-1) -- (6,1);
			\draw (7.5,-1) ellipse (0.5cm and 0.2cm);
			\draw (7.5,1) ellipse (0.5cm and 0.2cm);
			\draw (7,-1) -- (7,1);
			\draw (8,-1) -- (8,1);
			\node at (1.5,-1.4) {$C_1 \times \xi_1$};
			\node at (1.5,1.4) {$\eta_1 \times A_1$};
			\node at (3.5,-1.4) {$C_1 \times \xi_2$};
			\node at (3.5,1.4) {$\eta_1 \times A_2$};
			\node at (5.5,-1.4) {$C_2 \times \xi_1$};
			\node at (5.5,1.4) {$\eta_2 \times A_1$};
			\node at (7.5,-1.4) {$C_2 \times \xi_2$};
			\node at (7.5,1.4) {$\eta_2 \times A_2$};
		\end{tikzpicture}
	\end{center}
	where we write $\M(a,c) = \xi_1\cup \xi_2$ and $\M(b,d) = \eta_1 \cup \eta_2$. Note that there are two possible Whitney tricks we could have performed to cancel $P$ against $M$ in $\M(b,c)$, and these correspond to the two different ways of framing an arc that is the nullbordism of $P\sqcup M$ used in the trick. The different choices result in the two circles in $\M(b,d)$ being either both non-trivially framed or both trivially framed.
	This forces the framing on the circles in $\M(a,c)$ to have the same framing as well, as witnessed by the framed cobordisms in $\M(a,d)$.
	
	We can now perform an extended Whitney trick in both $\M(a,c)$ and $\M(b,d)$ using a cylinder between $\xi_1$ and $\xi_2$, and a cylinder between $\eta_1$ and $\eta_2$. This adds four cylinders to $\M(a,d)$ which turns the new $\M(a,d)$ into a torus. With the same argument as in Example \ref{exam:Adem_relation2} this torus has a non-trivial framing on the longitude.
	If the framing of the circles is non-trivial, we thus get $\M(a,d) = \varepsilon$, but if the framing of the circles is trivial, so is the framing on the torus. As we can get both situations, we see that $\B(_2\varepsilon^2)$ is move equivalent to a disjoint union of Moore flow categories.
	
	Note that this argument works for any $\varepsilon$-Baues--Hennes flow category that contains the word $_2\varepsilon^2$.
\end{example}

\begin{example}\label{exam:Baues_style}
	Let $w= \xi_2\varepsilon v$, where $v$ is either empty or a basic word, and consider $\B(w)$. We have objects $a,c,d,a'$ with $|a|=|a'|=|c|+2=|d|+3$ and $\M(a,c)=\eta$, $\M(c,d) = \{C_1, C_2\}$ two positively framed points and $\M(a',d)=\varepsilon$. Depending on $v$ there may be further objects and moduli spaces.
	
	Note that $\M(a,d)$ consists of a cylinder between two non-trivially framed circles. If we slide $a'$ twice over $a$, we get two non-trivially framed circles in $\M(a',c)$, and two extra cylinders in $\M(a',d)$. Performing the extended Whitney trick on the two non-trivially framed circles in  $\M(a',c)$ turns this moduli space into the empty set, but adds (similarly to the previous example) an $\varepsilon$ to $\M(a',d)$. As we already have one $\varepsilon$ in $\M(a',d)$, we can do another extended Whitney trick to get this moduli space empty as well.
	
	Hence $\B(w)$ is move equivalent to $\B(_2\xi) \sqcup \cC$, where $\cC=\Sp$ or $\B(v)$.
\end{example}

To detect when we have an $\varepsilon$-word that can be further reduced, as in Examples \ref{exam:no2eps2} and \ref{exam:Baues_style}, we introduce the following notation, and the concept of a \emph{special} word.

\begin{definition}
	Let $\Xi$ be the set consisting of $\emptyset$ and all basic words that start with $\xi$. For $w\in \Xi$, define the \emph{$\sigma$-symbol of $w$} to be the sequence $\sigma(w)=(\sigma_1(w),\sigma_2(w),\ldots)$ with $\sigma_i(w)\in \Z\cup \{-\infty,\infty\}$ inductively defined as
	\[
		\sigma(w) = \left\{ \begin{array}{ll}
			(0,\ldots,0,\ldots) & \mbox{if }w=\emptyset\\
			(\infty,0,\ldots,0,\ldots) & \mbox{if }w=\xi \\
			(s,0,\ldots,0,\ldots) & \mbox{if }w = \xi^s \\
			(s,-\infty,0,\ldots,0,\ldots) & \mbox{if }w = \xi^s\eta \\
			(s,-r,\sigma_1(u),\sigma_2(u),\ldots)&\mbox{if } w = \xi^s\eta_ru\mbox{ with }u\in \Xi
		\end{array}
		\right.
	\]
	We also define the \emph{$\sigma^\ast$-symbol of $w$} to be the sequence $\sigma^\ast(w)=(2\sigma_1(w),\sigma_2(w),\ldots)$.
	 
	Similarly, let $\Eta$ be the set consisting of $\emptyset$ and all basic words that start with $\eta$. For $w\in \Eta$ define the \emph{$\rho$-symbol of $w$} to be the sequence $\rho(w)=(\rho_1(w),\rho_2(w),\ldots)$ with $\rho_i(w)\in \Z\cup \{-\infty,\infty\}$ inductively defined as
	\[
	\rho(w) = \left\{ \begin{array}{ll}
	(0,\ldots,0,\ldots) & \mbox{if }w=\emptyset\\
	(\infty,0,\ldots,0,\ldots) & \mbox{if }w=\eta \\
	(r,0,\ldots,0,\ldots) & \mbox{if }w = \eta_r \\
	(r,-\infty,0,\ldots,0,\ldots) & \mbox{if }w = \eta_r\xi \\
	(r,-s,\rho_1(u),\rho_2(u),\ldots)&\mbox{if } w = \eta_r\xi^su\mbox{ with }u\in \Eta
	\end{array}
	\right.
	\]
	We also define the \emph{$\rho^\ast$-symbol of $w$} to be the sequence $\rho^\ast(w)=(2\rho_1(w),\rho_2(w),\ldots)$.
\end{definition}

Note that the $\sigma$-symbol and the $\sigma^\ast$-symbol differ at most in the first entry, and similarly the $\rho$-symbol differs from $\rho^\ast$ in at most the first entry.

The set $\Z\cup\{-\infty,\infty\}$ has a total order, which induces a total order on the set of sequences in $\Z\cup\{-\infty,\infty\}$ by the lexicographical order. In particular, statements such as $\sigma(w) < \sigma(w')$ for words $w,w'\in \Xi$ are well defined.

For the next definition, note that if $w=\bar{u}\varepsilon v$ is an $\varepsilon$-word, then $\xi v\in \Xi$, and $\eta u \in \Eta$.

\begin{definition}\label{def:special_word}
	A word $w$ in $\mathcal{A}$ is called \emph{special}, if one of the following is satisfied.
	\begin{enumerate}
		\item The word $w$ is a basic word such that the corresponding Baues--Hennes flow category $\B(w)$ has objects $a,d\in \Ob(\B(w))$ such that $|a|=|d|+3$.
		\item The word $w$ is a central word such that the corresponding Baues--Hennes flow category $\B(w)$ has objects $a,d\in \Ob(\B(w))$ such that $|a|=|d|+3$.
		\item The word $w=\bar{u}\varepsilon v$ is an $\varepsilon$-word, and one of the three following conditions is satisfied.
		\begin{enumerate}
			\item We have $\sigma_1(\xi v) \geq 4$ and $\rho_1(\eta u) \geq 4$.
			\item We have $\sigma_1(\xi v) = 2$, $\rho_1(\eta u)\geq 4$, and $\rho^\ast(v')< \rho(\eta u)$, where $v'\in \Eta$ satisfies $v = \,\!^2v'$.
			\item We have $\sigma_1(\xi v) \geq 4$, $\rho_1(\eta u) = 2$, and $\sigma^\ast(u') < \sigma(\xi v)$, where $u'\in \Xi$ satisfies $u = \,\!_2u'$.
		\end{enumerate}
	\end{enumerate}
\end{definition}

The extra condition for basic and central words to be special is merely there to avoid repetition with Chang and Moore flow categories.

\begin{example}
	If an $\varepsilon$-word contains the sub-word $_2\varepsilon^2$, it is not special. The words $\varepsilon v$ are special if and only if $v\not = \,\!^2\eta$. Indeed, note that $\rho_1(\eta u) = \infty \geq 4$, so unless $v$ starts with $^2$, it is special. If $v=\,\!^2v'$ for some $v'$, we only need to check whether $\rho^\ast(v') < (\infty,0,\ldots)$, which is the case unless $v'=\eta$.
	Similarly, $\bar{u}\varepsilon$ is special if and only if $u\not=\,\!_2\xi$.
\end{example}

The reader may want to convince herself that an $\varepsilon$-word is special in the sense of Definition \ref{def:special_word} if and only if it is special in the sense of \cite{Baues}.

\begin{definition}
	Let $w=\xi^{s_1}\eta_{r_1}\cdots \xi^{s_p}\eta_{r_p}$ be a basic word. A \emph{cyclic permutation of $w$} is a basic word $w'= \xi^{s_i}\eta_{r_i}\cdots \xi^{s_p}\eta_{r_p}\xi^{s_1}\eta_{r_1}\cdots \xi^{s_{i-1}}\eta_{r_{i-1}}$ for some $i=1,\ldots,p$.
		
	Two cyclic words $(w,A)$ and $(w',A')$ are \emph{equivalent} if $w'$ is a cyclic permutation of $w$, and there is an invertible matrix $B$ over $\F_2$ with $A=B^{-1}A'B$.
	
	An equivalence class $[w,A]$ of a cyclic word $(w,A)$ is called a \emph{special cyclic word}, if $w$ is not of the form $w=w'w'\cdots w'$ where the right hand side is a $j$-fold power of a basic word $w'$ with $j>1$, and $A$ is indecomposable, that is, $A$ is not similar to a matrix $\left(\begin{matrix} A_1 & 0 \\ 0 & A_2\end{matrix}\right)$ over $\F_2$.
\end{definition}

As we shall see, equivalent cyclic words have move equivalent Baues--Hennes flow categories. Also, if $w=(w')^j$ for some $j>1$, we can express the Baues--Hennes flow category of $(w,A)$ as a Baues--Hennes flow category of a cyclic word $(w',A')$ for some larger matrix $A'$. If $A$ is decomposable, so will be its Baues--Hennes flow category.

\begin{definition}
	A framed flow category $\cC$ of homological width $3$ is said to be in \emph{Baues--Hennes form} if it is a disjoint union of elementary Moore flow categories, Chang flow categories and Baues--Hennes flow categories of special words.
\end{definition}

\begin{remark}\label{rem:promise}
Recall that our Baues--Hennes flow categories are framed, so applying the Cohen--Jones--Segal construction to them gives rise to stable homotopy types. In fact, these will have the same stable homotopy type as the spaces considered in \cite{BauHen}, but we defer the proof of this to Proposition~\ref{prop:refereeaddition}.
\end{remark}

\section{Partitions of flow categories}
\label{sec:partitions}

Let $\cC$ be a width-$3$ framed flow category in Chang form. For simplicity we assume that $|x|\in \{0,1,2,3\}$ for all objects $x\in \Ob(\cC)$. The purpose Sections \ref{sec:partitions} and \ref{sec:main_theorem} is to improve the flow category $\cC$ by move equivalence so that its score agrees with the score of a flow category which is the disjoint union of Baues--Hennes flow categories.

In this section we will set up the language and notation, and in Section \ref{sec:main_theorem} we will use that setup to prove the main theorem of the article.

\subsection{Parts and partitions of a flow category}\label{susbec:parts}

\begin{definition}
	For a framed flow category in Chang form, the length 2 edges connected to vertices in level 0 of the score are called \emph{$\eta$-edges}. The length 2 edges connected to vertices in level 1 of the score are called \emph{$\xi$-edges}.
	
	By convention, when a $\xi$ edge exists, we say it goes \emph{out} of a level 3 object and goes \emph{in} to a level 1 object.
\end{definition}

Since the flow category is in Chang form the score has the property that every vertex in level $2$ or $0$ is connected to at most one $\eta$-edge. The $\xi$-edges however have no restrictions.

Recall that a Baues--Hennes flow category $\B$ of a basic word has a distinguished start object and end object. We write $\alpha(\B)$ for the start object and $\omega(\B)$ for the end object. Similarly, we extend this notion for the sphere flow category $\Sp$ by $\alpha(\Sp)=\omega(\Sp)=x$, where $x$ is the only object of $\Sp$.

\begin{definition}
	Let $\cC$ be a framed flow category in Chang form. A subcategory $\cP$ of $\cC$ is called a \emph{part}, if $\cP$ is either a Baues--Hennes flow category for a basic word $w=w(\cP)$, or a sphere flow category $\Sp$. Furthermore, a part $\cP$ must satisfy the following properties.
	\begin{enumerate}
		\item $\M_\cP(x,y) = \M_\cC(x,y)$ for all objects $x,y\in \Ob(\cP)$, unless $x=\alpha(\cP)$ with $|x|=3$ and $y=\omega(\cP)$ with $|y|=1$.
		\item $\M_\cC(x,y) = \emptyset$ for all objects $x\in \Ob(\cP)$, $y\in \Ob(\cC)\setminus \Ob(\cP)$ with $|x|-|y|\leq2$, unless $x=\alpha(\cP)$ and $|x|=3$ or $2$.
		\item $\M_\cC(x,y) = \emptyset$ for all objects $x\in \Ob(\cC)\setminus \Ob(\cP)$, $y\in \Ob(\cP)$ with $|x|-|y|\leq 2$, unless $y=\omega(\cP)$ and $|y|=1$, or $y=\alpha(\cP)$, $|y|=1$ and $|x|=2$.
	\end{enumerate}
	We set $w(\Sp)=\emptyset$.
\end{definition}

\begin{remark}
	Assume $\cP$ is a part of a flow category $\cC$ in Chang form. If $x\in \Ob(\cP)$ is neither the start nor the end object of $\cP$, the $0$- and $1$-dimensional moduli spaces $\M_\cC(x,y)$ and $\M_\cC(u,x)$ with $y,u\in \Ob(\cC)$ are completely determined by the word $w(\cP)$. For the end object this is true unless $|\omega(\cP)|=1$. Let us consider the various cases for a part.
	
	If $|\alpha(\cP)|=3$ we can have the following situations.
	\begin{center}
	\begin{tikzpicture}
	\score{0.6}{12.6}
	\hookteb{0.6}{}{};
	\hookteb{2.2}{}{};
	\draw (1.4,0.6) -- (1.6,0.9);
	\draw (2,1.5) -- (2.2,1.8);
	\shookr{0.3}{1.6};
	\shookr{3}{0.6};
	\node at (1.1,1.6) {$\alpha(\cP)$};
	\node at (3.6,0.4) {$\omega(\cP)$};
	\node at (1.8,1.2) {$\cdot$};
	\node at (1.7,1.07) {$\cdot$};
	\node at (1.9,1.33) {$\cdot$};
	\hookteb{4.2}{}{};
	\hookmiet{5.8}{}{};
	\draw (5,0.6) -- (5.2,0.9);
	\draw (5.6,1.5) -- (5.8,1.8);
	\shookr{3.9}{1.6};
	\node at (4.4,1.6) {$\alpha$};
	\node at (6.9,0.2) {$\omega$};
	\node at (5.4,1.2) {$\cdot$};
	\node at (5.3,1.07) {$\cdot$};
	\node at (5.5,1.33) {$\cdot$};
	\hookteb{7.8}{}{};
	\hookt{9.4}{}{};
	\draw (8.6,0.6) -- (8.8,0.9);
	\draw (9.2,1.5) -- (9.4,1.8);
	\shookr{7.5}{1.6};
	\node at (8,1.6) {$\alpha$};
	\node at (9.7,1.4) {$\omega$};
	\node at (9,1.2) {$\cdot$};
	\node at (8.9,1.07) {$\cdot$};
	\node at (9.1,1.33) {$\cdot$};
	\hookteb{10.6}{}{};
	\dota{12.2};
	\draw (11.4,0.6) -- (11.6,0.9);
	\draw (12,1.5) -- (12.2,1.8);
	\shookr{10.3}{1.6};
	\node at (11.8,1.2) {$\cdot$};
	\node at (11.7,1.07) {$\cdot$};
	\node at (11.9,1.33) {$\cdot$};
	\node at (10.8,1.6) {$\alpha$};
	\node at (12.4,1.6) {$\omega$};
	\end{tikzpicture}
	\end{center}
	Note that there can be several $\xi$-edges going out of $\alpha(\cP)$, including one into $\omega(\cP)$, provided that $|\omega(\cP)|=1$. If $|\omega(\cP)|=0$, there is no $0$-dimensional moduli space $\M_\cC(c,\omega(\cP))\not=\emptyset$ with $c\in \Ob(\cC)$; if $|\omega(\cP)|=2$, there is no $\eta$-edge coming out of it, and if $|\omega(\cP)|=3$, there is no other $\xi$-edge coming out of it.
	
	If $|\alpha(\cP)|=2$, there is no $a\in \Ob(\cC)$ with $\M_\cC(a,\alpha(\cP))\not=\emptyset$. It is possible there is $c\in \Ob(\cC)\setminus \Ob(\cP)$ with $\M_\cC(\alpha(\cP),c)\neq\emptyset$ for some $|c|=1$, however it is not possible for $|c|=0$ because $\cC$ is in Chang form, and if there is an $\eta$ edge in coming out of $\alpha(\cP)$, we would have $c\in \Ob(\cP)$. The conditions on $\omega(\cP)$ are as in the above case.
	
	If $|\alpha(\cP)|=1$, there is no $d\in \Ob(\cC)$ with $\M_\cC(\alpha(\cP),d)\not=\emptyset$, but there could be $b\in \Ob(\cC)\setminus \Ob(\cP)$ with $\M_\cC(b,\alpha(\cP))$ a finite number of points.
	
	If $|\alpha(\cP)|=0$ there is no $\eta$-edge coming in. There could be $2$-dimensional moduli spaces $\M_\cC(a,\alpha(\cP))$, but we will only begin to worry about these in Section \ref{sec:fullBH} and they are not relevant right now.
	
	Parts are not necessarily maximal with respect to their word $w(\cP)$. In fact, we can remove a $\xi$-letter and divide a part into two parts. But this is not allowed with any of the other letters.   
\end{remark}

\begin{definition}
	Let $\cC$ be a framed flow category in Chang form. A \emph{partition of $\cC$} is a collection $\mathcal{P}=\{\cP_i\}_{i=1}^k$ of parts $\cP_i$ such that
	\[
		\Ob(\cC) = \coprod_{i=1}^k \Ob(\cP_i).
	\]
\end{definition}

\subsection{The base partition of the category}

Partitions do exist, and there is in fact a basic method to get them for any framed flow category in Chang form, as we now explain.

\begin{definition}
	Define a partition $\mathcal{T}$ of $\cC$, called the \emph{base partition}, as follows. In $\Gamma_2$, consider the connected components of the graph obtained by removing all $\xi$ edges and all $t$ edges. Each such connected component gives rise to a Baues--Hennes word and defines a part by taking the set of objects corresponding to the vertices of the connected component, together with the moduli spaces coming from the edges.
\end{definition}
Since we assume that $\cC$ is in Chang form $\mathcal{T}$ is indeed a partition.

There are 10 possible types for parts in $\mathcal{T}$ and their scores are illustrated below as $(a)$ -- $(j)$.
\begin{center}
\begin{tikzpicture}
\score{0.6cm}{12cm}
\dota{0.3};
\hookt{1.3}{};
\hookmiet{2.1}{};
\hookteb{3.9}{}{};
\hookmibe{5.6}{};
\dotc{7};
\hookb{8}{};
\hookmie{9};
\dotb{10.7};
\dotd{11.7};

\node at (0.3,-0.4) {$(a)$};
\node at (1.3,-0.4) {$(b)$};
\node at (2.5,-0.4) {$(c)$};
\node at (4.3,-0.4) {$(d)$};
\node at (6,-0.4) {$(e)$};
\node at (7,-0.4) {$(f)$};
\node at (8,-0.4) {$(g)$};
\node at (9.4,-0.4) {$(h)$};
\node at (10.7,-0.4) {$(i)$};
\node at (11.7,-0.4) {$(j)$};

\end{tikzpicture}
\end{center}

\subsection{Pre-orders on the partition}\label{sec:preorders}

In the proof of Theorem \ref{thm:chang} we moved our framed flow category into Chang form by means of an induction. The induction was controlled by first defining an appropriate pre-order on a set consisting of certain important features of the score diagram, and then  controlling the induction using maximal elements in this pre-order. That proof can be thought of as a vastly simplified version of the induction we intend to perform in Section \ref{sec:main_theorem}.

In this spirit, we will now define certain subsets -- called $\mathcal{P}_3$ and $\overline{\mathcal{P}}_1$ -- of a general partition of a flow category $\mathscr{C}$ in Chang form. We then proceed to describe pre-orders on those subsets.

\begin{example}
	Before we give the definition of $\mathcal{P}_3$ and $\overline{\mathcal{P}}_1$ for a general $\mathcal{P}$, as a rough idea for the general case, we mention that for the base partition $\mathcal{T}$, that $\mathcal{T}_3$ is made up of those parts with score $(a)$, $(b)$, $(c)$ and $(d)$, while $\overline{\mathcal{T}}_1$ will be made up of those parts with score $(d)$, $(e)$, $(f)$ and $(g)$.
\end{example}

In full generality, we now define two ways of dividing up the parts in a partition.

\begin{definition}
	Let $\mathcal{P}$ be a partition of a framed flow category $\cC$ in Chang form. Then for $i\in \{0,1,2,3\}$ let
	\[
		\mathcal{P}_i = \{\cP\in \mathcal{P}\,|\, |\alpha(\cP)| = i \}
	\]
	and let
	\[
		\overline{\mathcal{P}}_i = \{\cP\in \mathcal{P}\,|\, |\omega(\cP)| = i \}.
	\]
	
	Recalling that we use the lexicographical order to compare sequences $\sigma(w)$ for $w\in \Xi$, we define a pre-order $\preccurlyeq$ on $\mathcal{P}_3$ by $\cP_1 \preccurlyeq \cP_2$ if $\sigma(\xi w(\cP_1)) \leq \sigma(\xi w(\cP_2))$. We write $\cP_1 \prec \cP_2$, if $\cP_1 \preccurlyeq \cP_2$, but $\cP_2 \not\preccurlyeq \cP_1$.
\end{definition}

Defining the pre-order on $\overline{\mathcal{P}}_1$ will be considerably more delicate.

\begin{definition}
	Let $\overline{\Xi}$ be the set of basic words ending in $\xi$. For $w\in \overline{\Xi}$ define the \emph{$\bar{\sigma}$-symbol of $w$} to be the sequence $\bar{\sigma}(w)=(\bar{\sigma}_1(w),\ldots)$ with $\bar{\sigma}_i(w)\in \Q$ inductively as
	\[
	\bar{\sigma}(w) = \left\{
		\begin{array}{ll}
		(0,0,\ldots) & \mbox{if }\bar{w}=\xi \\
		(1/q,1,0,\ldots) & \mbox{if }\bar{w}=\xi_q \\
		(1/q,0,\ldots) & \mbox{if }\bar{w}=\xi_q\eta \\
		(1/q,-1/p,-1,0,\ldots) & \mbox{if }\bar{w}=\xi_q\eta^p \\
		(1/q,-1/p,\bar{\sigma}_1(u),\ldots) & \mbox{if }\bar{w}=\xi_q\eta^p \bar{u}\mbox{ with }u\in \overline{\Xi}
		\end{array}
	\right.
	\]
	The usual order on $\Q$ induces a lexicographic order on $\overline{\Xi}$. In particular $\bar{\sigma}(w_1)=\bar{\sigma}(w_2)$ implies $w_1=w_2$.
\end{definition}

The order on $\overline{\Xi}$ could be used to define a pre-order on $\overline{\mathcal{P}}_1$, but because of potential central words we will need to refine this further.

\begin{definition}
	Let $\mathcal{P}$ be a partition of a framed flow category $\cC$ in Chang form, and $i,j\in \{0,1,2,3\}$. Then let
	\[
		\mathcal{P}_{i,j}:=\mathcal{P}_i\cap \overline{\mathcal{P}}_j.
	\]
	For $\cP\in \mathcal{P}_{1,1}$ define 
	\[
	\tau(\cP) = \left\{
		\begin{array}{cl}
			\infty & \mbox{if }\M(b,\alpha(\cP))=\emptyset \mbox{ for all }b\in \Ob(\cC), |b|=2 \\
			p & \mbox{if there exists }b\in \Ob(\cC), |b|=2, \M(b,\alpha(\cP))\mbox{ contains }p\mbox{ points}
		\end{array}
	\right.
	\]

	For $\cP\in \mathcal{P}_{2,1}$ define 
	\[
	\tau(\cP) = \left\{
	\begin{array}{cl}
	-\infty & \mbox{if }\M(\alpha(\cP),c)=\emptyset \mbox{ for all }c\in \Ob(\cC), |c|=1 \\
	-p & \mbox{if there exists }c\in \Ob(\cC), |c|=1, \M(\alpha(\cP),c)\mbox{ contains }p\mbox{ points}
	\end{array}
	\right.
	\]
\end{definition}

\begin{definition}
Let $\mathcal{P}$ be a partition of a framed flow category $\cC$ in Chang form. Define
\begin{align*}
 \mathcal{P}_2^T&=\{\cP\in \mathcal{P}_2\,|\,\mbox{There exists }c\in \Ob(\cC) \mbox{ with }|c|=1, \M(\alpha(\cP),c)\not=\emptyset\} \\
 \mathcal{P}_1^T&=\{\cP \in \mathcal{P}_1\,|\,\mbox{There exists }b\in \Ob(\cC) \mbox{ with }|b|=2, \M(b,\alpha(\cP))\not=\emptyset\}\} \\
\end{align*}
\end{definition}

\begin{lemma}
	Let $\mathcal{P}$ be a partition of a framed flow category $\cC$ in Chang form.
	\begin{enumerate}
		\item Let $\cP\in \mathcal{P}_2^T$. Then there exists a unique $\cP'\in \mathcal{P}_1^T$ with 
		\[\M_\cC(\alpha(\cP),\alpha(\cP'))\not=\emptyset.\]
		\item Let $\cP'\in \mathcal{P}_1^T$. Then there exists a unique $\cP\in \mathcal{P}_2^T$ with 
		\[\M_\cC(\alpha(\cP),\alpha(\cP'))\not=\emptyset.\]
	\end{enumerate}
	We call the two parts $\cP$ and $\cP'$ \emph{dual} to each other.
\end{lemma}

\begin{proof}
	This is immediate from the fact that $\cC$ is in primary Smith normal form.
\end{proof}

If $\cP_1,\cP_2$ are in $\mathcal{P}_1^T$ (resp.\ $\mathcal{P}_2^T$), satisfying $w(\cP_1)=w(\cP_2)$ and $\tau(\cP_1)=\tau(\cP_2)$, then they have duals $\cP_1',\cP_2'$ in $\mathcal{P}_2^T$ (resp.\ $\mathcal{P}_1^T$). The pre-order to compare $\cP_1$ and $\cP_2$ is going to take the duals into account. Note that if $\cP'\in \mathcal{P}_1$, then $w(\cP')\in \Xi$, and if $\cP'\in \mathcal{P}_2$, then $w(\cP')\in \Eta$.

\begin{definition}
Let $\cC$ be a framed flow category in Chang form, and $\mathcal{P}$ be a partition. We define a pre-order $\trianglelefteq$ on $\overline{\mathcal{P}}_1$ as follows.

Let $\cP_1,\cP_2\in \overline{\mathcal{P}}_1$.  We write $\cP_1 \trianglelefteq \cP_2$ if one of the following is satisfied.
\begin{enumerate}
	\item We have $\bar{\sigma}(w(\cP_1)\xi) < \bar{\sigma}(w(\cP_2)\xi)$.
	\item We have $\bar{\sigma}(w(\cP_1)\xi) = \bar{\sigma}(w(\cP_2)\xi)$, and if $\cP_1, \cP_2\in \mathcal{P}_{1,1}\cup \mathcal{P}_{2,1}$ we have $\tau(\cP_2)=\infty$ or $\tau(\cP_1)=-\infty$ or $\tau(\cP_1)<\tau(\cP_2)$.
	\item We have $\cP_1,\cP_2\in \mathcal{P}_{1,1}$, $\bar{\sigma}(w(\cP_1)\xi)=\bar{\sigma}(w(\cP_2)\xi)$, $\tau(\cP_1)=\tau(\cP_2)\in \Z$ and $\rho(w(\cP_2'))\leq \rho(w(\cP_1'))$, where $\cP_1'$ is the dual of $\cP_1$ and $\cP_2'$ is the dual of $\cP_2$.
	\item We have $\cP_1,\cP_2\in \mathcal{P}_{2,1}$, $\bar{\sigma}(w(\cP_1)\xi)=\bar{\sigma}(w(\cP_2)\xi)$, $\tau(\cP_1)=\tau(\cP_2)\in \Z$ and $\sigma(w(\cP_2'))\leq \sigma(w(\cP_1'))$, where $\cP_1'$ is the dual of $\cP_1$ and $\cP_2'$ is the dual of $\cP_2$.
\end{enumerate}
We also write $\cP_1 \triangleleft \cP_2$, if $\cP_1 \trianglelefteq \cP_2$ and $\cP_2 \not\trianglelefteq \cP_1$.
\end{definition}

So when comparing $\cP_1,\cP_2\in \mathcal{P}_{1,1}\cup\mathcal{P}_{2,1}$ with $w(\cP_1)=w(\cP_2)$, we first compare $\tau(\cP_1)$ and $\tau(\cP_2)$, and if these happen to be the same integer, we still have to compare the duals. Note that when comparing duals, the order is reversed.

\subsection{The incidence matrix}

The start vertex of a part in $ \mathcal{P}_3$ may be connected to the end vertex of a part in $ \overline{\mathcal{P}}_1$ by some number of $\xi$ edges. The ways in which these connections are made will now be encoded in a matrix. Much of the induction in Section \ref{sec:main_theorem} is concerned with simplifying such matrices in a methodical fashion using flow category moves.

\begin{definition}\label{def:incidence}
	Let $\cC$ be a framed flow category in Chang form, and let $\mathcal{P}$ be a partition of $\cC$. For $\cP\in \mathcal{P}_3$ and $\cP'\in \overline{\mathcal{P}}_1$ define $[\cP:\cP']\in \F_2$ to be
	\[
	[\cP:\cP'] = \left\{ \begin{array}{cl}
		1 & \mbox{if }\M(\alpha(\cP),\omega(\cP'))= \eta\\
		0 & \mbox{if }\M(\alpha(\cP),\omega(\cP'))=\emptyset
	\end{array}
	\right.
	\]
	We will sometimes write $[\cP:\cP']_\cC$ to emphasize the dependence on the flow category.
	
	Let $V_\alpha$ be the $\F_2$-vector space with basis $\mathcal{P}_3$, and let $V_\omega$ be the $\F_2$-vector space with basis $\overline{\mathcal{P}}_1$. Also, let $A(\cC,\mathcal{P})\colon V_\alpha \to V_\omega$ be given by
	\[
	A(\cC,\mathcal{P})(\cP) = \sum_{\cP'\in \overline{\mathcal{P}}_1} [\cP:\cP']\cP'. 
	\] 
	for $\cP\in \mathcal{P}_3$. 
\end{definition}

In a slight abuse of language (as the bases $\mathcal{P}_3$ and $\overline{\mathcal{P}}_1$ have not been ordered) we will consider $A(\cC,\mathcal{P})$ as a matrix in $\F_2$, called the \emph{incidence matrix of the pair $(\cC,\mathcal{P})$}. 

The pre-order $\preccurlyeq$ on $\mathcal{P}_3$ induces an equivalence relation on $\mathcal{P}_3$ by $\cP \approx \cP'$ if $\cP \preccurlyeq \cP'$ and $\cP'\preccurlyeq \cP$. We write the equivalence class of $\cP$ as $[\cP]$.

Similarly, the pre-order $\trianglelefteq$ on $\overline{\mathcal{P}}_1$ induces an equivalence relation on $\overline{\mathcal{P}}_1$ which we also write as $\cP_1\approx \cP_2$. Again we write the equivalence class of $\cP$ as $[\cP]$.

Notice that on $\mathcal{P}_{3,1}= \mathcal{P}_3\cap \overline{\mathcal{P}}_1$ these equivalence relations both reduce to $w(\cP)=w(\cP')$, so we get the same equivalence classes on $\mathcal{P}_{3,1}$.

If $A(\cC,\mathcal{P})$ is the zero matrix, then the score of $\cC$ agrees with the score of a framed flow category which is a disjoint union of Moore, Chang, and Baues-Hennes flow categories. So if we could use move-equivalence to achieve this zero matrix, we would be done.

However, there is one group of Baues-Hennes flow categories that do not have zero incidence matrix, namely the categories corresponding to cyclic words. For cyclic words we need to understand the sub-matrices of $A(\cC,\mathcal{P})$ corresponding to equivalence classes $[\cP]$ for $\cP\in \mathcal{P}_{3,1}$.

\begin{definition}
Let $\cP_1,\cP_2\in \mathcal{P}_3$. Define $E_\alpha^{\cP_1,\cP_2}\colon V_\alpha \to V_\alpha$ by
\[
 E_\alpha^{\cP_1,\cP_2}(\cP) = \left\{ \begin{array}{cc}
                                        \cP & \mbox{if } \cP\not = \cP_2 \\
                                        \cP_1+\cP_2 & \mbox{if }\cP = \cP_2
                                       \end{array}
\right.
\]
Let $\cP_1,\cP_2\in \overline{\mathcal{P}}_1$. Define $E_\omega^{\cP_1,\cP_2}\colon V_\omega \to V_\omega$ by
\[
 E_\omega^{\cP_1,\cP_2}(\cP) = \left\{ \begin{array}{cc}
                                        \cP & \mbox{if } \cP\not = \cP_2 \\
                                        \cP_1+\cP_2 & \mbox{if }\cP = \cP_2
                                       \end{array}
\right.
\]

\end{definition}

\begin{proposition}\label{prop:duking_one}
Let $\mathcal{P}$ be a partition of a framed flow category $\cC$ in Chang form, and let $\cP_1,\cP_2\in \mathcal{P}_3$ with $\cP_1\not=\cP_2$.
\begin{enumerate}
 \item Assume that $\cP_1 \approx \cP_2$. Then there exists a framed flow category $\cC'$ in Chang form move equivalent to $\cC$ with $\Ob(\cC')=\Ob(\cC)$ such that $\mathcal{P}$ is a partition of $\cC'$ and
\[
 A(\cC',\mathcal{P}) = \left\{ \begin{array}{cl}
                                A(\cC,\mathcal{P}) E_\alpha^{\cP_1,\cP_2} & \mbox{if }\cP_1,\cP_2\notin \mathcal{P}_{3,1}\\
                                E_\omega^{\cP_1,\cP_2} A(\cC,\mathcal{P}) E_\alpha^{\cP_1,\cP_2} & \mbox{if }\cP_1,\cP_2\in \mathcal{P}_{3,1}
                               \end{array}
\right.
\]
\item Assume that $\cP_2 \prec \cP_1$ and $w(\cP_1)\not= w(\cP_2)\xi u$ for some $\xi u \in \Xi$. Then there exists a framed flow category $\cC'$ in Chang form move equivalent to $\cC$ with $\Ob(\cC')=\Ob(\cC)$ such that $\mathcal{P}$ is a partition of $\cC'$ and
\[
 A(\cC',\mathcal{P}) = A(\cC,\mathcal{P}) E_\alpha^{\cP_1,\cP_2}.
\]
\end{enumerate}
\end{proposition}

Right multiplication by $E_\alpha^{\cP_1,\cP_2}$ means that every $\xi$-edge going out of $\alpha(\cP_1)$ and into some object $c$ leads to an extra $\xi$-edge going out of $\alpha(\cP_2)$ and into $c$. If there was already a $\xi$-edge, the two edges are removed, corresponding to a $0$ in the incidence matrix.

Left multiplication by $E_\omega^{\cP_1,\cP_2}$ means that every $\xi$-edge going into $\omega(\cP_2)$ from some object $a$ leads to an extra $\xi$-edge going into $\omega(\cP_1)$, again with two identical $\xi$-edges cancelling each other. Note that this only occurs for $\cP_1,\cP_2\in \mathcal{P}_{3,1}$.

\begin{proof}
The proof of (1) is by induction on the length of $w(\cP_1)$. Slide $\alpha(\cP_2)$ over $\alpha(\cP_1)$ so that every $\xi$-edge coming out of $\alpha(\cP_1)$ is now also coming out of $\alpha(\cP_2)$. If $w(\cP_1) = \emptyset$ we are done after extended Whitney tricks to remove double $\xi$-edges. If $w(\cP_1)$ starts with $^p$ the resulting flow category is not in primary Smith normal form, but this can be fixed with another handle slide. If $w(\cP_1) = \,\!^p\eta$ or $^p\eta_q$, we have an extra $\eta$-edge preventing $\cC'$ from being in Chang form, compare Figure \ref{fig:first_move}.

\begin{figure}[ht]
\begin{tikzpicture}
\score{0.6}{3.2}
\hookbet{0.4}{$p$}{$q$};
\hookteb{2}{$p$}{$q$};
\draw (1.2,1.2) -- (2.8,0);
\shookl{0.4}{0.6};
\node at (0.8,2.2) {$\cP_2$};
\node at (2.4,2.2) {$\cP_1$};
\end{tikzpicture} 
\label{fig:first_move}
\caption{$\cC'$ after two handle slides.}
\end{figure}
Using Lemma \ref{lem:hook_be} (2) we remove this extra $\eta$-edge at the cost of copying all $\xi$-edges ending in $\omega(\cP_2)$ to also end in $\omega(\cP_1)$ (if $w(\cP_1)$ ends in $_q$). This proves (1) for $w(\cP_1)$ a subword of $^p\eta_q$.

If we inductively assume that $w(\cP_1)=\,\!^{p_1}\eta_{q_1}\xi u$ for some $\xi u\in \Xi$, we can proceed as before to get a score as in Figure \ref{fig:second_move}.

\begin{figure}[ht]
\begin{tikzpicture}
\score{0.6}{6.8}
\hookbet{0.4}{$p_2$}{$q_2$};
\hookmix{1.2};
\hookbet{2}{$p_1$}{$q_1$};
\hookteb{4}{$p_1$}{$q_1$};
\hookx{4.8};
\hookteb{5.6}{$p_2$}{$q_2$};
\draw (1.2,1.8) -- (4.8,0.6);
\draw (0.4,0.6) -- (0.2,0.9);
\draw (6.4,0.6) -- (6.6,0.9);
\end{tikzpicture}
\label{fig:second_move}
\caption{$\cC'$ after Lemma \ref{lem:hook_be}(2).}
\end{figure}
For $i=1,2$ break the parts $\cP_i$ into parts $\cP_i',\cP_i''$ with $w(\cP_i')=\,\!^{p_1}\eta_{q_1}$ and $w(\cP_i'')=u$. There are two $\xi$-edges going out of $\alpha(\cP_2'')$, one ending in $\omega(\cP_2')$ and the other in $\omega(\cP_1')$, while only one $\xi$-edge is going out of $\alpha(\cP_1'')$, ending in $\omega(\cP_1')$.

Apply induction to $\cP_1''\approx \cP_2''$, and then put the parts $\cP_i',\cP_i''$ back together, for $i=1,2$.

To prove (2), we also begin by sliding $\alpha(\cP_2)$ over $\alpha(\cP_1)$. This is enough if $w(\cP_1)=\emptyset$, and if $w(\cP_1)$ starts with $^p$ and $w(\cP_2)$ with $^{p'}$ where $p>p'$ we have to make an even number of slides to get the flow category back into Smith normal form. We may get an even number of $\eta$ edges this way, which can be removed with Whitney tricks. After this we get the result. If $w(\cP_1)= \,\!^p$ and $w(\cP_2)$ starts with $^p\eta$, we only need to do this slide once to get the result.

If $w(\cP_1)$ starts with $^p\eta q$, and $w(\cP_2)=\,\!^p\eta$ or starts with $^p\eta _{q'}$ where $q<q'$, begin with two handle slides and then Lemma \ref{lem:hook_be} (5) or (1) to get the result.

Inductively, assume $w(\cP_1)=u\xi v$ and $w(\cP_2)=u\xi v'$ with $\sigma(\xi v) < \sigma(\xi v')$. Then use part (1) on $u$, followed by the induction on $v$.
\end{proof}

\begin{proposition}\label{prop:duking_two}
Let $\mathcal{P}$ be a partition of a framed flow category $\cC$ in Chang form, let $\cP_1,\cP_2\in \overline{\mathcal{P}}_1$ with $\cP_1\not=\cP_2$.
\begin{enumerate}
 \item Assume that $\cP_1\approx \cP_2$, and if $\cP_1,\cP_2\in \mathcal{P}_{1,1}\cup \mathcal{P}_{2,1}$ additionally assume that $\tau(\cP_1),\tau(\cP_2)\in \{\pm \infty\}$. Then there exists a framed flow category $\cC'$ in Chang form, move equivalent to $\cC$, with $\Ob(\cC')=\Ob(\cC)$, such that $\mathcal{P}$ is a partition of $\cC'$, and such that
\[
 A(\cC',\mathcal{P}) = \left\{ \begin{array}{cl}
                                E_\omega^{\cP_1,\cP_2} A(\cC,\mathcal{P}) & \mbox{if }\cP_1,\cP_2\notin \mathcal{P}_{3,1}\\
                                E_\omega^{\cP_1,\cP_2} A(\cC,\mathcal{P}) E_\alpha^{\cP_1,\cP_2} & \mbox{if }\cP_1,\cP_2\in \mathcal{P}_{3,1}
                               \end{array}
\right.
\]
\item Assume that $\bar{\sigma}(w(\cP_1)\xi) < \bar{\sigma}(w(\cP_2)\xi)$ and $w(\cP_2)\not= v\xi w(\cP_1)$ for any basic word $v\xi\in \overline{\Xi}$. Then there exists a framed flow category $\cC'$ in Chang form, move equivalent to $\cC$, with $\Ob(\cC')=\Ob(\cC)$, such that $\mathcal{P}$ is a partition of $\cC'$, and such that
\[
 A(\cC',\mathcal{P}) = E_\omega^{\cP_1,\cP_2} A(\cC,\mathcal{P}).
\]
\end{enumerate}
\end{proposition}

\begin{proof}
The proof of (1) is very similar to the proof of Proposition \ref{prop:duking_one} (1) and is also done by using induction on the length of $w(\cP_2)$. We begin with sliding $\omega(\cP_2)$ over $\omega(\cP_1)$ which results in the left multiplication by $E_\omega^{\cP_1,\cP_2}$. Let us consider the various cases depending on whether $\cP_1,\cP_2$ are in $\mathcal{P}_{1,1}$, $\mathcal{P}_{2,1}$ and $\mathcal{P}_{3,1}$.

If $\cP_2\in \mathcal{P}_{2,1}$, the last slide is moving $\alpha(\cP_2)$ over $\alpha(\cP_1)$.
\begin{center}
\begin{tikzpicture}
\score{0.6}{4}
\hookmie{0.4};
\node at (0.4,1.5){$\alpha(\cP_2)$};
\draw (1.2,0) -- (1.2,0.2);
\node at (1.2,0.275) {$\cdot$};
\draw (1.2,0.4) -- (1.2,0.6);
\dotc{1.2};
\hookx{1.2};
\dotc{2.8};
\draw (2.8,0) -- (2.8,0.2);
\node at (2.8,0.275) {$\cdot$};
\draw (2.8,0.4) -- (2.8,0.6);
\node at (3.6,1.5) {$\alpha(\cP_1)$};
\draw[gray,->] (0.6,1.25) to[out=10,in=170] (3.4,1.25);
\hooke{2.8};
\draw (0.4,1.2) -- (2.8,0);
\node at (4.5,0.9) {$\sim$};
\scorex{0.6}{4}{5}
\hookmie{5.4};
\draw (6.2,0) -- (6.2,0.2);
\node at (6.2,0.275) {$\cdot$};
\draw (6.2,0.4) -- (6.2,0.6);
\dotc{6.2};
\hookx{6.2};
\dotc{7.8};
\draw (7.8,0) -- (7.8,0.2);
\node at (7.8,0.275) {$\cdot$};
\draw (7.8,0.4) -- (7.8,0.6);
\hooke{7.8};
\end{tikzpicture}
\end{center}
Since the condition $\tau(\cP_1)=-\infty$ means that there is no $c\in \Ob(\cC)\setminus \Ob(\cP_1)$ with $\M(\alpha(\cP_1),c)\not=\emptyset$, this case is now finished.

If $\cP_1\in \mathcal{P}_{3,1}$ we essentially have the same picture, but with $|\alpha(\cP_2)|=|\alpha(\cP_1)|=3$. After the previous slide the graph is as below, and after a slide of $\alpha(\cP_2)$ over $\alpha(\cP_1)$, noting that all $\xi$-edges going out of $\alpha(\cP_1)$ now also go out of $\alpha(\cP_2)$, explaining the multiplication by $E_{\alpha}^{\cP_1,\cP_2}$.
\begin{center}
\begin{tikzpicture}
\score{0.6}{4}
\hookmiet{0.4}{};
\draw (1.2,0) -- (1.2,0.2);
\node at (1.2,0.275) {$\cdot$};
\draw (1.2,0.4) -- (1.2,0.6);
\dotc{1.2};
\hookx{1.2};
\dotc{2.8};
\draw (2.8,0) -- (2.8,0.2);
\node at (2.8,0.275) {$\cdot$};
\draw (2.8,0.4) -- (2.8,0.6);
\hooket{2.8}{};
\draw (0.4,1.8) -- (3.6,1.2);
\shookl{3.9}{1.6};
\draw[gray,->] (0.6,1.85) to[out=10,in=170] (3.4,1.85);
\node at (4.5,0.9) {$\sim$};
\scorex{0.6}{4}{5}
\hookmiet{5.4}{};
\draw (6.2,0) -- (6.2,0.2);
\node at (6.2,0.275) {$\cdot$};
\draw (6.2,0.4) -- (6.2,0.6);
\dotc{6.2};
\hookx{6.2};
\dotc{7.8};
\draw (7.8,0) -- (7.8,0.2);
\node at (7.8,0.275) {$\cdot$};
\draw (7.8,0.4) -- (7.8,0.6);
\hooket{7.8}{};
\shookl{8.9}{1.6};
\shookl{5.7}{1.6};
\end{tikzpicture}
\end{center}
If $\cP_2\in \mathcal{P}_{1,1}$ we end up again with the situation that we have to slide $\alpha(\cP_2)$ over $\alpha(\cP_1)$, and since there are no objects $d\in \Ob(\cC)\setminus\Ob(\cP_1)$ with $\M(\alpha(\cP_1),d)\not=\emptyset$, we are finished.
\begin{center}
\begin{tikzpicture}
\score{0.6}{5.6}
\hookx{0.4};
\hookmiet{1.2}{};
\draw (2,0) -- (2,0.2);
\node at (2,0.275) {$\cdot$};
\draw (2,0.4) -- (2,0.6);
\dotc{2};
\hookx{2};
\dotc{3.6};
\draw (3.6,0) -- (3.6,0.2);
\node at (3.6,0.275) {$\cdot$};
\draw (3.6,0.4) -- (3.6,0.6);
\hooket{3.6}{};
\hookmix{4.4};
\draw (1.2,1.8) -- (5.2,0.6);
\draw[gray,->] (0.6,0.55) to[out=-10,in=-170] (5,0.55);
\node at (6.1,0.9) {$\sim$};
\scorex{0.6}{5.6}{6.6}
\hookx{7};
\hookmiet{7.8}{};
\draw (8.6,0) -- (8.6,0.2);
\node at (8.6,0.275) {$\cdot$};
\draw (8.6,0.4) -- (8.6,0.6);
\dotc{8.6};
\hookx{8.6};
\dotc{10.2};
\draw (10.2,0) -- (10.2,0.2);
\node at (10.2,0.275) {$\cdot$};
\draw (10.2,0.4) -- (10.2,0.6);
\hooket{10.2}{};
\hookmix{11};
\end{tikzpicture}
\end{center}
The proof of (2) is again similar, with the condition $\bar{\sigma}(w(\cP_1)\xi) < \bar{\sigma}(w(\cP_2)\xi)$ and $w(\cP_2)\not= v\xi w(\cP_1)$ implying that at some point we have to make an even number of slides to get the flow category back into primary Smith normal form, which implies that we do not get $\xi$ or $\eta$-edges preventing the parts from being in Baues--Hennes form.
\end{proof}

\begin{proposition}\label{prop:duking_three}
Let $\mathcal{P}$ be a partition of a framed flow category $\cC$ in Chang form, let $\cP_1,\cP_2\in \mathcal{P}_{1,1}\cup \mathcal{P}_{2,1}$ with $\cP_1\not=\cP_2$ such that $w(\cP_1) = w(\cP_2)$ and $\tau(\cP_1),\tau(\cP_2)\in \Z$. Let $\cP_1'$ be the dual of $\cP_1$, and $\cP_2'$ be the dual of $\cP_2$.
\begin{enumerate}
\item Assume that $\cP_1\approx \cP_2$. Then there is a framed flow category $\cC'$ in Chang form move equivalent to $\cC$ with $\Ob(\cC')=\Ob(\cC)$ and $\mathcal{P}$ a partition of $\cC'$, such that
\[
 A(\cC',\mathcal{P}) = \left\{ \begin{array}{cl}
                                E_\omega^{\cP_1,\cP_2} A(\cC,\mathcal{P}) & \mbox{if }\cP_1',\cP_2'\notin \overline{\mathcal{P}}_1\\
                                E_\omega^{\cP_1',\cP_2'} E_\omega^{\cP_1,\cP_2} A(\cC,\mathcal{P}) & \mbox{if }\cP_1',\cP_2'\in \overline{\mathcal{P}}_1
                               \end{array}
\right.
\]
\item Assume that $\cP_1 \triangleleft \cP_2$ and that $w(\cP_1')\not= w(\cP_2')\xi u$ for any word $\xi u\in \Xi$. Then there is a framed flow category $\cC'$ in Chang form move equivalent to $\cC$ with $\Ob(\cC')=\Ob(\cC)$ and $\mathcal{P}$ a partition of $\cC'$, such that
\[
 A(\cC',\mathcal{P}) = E_\omega^{\cP_1,\cP_2} A(\cC,\mathcal{P}).
\]
\end{enumerate}
\end{proposition}

\begin{proof}
The proof starts in the same way as the proof of Proposition \ref{prop:duking_two}. If $\cP_1,\cP_2\in \mathcal{P}_{2,1}$ we get a situation where we need to slide $\alpha(\cP_2)$ over $\alpha(\cP_1)$.
\begin{center}
\begin{tikzpicture}
\score{0.6}{4}
\hookmie{0.4};
\hookm{0.4}{$p$};
\draw (1.2,0) -- (1.2,0.2);
\node at (1.2,0.275) {$\cdot$};
\draw (1.2,0.4) -- (1.2,0.6);
\dotc{1.2};
\hookx{1.2};
\dotc{2.8};
\draw (2.8,0) -- (2.8,0.2);
\node at (2.8,0.275) {$\cdot$};
\draw (2.8,0.4) -- (2.8,0.6);
\hooke{2.8};
\hookmim{3.6}{$q$};
\draw (0.4,1.2) -- (2.8,0);
\draw[gray,->] (0.6,1.25) to[out=10,in=170] (3.4,1.25);
\node at (4.5,0.9) {$\sim$};
\scorex{0.6}{4}{5}
\hookmie{5.4};
\hookm{5.4}{$p$};
\draw (6.2,0) -- (6.2,0.2);
\node at (6.2,0.275) {$\cdot$};
\draw (6.2,0.4) -- (6.2,0.6);
\dotc{6.2};
\hookx{6.2};
\dotc{7.8};
\draw (7.8,0) -- (7.8,0.2);
\node at (7.8,0.275) {$\cdot$};
\draw (7.8,0.4) -- (7.8,0.6);
\hooke{7.8};
\hookmim{8.6}{$q$};
\draw (5.4,1.2) -- (8.6,0.6);
\node at (7,0.7) {$q$};
\end{tikzpicture}
\end{center}
To get this back into primary Smith normal form, we need to slide $\alpha(\cP_2')$ over $\alpha(\cP_1')$. If $p < q$ we have to slide an even number of times, and no further $\xi$-edges prevent the parts to be in Baues--Hennes form. If $p=q$, possible $\xi$-edges going into $\alpha(\cP_2')$ will now lead to new $\xi$-edges going into $\alpha(\cP_1')$. If $\omega(\cP_2')=\alpha(\cP_2')$ this reflects the extra factor $E_\omega^{\cP_1',\cP_2'}$. 
Otherwise we can temporarily break the dual partitions into $\cP_2''$ and $\Sp_2$, and $\cP_1''$ and $\Sp_1$, such that $w(\cP_i') = \xi w(\cP_i'')$ and $\Ob(\Sp_i)=\{\alpha(\cP_i')\}$ for $i=1,2$. Note that there is exactly one $\xi$-edge going out of $\alpha(\cP_1'')$ and into $\alpha(\cP_1')$, while there is a $\xi$-edge going out of $\alpha(\cP_2'')$ going into $\alpha(\cP_2')$, and another $\xi$-edge going out of $\alpha(\cP_2'')$ and into $\alpha(\cP_1')$. We now use Proposition \ref{prop:duking_one} with $\cP_2''\preccurlyeq \cP_1''$ to finish the argument. Note that $\cP_1\trianglelefteq \cP_2$ implies $\cP_2''\preccurlyeq \cP_1''$.

The case $\cP_1,\cP_2\in \mathcal{P}_{1,1}$ is similar, and will be omitted.
\end{proof}

\subsection{Admissible partitions}

We must now consider a subtlety of the hypotheses of the Propositions \ref{prop:duking_one}, \ref{prop:duking_two} and \ref{prop:duking_three}. Part of the hypotheses of Proposition~\ref{prop:duking_one}~(2) require that $w(\cP_1) \not= w(\cP_2)\xi u$ for any basic word $u$. Similar conditions are present in Propositions \ref{prop:duking_two} (2) and \ref{prop:duking_three} (2). But we will also need to use similar results to these Propositions in the cases when $w(\cP_1)=w(\cP_2)\xi u$, so we discuss this further and show how to proceed in that situation by adding an extra condition to the partitions called \emph{admissibility}.

\begin{construction}\label{rem:no_worries}

Assume that $\mathcal{P}$ is a partition of a framed flow category $\cC$ in Chang form, and let $\cP_1,\cP_2\in \overline{\mathcal{P}}_1$ with $\cP_1\not=\cP_2$. Assume moreover that $\cP_2 \prec \cP_1$ and $w(\cP_1)= w(\cP_2)\xi u$ for some $\xi u \in \Xi$.

Temporarily break the part $\cP_1$ into two parts $\cP_1'$ and $\cP_1''$ with $w(\cP_1')=w(\cP_2)$ and $w(\cP_1'')=u$. Then $\cP_1'\approx \cP_2$ and we can apply Proposition \ref{prop:duking_one} (1) to $\cP_1'$ and $\cP_2$ resulting in a framed flow category where all $\xi$-edges going into $\omega(\cP_2)$ now have copies going into $\omega(\cP_1')$.
\begin{center}
\begin{tikzpicture}
\score{0.6}{7.6}
\hookbet{0.4}{$p_k$}{$q_k$};
\draw (1.2,1.8) -- (1.4,1.5);
\draw [dotted] (1.5,1.35) -- (1.7,1.05);
\draw (1.8,0.9) -- (2,0.6);
\hookbet{2}{$p_1$}{$q_1$};
\hookteb{3.6}{$p_1$}{$q_1$};
\draw (4.4,0.6) -- (4.6,0.9);
\draw [dotted] (4.7,1.05) -- (4.9,1.35);
\draw (5,1.5) -- (5.2,1.8);
\hookteb{5.2}{$p_k$}{$q_k$};
\hookx{6};
\hookmit{6.8}{$p_{k+1}$};
\draw (6.8,1.2) -- (7,0.9);
\draw [dotted] (7.1,0.75) -- (7.3,0.45);
\shookl{0.4}{0.6};
\shookl{6}{0.6};
\node at (1.4,2.2) {$\cP_2$};
\node at (5,2.2) {$\cP_1'$};
\node at (7.2,2.2) {$\cP_1''$};
\end{tikzpicture}
\end{center}
We cannot simply piece $\cP_1'$ and $\cP_1''$ together again to get back the part $\cP_1$, because there are now too many $\xi$-edges going into $\omega(\cP_1')$.
But we can apply Proposition~\ref{prop:duking_one}~(2) for those parts $\cP$ with $[\cP:\cP_2]=1$ and $\cP\prec \cP_1''$, removing some of the extra $\xi$-edges going into $\omega(\cP_1')$.  Similarly, one could remove unwanted $\xi$-edges using Proposition \ref{prop:duking_two}~(2) and Proposition \ref{prop:duking_three}~(2), provided the respective hypotheses of these propositions were satisfied.

If all the unwanted $\xi$-edges could be removed then we could combine the parts $\cP_1'$ and $\cP_1''$ and get back to our original partition.
\end{construction}

It would be highly desirable to be able to remove all of the $\xi$-edges in the manner described in Construction \ref{rem:no_worries}. The following definition of \em admissible \em partition will allow us to indeed remove all extra $\xi$-edges.

Notice that if $\cP\in \mathcal{P}$, we get
\[
 w(\cP) = w(\cP_1)\xi w(\cP_2) \xi \cdots \xi w(\cP_k)
\]
with $\cP_i\in \mathcal{T}$ for $i=1,\ldots,k$ and $k\geq 1$.


\begin{definition}
Let $\cC$ be a framed flow category in Chang form and $\mathcal{P}$ a partition of $\cC$. We write 
\[
\mathcal{P}_3^\cC=\{ \cP\in \mathcal{P}_3\,|\,\mbox{there exists }\cP'\in \overline{\mathcal{P}}_1 \mbox{ with }[\cP:\cP']=1\}.
\]

We call $\mathcal{P}$ \emph{admissible}, if the following conditions are satisfied. For all $\cP_i\in \mathcal{P}$ with
\begin{equation}\label{eq:wordcpi}
 w(\cP_i)=w(\cP_{i,1})\xi \cdots \xi w(\cP_{i,k_i}),
\end{equation}
where $\cP_{i,j}\in \mathcal{T}$ for $j=1,\ldots,k_i$ and $k_i\geq 2$, we have:
\begin{enumerate}
 \item[(A1)] For $u\geq 2$ we have
\[
 \sigma(\xi w(\cP)) \leq \sigma(\xi w(\cP_{i,u})\xi w(\cP_{i,u+1})\xi \cdots \xi w(\cP_{i,k_i}))
\]
for all $\cP\in \mathcal{P}_3^\cC$.
 \item[(A2)] If there is a $\cP\in \mathcal{P}_3^\cC$ with
\[
 w(\cP) = w(\cP_{i,u})\xi w(\cP_{i,u+1})\xi \cdots \xi w(\cP_{i,k_i})
\]
for some $u\geq 2$, then
\[
 \bar{\sigma}(w(\cP')\xi) \leq \bar{\sigma}(w(\cP_{i,1})\xi \cdots \xi w(\cP_{i,u-1})\xi)
\]
for all $\cP'\in \overline{\mathcal{P}}_1$ with $[\cP:\cP']=1$.
\end{enumerate}
\end{definition}

 Note that the base partition $\mathcal{T}$ is always admissible, as each $k_i=1$ and there is nothing to check.

\section{Proof of main theorem}
\label{sec:main_theorem}
We prove the main theorem in two stages. The first stage focusses on making the score of the category be in the correct form. The second stage makes the necessary further modifications so that moduli spaces not seen by the score are also of the correct form.

Specifically, in Section \ref{sec:almostBH} we show how to reduce to \emph{almost Baues--Hennes form}. 

\begin{definition}
Let $\cC$ be a framed flow category in Chang form. We say that $\cC$ is in \emph{almost Baues--Hennes form}, if its score $\Sigma(\cC)$ agrees with the score $\Sigma(\cC')$ of a framed flow category $\cC'$ which is the disjoint union of Moore, Chang, and Baues--Hennes flow categories. 
\end{definition}

(Note that for almost Baues--Hennes form we do not require Baues--Hennes flow categories for $\varepsilon$-words, as $2$-dimensional moduli spaces are not recognized by the score. For similar reasons, cyclic words need not be special in almost Baues--Hennes form.)

Once the framed flow category is in almost Baues--Hennes form, we describe in Section \ref{sec:fullBH} how to make the further reductions to complete the proof of our main theorem.

\subsection{Reduction to almost Baues-Hennes form}\label{sec:almostBH}

\begin{proposition}
	\label{prop:almost_BH}
	Let $\cC$ be a framed flow category in Chang form. Then $\cC$ is move equivalent to a framed flow category $\cC'$ in almost Baues--Hennes form.
\end{proposition}

The proof of Proposition \ref{prop:almost_BH} is the most technical argument in the article, so we provide an outline before proceeding. The overarching idea is that we start with a framed flow category $\cC$ in Chang form together with a partition $\mathcal{P}$, and from this we wish to produce a new partition with fewer parts. This process is then repeated until we cannot reduce any more. By the way we reduce, it will be clear that at termination the category is in almost Baues-Hennes form.

In more detail, the main inductive step is as follows. Suppose there is a part $\cP$ of $\mathcal{P}$ which has $\xi$-edges coming out of $\alpha(\cP)$. There may be multiple $\xi$-edges coming out, and they may go into multiple other parts of the partition. We want to repeatedly apply Propositions \ref{prop:duking_one}, \ref{prop:duking_two} and \ref{prop:duking_three} in order that  afterwards there is exactly one $\xi$-edge coming out of $\alpha(\cP)$, and thus only going into $\omega(\cP')$ for a unique other part $\cP'$ of $\mathcal{P}$. Propositions \ref{prop:duking_one}, \ref{prop:duking_two} and \ref{prop:duking_three} will also be applied to ensure that no other $\xi$-edges go into $\omega(\cP')$. With this achieved, we will redefine the partition $\mathcal{P}$ so that $\cP$ and $\cP'$ are combined into a new and bigger part, thus reducing the number of parts in the partition. (This should be compared with the proof of Theorem \ref{thm:chang}, where a similar, but much simpler, version of the argument was employed to systematically eradicate multiple unwanted $\eta$ edges.)

A danger when applying Propositions \ref{prop:duking_one}, \ref{prop:duking_two}, and \ref{prop:duking_three} to eradicate $\xi$-edges, as suggested above, is that it could lead to the unwanted breaking up of other parts of the partition, defeating the ultimate point of the process which is to reduce the number of parts. The preorders of Section \ref{sec:preorders} were introduced to guide a systematic approach that prevents this unwanted outcome. Precisely, we will always choose $\cP$  maximal with respect to $\preccurlyeq$, and will always choose $\cP'$ maximal with respect to $\trianglelefteq$. If this choice gives $\cP\not\approx \cP'$, we proceed to eradicate $\xi$-edges as described in the previous paragraph, then combine $\cP$ and $\cP'$ into one larger part as desired. We would then move to the next step of the induction.

However, if we have $\cP\approx \cP'$, we need to do something different. In this case we consider the restriction of the incidence matrix $A(\cC,\mathcal{P})$ (Definition \ref{def:incidence}) to the parts in the equivalence class of $\cP$. This gives a square matrix that can be put into a primary decomposition, with blocks that are either Jordan blocks with $0$ on the diagonal, or invertible matrices with irreducible characteristic polynomial.

Invertible matrices with irreducible characteristic polynomial give rise to a cyclic Baues--Hennes flow category that we can isolate from the rest of the flow category. With this cyclic flow category put aside, we then focus on the remainder of the flow category and continue to the next step of the induction.

Jordan blocks with $0$ on the diagonal mean that we need to choose another part $\cP''$ and produce a new part $\cP_{\new}$ with $w(\cP_{\new})=w(\cP'')\left(\xi w(\cP)\right)^k$ where $k$ is the size of the Jordan block. This gives rise to a new admissible partition with fewer parts, and we can argue with induction.


\begin{proof}[Proof of Proposition \ref{prop:almost_BH}]
Consider the collection of pairs $(\cC,\mathcal{P})$, where $\cC$ is a framed flow category in Chang form and $\mathcal{P}$ is an admissible partition. Given such a pair, we claim that $\cC$ is move equivalent to a framed flow category in almost Baues--Hennes form.

The proof is by induction on the number of parts in $\mathcal{P}^\cC_3$. If $\mathcal{P}^\cC_3$ is empty, $\cC$ is in almost Baues--Hennes form where $\cC'$ is the disjoint union of the parts\footnote{The dual parts from $\mathcal{P}_1^T$ and $\mathcal{P}_2^T$ need to be combined to a central Baues--Hennes flow category.} in $\mathcal{P}$.

Let $(\cC,\mathcal{P})$ be a pair where $\cC$ is a framed flow category in Chang form and $\mathcal{P}$ is an admissible partition. In $\mathcal{P}_3^\cC$ let $\cP_{\max}$ be maximal with respect to $\preccurlyeq$, and let $\cP_{\omax}'$ be maximal with respect to $\trianglelefteq$ among 
\[
 C'=\{\cP'\in \overline{\mathcal{P}}_1 \,|\, \mbox{there exists }\cP\approx \cP_{\max}\mbox{ with }[\cP:\cP']=1\}
\]
If there is a $\cP'\in C'$ with $\cP'\approx \cP_{\omax}'$ with $\cP'\in \mathcal{P}_3^\cC$, we choose $\cP_{\omax}'$ to also be in $\mathcal{P}_3^\cC$.
After possibly renaming $\cP_{\max}$ we can assume $[\cP_{\max}:\cP_{\omax}']=1$.

Write
\[
 w(\cP_{\max}) = w(\cP_1)\xi \cdots \xi w(\cP_k)
\]
with $k\geq 1$ and $\cP_i\in \mathcal{T}$, and
\[
 w(\cP_{\omax}') = w(\cP_1')\xi \cdots \xi w(\cP_l')
\]
with $l\geq 1$ and $\cP_j'\in \mathcal{T}$.
\subsection*{Case 1}Assume $\cP_{\omax}'\not\approx \cP_{\max}$. Note that for 
\[
 w_u = w(\cP_u)\xi \cdots \xi w(\cP_k)
\]
with $u\geq 2$ we have $\sigma(\xi w(\cP_{\max})) < \sigma(\xi w_u)$, since we get $\leq$ from the admissibility condition, and we cannot have $=$ as $w_u$ is a proper subword of $w(\cP_{\max})$. In particular there is no $\cP\in \mathcal{P}_3^\cC$ with $w(\cP)=w_u$ by maximality of $\cP_{\max}$.

Suppose there exists $\cP\in \mathcal{P}_3^\cC$ satisfying $[\cP:\cP_{\omax}']=1$ and $\cP\not=\cP_{\max}$. Then $\cP\preccurlyeq \cP_{\max}$ and by Proposition \ref{prop:duking_one}  we get a move equivalent flow category $\cC'$ such that $\mathcal{P}$ is a partition of it and 
\[
A(\cC',\mathcal{P}) = \left\{ \begin{array}{cl}
                                E_\omega^{\cP_{\max},\cP} A(\cC,\mathcal{P}) E_\alpha^{\cP_{\max},\cP} & \mbox{if }\cP_{\max},\cP\in \mathcal{P}_{3,1}\mbox{ and }\cP_{\max}\approx \cP\\
                                A(\cC,\mathcal{P}) E_\alpha^{\cP_{\max},\cP} & \mbox{else}
                              \end{array}
\right.
\]
Note that if $w(\cP_{\max})= w(\cP)\xi w_u$ for some $u\geq 2$ we are in the situation of Construction \ref{rem:no_worries}. Indeed, we may have $w(\cP)= w(\cP^\ast)\xi w_{u'}$ for some $\cP^\ast \in \mathcal{P}_3^\cC$ and $u'>u$. Applying Construction \ref{rem:no_worries} with $\cP_1 = \cP_{\max}$ and $\cP_2 = \cP$ leads to potentially a few unwanted $\xi$-edges going into an object of $\cP_{\max}$, but because of the assumed admissibility condition we can remove them again.

We now have $[\cP:\cP_{\omax}']_{\cC'}=0$. Note that $\mathcal{P}$ is still admissible, as $\mathcal{P}_3^{\cC'} \subset \mathcal{P}_3^\cC$ with the only possible change that $\cP$ may no longer be in $\mathcal{P}_3^{\cC'}$. We can repeat this argument until $[\cP:\cP_{\omax}']=1$ if and only if $\cP = \cP_{\max}$. To avoid overusing primes, we shall call the current flow category $\cC$.

Now assume that $[\cP_{\max}:\cP']_\cC=1$ for some $\cP'\in \overline{\mathcal{P}}_1\setminus \{\cP_{\omax}'\}$. Using Proposition \ref{prop:duking_two} or \ref{prop:duking_three} with $\cP'\trianglelefteq \cP_{\omax}'$ and possibly a dual argument as in Construction \ref{rem:no_worries} we get a new flow category $\cC'$ such that $[\cP_{\max}:\cP']_{\cC'}=0$. Let us elaborate slightly on this dual argument. Assume that
\[
 w(\cP_{\omax}')= w(\cP_1')\xi \cdots \xi w(\cP_{u-1}')\xi w(\cP')
\]
for some $u\leq l$. By the admissibility condition (A2) we get
\[
 \bar{\sigma}(w(\cP'')\xi) \leq \bar{\sigma}(w(\cP_1')\xi \cdots \xi w(\cP_{u-1}')\xi)
\]
for all $\cP''$ with $[\cP':\cP'']=1$. As in Construction \ref{rem:no_worries} we can temporarily break $\cP_{\omax}'$ and deal with the new $\xi$-edges. Again we may have to iterate the argument if $w(\cP')=v \xi w(\cP''')$ for some $v$ and $\cP'''$.

If $\cP_{\omax}'\approx \cP'$ we may get extra $\xi$ edges going into $\omega(\cP_{\omax}')$, or $\omega(\cP_{\omax}^{\text{dual}})$ if we invoke Proposition \ref{prop:duking_three}, but because of the equivalences of the involved parts the admissibility condition still holds.

If $\cP_{\omax}',\cP'\in \mathcal{P}_1^T\cup \mathcal{P}_2^T$ with $w(\cP_{\omax}')=w(\cP')$ and $\tau(\cP_{\omax}')=\tau(\cP')$, we may have $w(\cP^{\text{dual}})=w(\cP^{\text{dual}}_{\omax})\xi u$ for some basic word $u$, and where $\cP^{\text{dual}}$ is the dual of $\cP'$, and $\cP^{\text{dual}}_{\omax}$ is the dual of $\cP_{\omax}'$. Again we temporarily break $\cP^{\text{dual}}$ and remove any extra $\xi$-edges using the admissibility condition (A1).

With the same arguments as before we see that $\mathcal{P}$ is admissible with respect to $\cC'$. Repeating this step leads to a flow category $\cC''$ with $[\cP_{\max}:\cP']_{\cC''}=1$ if and only if $\cP'=\cP_{\omax}'$, and $[\cP:\cP_{\omax}']_{\cC''}=1$ if and only if $\cP = \cP_{\max}$, and $\mathcal{P}$ is admissible with respect to $\cC''$.

We now combine the parts $\cP_{\max}$ and $\cP_{\omax}'$ to a new part $\cP_{\new}$ with $w(\cP_{\new})=w(\cP_{\omax}')\xi w(\cP_{\max})$ and get a new partition $\mathcal{P}'$ where $\cP_{\max}$ and $\cP_{\omax}'$ are replaced by $\cP_{\new}$ and the other parts remain. 

\begin{lemma}
 The new partition $\mathcal{P}'$ is an admissible partition of $\cC''$.
\end{lemma}

\begin{proof}
First observe that if $\cP_{\new}\in (\mathcal{P}_3')^{\cC''}$, then 
\begin{equation}\label{eq:pnewless}
\sigma(\xi w(\cP_{\new})) < \sigma(\xi w(\cP_{\max})).
\end{equation}
For otherwise we would have $w(\cP_{\new}) = w(\cP_{\max})\xi u$ for some $u$ by the maximality of $\cP_{\max}$ with respect to $\mathcal{P}_3^\cC$. Again by the maximality of $\cP_{\max}$ this implies $w(\cP_{\max})=w(\cP_{\omax}')\xi u'$ for some $u'$. 
But by the admissibility condition for the original $\cC$ we have $\sigma(\xi u')\geq \sigma(\xi w(\cP_{\max}))$, but since $\xi u'$ is a proper subword of $w(\cP_{\max})$, we cannot have equality. Hence 
\[
\sigma(\xi w(\cP_{\omax}')\xi w(\cP_{\max})) < \sigma(\xi w(\cP_{\omax}')\xi u')=\sigma(\xi w(\cP_{\max})), 
\]
and therefore $w(\cP_{\new})\not=w(\cP_{\max})\xi u$.

Now let $\cP_i\in \mathcal{P}'$ be different from $\cP_{\new}$. Then already $\cP_i\in \mathcal{P}$, and if $\cP\in (\mathcal{P}'_3)^{\cC''}$, we get condition (A1) either from admissibility of $\mathcal{P}$, or from (\ref{eq:pnewless}) for $\cP=\cP_{\new}$.

To see that $\cP_i$ satisfies condition (A2), assume (\ref{eq:wordcpi}) and let $\cP\in (\mathcal{P}'_3)^{\cC''}$ satisfy $w(\cP)=w(\cP_{i,u})\xi\cdots w(\cP_{i,k_i})$. Since $\cP_i, \cP_{\max}\in \mathcal{P}$, admissibility of $\mathcal{P}$ implies $\sigma(\xi w(\cP)) \geq \sigma(\xi w(\cP_{\max}))$ (by assumption $\cP$ matches a truncation of $\cP_i$). If we had $\cP = \cP_{\new}$, this would contradict (\ref{eq:pnewless}). Therefore we must have $\cP\in \mathcal{P}$ and then maximality of $\cP_{\max}$ implies $\cP\approx \cP_{\max}$.

Now let $\cP'\in \overline{\mathcal{P}'}_1$ be as in (A2). If $\cP'\not=\cP_{\new}$, then $\cP'\in \mathcal{P}$ and admissibility of $\mathcal{P}$ gives
\[
\bar{\sigma}(w(\cP')\xi) \leq \bar{\sigma}(w(\cP_{i,1})\xi \cdots \xi w(\cP_{i,u-1})\xi).
\]
In the case that  $\cP'=\cP_{\new}$ we have $[\cP:\cP_{\new}]=1$ which implies $[\cP:\cP_{\max}]=1$ in the original setting. Since we assume $\cP_{\omax}'\not\approx \cP_{\max}$ this implies 
\[
\bar{\sigma}(w(\cP_{\max})\xi) < \bar{\sigma}(w(\cP_{\omax}')\xi)\leq \bar{\sigma}(w(\cP_{i,1})\xi \cdots \xi w(\cP_{i,u-1})\xi).
\]
If $\bar{\sigma}(w(\cP_{i,1})\xi \cdots \xi w(\cP_{i,u-1})\xi) < \bar{\sigma}(w(\cP_{\new}))$ then $w(\cP_{\omax}')=v\xi w(\cP_{\max})$ which implies $\bar{\sigma}(w(\cP_{\omax}')\xi) \leq \bar{\sigma}(v\xi)$ by the original admissibility condition. Since $v$ is shorter than $w(\cP_{\omax}')$ we get $\bar{\sigma}(w(\cP_{\omax}')\xi w(\cP_{\max})\xi) \leq \bar{\sigma}(v\xi w(\cP_{\max})\xi)$ which contradicts $\bar{\sigma}(w(\cP_{i,1})\xi \cdots \xi w(\cP_{i,u-1})\xi) < \bar{\sigma}(w(\cP_{\new}))$.


It remains to show that (A1) and (A2) also hold for $\cP_{\new}$. First note that
\[
 w(\cP_{\new}) = w(\cP_1')\xi \cdots \xi w(\cP_l') \xi w(\cP_1)\xi \cdots \xi w(\cP_k).
\]
Condition (A1) is satisfied since $\sigma(\xi w(\cP_{\new})) < \sigma(\xi w(\cP_{\max}))$ if $\cP_{\new}\in (\mathcal{P}_3')^{\cC''}$. To see condition (A2) we claim that there is no $\cP\in (\mathcal{P}_3')^{\cC''}$ with
\[
 w(\cP) = w(\cP_u')\xi \cdots \xi w(\cP_l') \xi w(\cP_1)\xi \cdots \xi w(\cP_k)
\]
for some $u\in \{2,\ldots,l\}$. Clearly such $\cP$ cannot be $\cP_{\new}$, so we would already have $\cP\in \mathcal{P}_3^\cC$. But then $\cP_{\omax}'$ would violate (A1) of the admissibility condition of $\mathcal{P}$ with respect to $\cC$, as 
\[
 \sigma(\xi w(\cP)) > \sigma(\xi w(\cP_u')\xi \cdots \xi w(\cP_l')).
\]
If $\cP\in (\mathcal{P}_3')^{\cC''}$ satisfies
\[
 w(\cP) = w(\cP_u)\xi \cdots \xi w(\cP_k)
\]
for some $u\in \{2,\ldots,k\}$, we get
\begin{align*}
 \bar{\sigma}(w(\cP')\xi) &\leq \bar{\sigma}(w(\cP_1)\xi \cdots \xi w(\cP_{u-1})\xi) \\
 &< \bar{\sigma}(w(\cP_1)\xi \cdots \xi w(\cP_l') \xi w(\cP_1)\xi \cdots \xi w(\cP_{u-1})\xi)
\end{align*}
for all $\cP'\in \overline{\mathcal{P}}_1$ with $[\cP:\cP']=1$ from the old admissibility condition.

If $\cP\in (\mathcal{P}_3')^{\cC''}$ satisfies $\cP\approx \cP_{\max}$ (that is, $u=1$ above), we have
\[
 \bar{\sigma}(w(\cP')\xi) \leq \bar{\sigma}(w(\cP_1)\xi \cdots \xi w(\cP_l') \xi) = \bar{\sigma}(w(\cP_{\omax}')\xi)
\]
for all $\cP'\in \overline{\mathcal{P}}_1$ with $[\cP:\cP']=1$ by maximality of $\cP_{\omax}'$.
\end{proof}

The partition $\mathcal{P}'$ has fewer elements in $(\mathcal{P}_3')^{\cC''}$, so we can apply induction.

\subsection*{Case 2}Assume that $\cP_{\max} \approx \cP_{\omax}'$. Write $[\cP_{\max}]=\{\cP_1,\ldots,\cP_r\}$ for the equivalence class. 

Let $A$ be the matrix with entries $A_{j,i}=[\cP_i:\cP_j]$ for $i,j\in \{1,\ldots,r\}$. Using the primary decomposition \cite[\S 10.6, Thm.3]{MR0360046} we can get this matrix similar to one in the form
\[
 \left(\begin{array}{ccc}
        B_1 & 0 & 0 \\[-0.2cm]
        0 & \ddots & 0 \\
        0 & 0 & B_u
       \end{array}
\right)
\]
where each $B_j$ has characteristic polynomial $(p_j(x))^{k_j}$ with each $p_j(x)$ irreducible. Furthermore, if $p_j(x)=x^{k_j}$ we can assume that $B_j$ is a Jordan block with $0$ on the diagonal. By Proposition \ref{prop:duking_one} (1) we can assume that $A$ has this form, and as in Case 1 we get that $\mathcal{P}$ is still admissible.

If some of the $\cP_i$ are in $\mathcal{P}_3\setminus \mathcal{P}_3^\cC$, the $i$-th column in $A$ just consists of $0$, and we can perform the similarities so that $\cP_i$ corresponds to the $0$-column of a Jordan block. In particular we have $\mathcal{P}_3^{\cC'}\subset \mathcal{P}_3^\cC$ for the new flow category $\cC'$.

\subsubsection*{Case 2a} Assume that $p_j(1)=1$ for some $j$. Then the matrix $B_j$ is invertible, and after the decomposition into cyclic spaces \cite[\S 10.6, Thm.2]{MR0360046} we may assume that $B_j$ is indecomposable. Furthermore, by \cite[\S 11.3]{MR0360046} we may assume that
\[
 B_j = \left(\begin{array}{ccccc}
    a_1 & 1 & 0 & \cdots & 0 \\[-0.1cm]
    a_2 & 0 & \ddots & \ddots & \vdots \\[-0.1cm]
    \vdots & \vdots & \ddots & \ddots & 0\\[-0.1cm]
    a_{n_j-1} & \vdots & & \ddots & 1 \\[0.1cm]
    a_{n_j} & 0 & \cdots & \cdots & 0
   \end{array}
\right)
\]
with $a_{n_j}=1$. Denote by $\cP^j_1,\ldots,\cP^j_{n_j}$ the parts so that $[\cP^j_i:\cP^j_k]$ is the $(k,i)$-entry in the matrix $B_j$. 

Consider a $\cP'$ with $\cP'\vartriangleleft \cP_{\max}\approx \cP^j_i$. We want to have $[\cP^j_i:\cP']=0$ for all $i=1,\ldots,n_j$. Using Proposition \ref{prop:duking_two} (2) with $\cP' \vartriangleleft \cP^j_{n_j-1}$, we get $[\cP^j_{n_j}:\cP']=0$. Notice that $\bar{\sigma}(w(\cP'))<\bar{\sigma}(w(\cP^j_{n_j-1}))$, since $\cP^j_{n_j-1}\notin \mathcal{P}_{1,1}\cup \mathcal{P}_{2,1}$. Inductively we can achieve $[\cP^j_m:\cP']=0$ for all $m=2,\ldots,n_j$, and to achieve $[\cP^j_1:\cP']=0$ use Proposition \ref{prop:duking_two} with $\cP' \vartriangleleft \cP^j_{n_j}$.

Now consider a $\cP$ with $\cP^j_i \approx \cP_{\max} \prec \cP$. Again we want to get $[\cP:\cP^j_i]=0$ for all $i$. This is achieved similarly, using Proposition \ref{prop:duking_one}. The moves are done in the same way as in Case 1, so $\mathcal{P}$ is still admissible with respect to the new flow category.

Notice that the score of the parts $\cP^j_1,\ldots,\cP^j_{n_j}$ is now completely isolated from the score of the rest of the flow category, and the full subcategory of these objects is a cyclic Baues--Hennes flow category $\B(\xi w(\cP_{\max}),B_j)$. We can treat the rest of the flow category as a full subcategory $\cC''$ with admissible partition $\mathcal{P}''$, and since there are fewer elements in $(\mathcal{P}''_3)^{\cC''}$ than in $\mathcal{P}_3^\cC$, we can apply induction.

\subsubsection*{Case 2b} We now assume that each $B_j$ is a Jordan block with $0$ on the diagonal.

Again denote by $\cP^j_1,\ldots,\cP^j_{n_j}$ the parts so that $[\cP^j_i:\cP^j_k]$ is the $(k,i)$-entry in the matrix $B_j$. Similarly to Case 2a we can achieve that $[\cP^j_i:\cP']=0$ for all $\cP' \vartriangleleft \cP^j_i$, where $i\geq 2$, and $[\cP:\cP^j_k]=0$ for all $\cP$ with $\cP^j_k \prec \cP$, where $k\leq b_j-1$. Note that unlike in Case 2a, we may still have $\cP'$ with $[\cP^j_1:\cP']=1$ and $\cP$ with $[\cP:\cP^j_{n_j}]=1$.

It is tempting to combine the parts $\cP^j_1,\ldots,\cP^j_{n_j}$ into a part $\cP^j$, but this need not result in an admissible partition. Instead we focus on the $j$ with $n_j$ maximal. There has to be a $n_j>1$ for otherwise we would not have $\cP_{\max}\approx \cP_{\omax}'$.

So let $n_{\max}\geq 2$, and after reordering we may assume $n_j=n_{\max}$ for $j=1,\ldots,s$ and some $s\geq 1$. If there is no $\cP'$ with $[\cP^j_1:\cP']=1$ for any $j=1,\ldots,s$, we can simply combine the parts $\cP^1_1,\ldots,\cP^1_{n_1}$ into a part $\cP_{\new}$ with 
\[
 w(\cP_{\new}) = w(\cP_{\max})\xi \cdots \xi w(\cP_{\max})
\]
and we obtain a new admissible partition with fewer parts.

Otherwise let $\cP_{\omax}$ be maximal with respect to $\vartriangleleft$ among those $\cP'$ that satisfy $[\cP^j_1:\cP']=1$ for some $j\in \{1,\ldots,s\}$. After possibly reordering, we may assume $[\cP^1_1:\cP_{\omax}]=1$. Call the current flow category $\cC'$.

%

The goal is to combine $\cP^1_1,\ldots,\cP^1_{n_1}$ with $\cP_{\omax}$ to a new part $\cP_{\new}$ satisfying
\[
 w(\cP_{\new})=w(\cP_{\omax})\xi w(\cP_{\max})\xi \cdots \xi w(\cP_{\max}).
\]
Since we have $\cP_{\max}\not\approx \cP_{\omax}$, we can argue similarly to Case 1. We first want to get that $[\cP:\cP_{\omax}]=1$ if and only if $\cP=\cP^1_1$. This is done using Proposition \ref{prop:duking_one} as in Case 1. Note that if $[\cP^j_1:\cP_{\omax}]=1$ for some $j\geq 2$, we get this to be $0$ using Proposition \ref{prop:duking_one} (1), but possibly at the cost of creating extra $\xi$-edges going into $\omega(\cP^1_1)$. Using the maximality of $n_1$ these can be removed again.
After finitely many steps we can assume that $[\cP:\cP_{\omax}]=1$ if and only if $\cP=\cP^1_1$.

We now want to have $[\cP^1_1:\cP']=1$ if and only if $\cP'=\cP_{\omax}$. But this can be done using the maximality of $\cP_{\omax}$. We can now obtain the new part $\cP_{\new}$, and with the remaining parts of $\mathcal{P}$ this leads to a new partition $\mathcal{P}'$ of the flow category we now call $\cC''$. The proof that $\mathcal{P}'$ is admissible for $\cC''$ is essentially the same as in Case 1, in fact the arguments for showing property (A2) slightly simplify.

As $(\mathcal{P}'_3)^{\cC''}$ has fewer elements than the original $\mathcal{P}_3^\cC$, this finishes the proof of Proposition \ref{prop:almost_BH}.
\end{proof}

\subsection{Reduction to Baues--Hennes form}\label{sec:fullBH}

Recall that a framed flow category $\cC$ of homological width $3$ is in Baues--Hennes form, if it is a disjoint union of elementary Moore flow categories, Chang flow categories and Baues--Hennes flow categories of special words. It remains to show that a framed flow category in almost Baues--Hennes form can be turned into Baues--Hennes form.

\begin{lemma}\label{lem:spec_cyc}
	Let $(w,A)$ and $(w',A')$ be equivalent cyclic words. Then $\B(w,A)$ is move equivalent to $\B(w',A')$.
\end{lemma}

\begin{proof}
By Proposition \ref{prop:duking_one} we can replace the matrix $A$ by any matrix $B$ it is similar to. In particular, if $A$ is decomposable, we can decompose $\B(w,A)$ as well so we may as well assume that $A$ is indecomposable.

Since $A$ is indecomposable, we can think of it as a cyclic endomorphism in the sense of \cite[\S 11.3]{MR0360046}, and furthermore $A$ is similar to a matrix $B$ of the form
\[
 B=\left(\begin{array}{ccccc}
    a_0 & 1 & 0 & \cdots & 0 \\[-0.1cm]
    a_1 & 0 & \ddots & \ddots & \vdots \\[-0.1cm]
    \vdots & \vdots & \ddots & \ddots & 0\\[-0.1cm]
    a_{k-1} & \vdots & & \ddots & 1 \\[0.1cm]
    a_k & 0 & \cdots & \cdots & 0
   \end{array}
\right)
\]
Also notice that $a_k=1$ as $A$ is invertible.

Let $w = \xi^{s_1}\eta_{r_1}\cdots \xi^{s_p}\eta_{r_p}$The score of $\B(w,B)$ is given by
\begin{center}
\begin{tikzpicture}
\score{0.6}{10.4}
\hookbet{0.4}{$s_p$}{$r_p$};
\hookmix{1.2};
\hookb{2}{$r_{p-1}$};
\draw (2,0) -- (2.2,0.3);
\draw [dotted] (2.3,0.45) -- (2.5,0.75);
\draw (2.6,0.9) -- (2.8,1.2);
\hookt{2.8}{$s_1$};
\hookmix{2.8};
\hookbet{3.6}{$s_p$}{$r_p$};
\hookmix{4.4};
\hookb{5.2}{$r_{p-1}$};
\draw (5.2,0) -- (5.4,0.3);
\draw [dotted] (5.5,0.45) -- (5.7,0.75);
\draw (5.8,0.9) -- (6,1.2);
\hookt{6}{$s_1$};
\draw (6,1.8) -- (6.2,1.5);
\draw [dotted] (6.3,1.35) -- (6.5,1.05);
\node at (6.8,0.9) {$\cdots$};
\draw (7.6,0.6) -- (7.4,0.9);
\draw [dotted] (7.3,1.05) -- (7.1,1.35);
\hookbet{7.6}{$s_p$}{$r_p$};
\hookmix{8.4};
\hookb{9.2}{$r_{p-1}$};
\draw (9.2,0) -- (9.4,0.3);
\draw [dotted] (9.5,0.45) -- (9.7,0.75);
\draw (9.8,0.9) -- (10,1.2);
\hookt{10}{$s_1$};
\draw (10,1.8) -- (0.4,0.6);
\draw (10,1.8) -- (3.6,0.6);
\end{tikzpicture}
\end{center}
and there is an obvious partition of $\B(w,B)$ into $k+1$ parts $\cP_0,\ldots,\cP_k$ such that every $w(\cP_i)=\,\!^{s_1}\eta_{r_1}\cdots \xi^{s_p}\eta_{r_p}$. We assume they are ordered so that there is a $\xi$-edge going out of $\alpha(\cP_0)$ into $\omega(\cP_i)$ if and only if $a_i=1$.

Now split every part $\cP_i$ into two parts $\cP_i'$ and $\cP_i''$ such that 
\[
w(\cP_i') = \,\!^{s_1}\eta_{r_1}\cdots \xi^{s_{p-1}}\eta_{r_{p-1}}
\]
and 
\[
w(\cP_i'')= \,\!^{s_p}\eta_{r_p}.
\]
Let $i < k$ be such that $a_i=1$. We can apply Proposition \ref{prop:duking_two} (1) with $\cP_i''$ and $\cP_k''$. Then the $\xi$-edge going out of $\alpha(\cP_0')$ and into $\omega(\cP_i'')$ is no longer there, and we get an additional $\xi$-edge going out of $\alpha(\cP_k'')$ and going into $\omega(\cP_i')$.
After doing this for all $i=0,\ldots,k-1$ there is only one $\xi$-edge going out of $\alpha(\cP_0')$ which goes into $\omega(\cP_k'')$. Also, there is a $\xi$-edge going out of $\alpha(\cP_k'')$ and into $\omega(\cP_i')$ if and only if $a_i=1$.

We can therefore combine the parts $\cP_0'$ and $\cP_k''$ to a part $\tilde{\cP}_0$ such that
\[
 w(\tilde{\cP}_0) = \,\!^{s_p}\eta_{r_p}\xi ^{s_1}\eta_{r_1}\cdots \xi^{s_{p-1}}\eta_{r_{p-1}}
\]
Similarly for $i=1,\ldots,k$ we can combine the parts $\cP_i'$ and $\cP_{i-1}''$ to a part $\tilde{\cP}_i$ with $w(\tilde{\cP}_i)=w(\tilde{\cP}_0)$ and the framed flow category is a Baues--Hennes flow category for the cyclic word $(\xi w(\tilde{\cP}_0),B)$. It follows that $\B(w,A)$ is move equivalent to any $\B(w',A')$ where $(w,A)$ and $(w',A')$ are equivalent.
\end{proof}

\begin{lemma}\label{lem:spec_eps}
	Let $w=\bar{u}\varepsilon v$ be a non-special $\varepsilon$-word. Then $\B(w)$ is move equivalent to $\B(u)\sqcup \B(v)$, where we interpret $\B(\emptyset)$ as a sphere flow category of appropriate degree in case $u$ or $v$ is the empty word. 
\end{lemma}

\begin{proof}
We have $u=\,\!_2u'$ for some basic or empty word $u'$, or $v=\,\!^2v'$ for some basic or empty word. We shall assume the latter, the former case will follow by the dual argument. We can also assume that $\rho_1(\eta u)\geq 4$, for otherwise the result follows from Example \ref{exam:no2eps2}. If $v'$ is the empty word, this assumption makes $w$ special, so we can assume $v'\not=\emptyset$. If $v'=\eta$, the result follows from Example \ref{exam:Baues_style} (or rather from its dual).

That $w$ is not special means that either $\rho^\ast(v') = \rho(\eta u)$ or there is an integer $k\geq 1$ with $\rho^\ast_m(v') = \rho_m(\eta u)$ for all $m < k$, and $\rho^\ast_k(v')> \rho_k(\eta u)$.

Note that we have two parts $\B(v'')$ and $\B(u)$ as subcategories of $\B(w)$, where $v''$ is the word such that $v'=\eta v''$. If we slide $\alpha(\B(u))$ twice over $\alpha(\B(v''))$, we get
\begin{center}
\begin{tikzpicture}
\score{0.6}{3.8}
\hookbet{0.6}{$2$}{$q_1'$};
\draw (0.6,0.6) -- (0.4,0.9);
\draw [dotted] (0.3,1.05) -- (0.1,1.35);
\draw (1.4,1.8) -- (2.6,0);
\hookmib{2.6}{$q_1\geq 4$};
\draw (2.6,0.6) -- (2.8,0.9);
\draw [dotted] (2.9,1.05) -- (3.1,1.35);
\node at (4.2,0.9) {$\sim$};
\scorex{0.6}{3.8}{4.6}
\hookbet{5.2}{$2$}{$q_1'$};
\draw (5.2,0.6) -- (5,0.9);
\draw [dotted] (4.9,1.05) -- (4.7,1.35);
\draw (6,1.8) -- (7.2,0);
\hookmib{7.2}{$q_1\geq 4$};
\draw (7.2,0.6) -- (7.4,0.9);
\draw [dotted] (7.5,1.05) -- (7.7,1.35);
\draw (7.2,0) -- (5.2,0.6);
\draw [dashed] (7.2,0) -- (6,1.2);
\node at (5.7,0.15) {$2q_1'$};
\end{tikzpicture}
\end{center}
where the dashed line indicates two non-trivially framed circles in a $1$-dimensional moduli space. Note that performing an extended Whitney trick on this moduli space will also remove the $\varepsilon$-edge as in Example \ref{exam:Baues_style}.

To get the resulting framed flow category back into primary Smith normal form, we need to perform another handle slide at level $1$. For this to work, we need $2q_1' \geq q_1$ which is exactly the condition $\rho_1^\ast(v') \geq \rho_1(\eta u)$. If $2q_1' > q_1$ we need to do this handle slide an even number of times, leading to the required result.

If $2q_1'= q_1$, the score after this handle slide and the cancellation of the $\varepsilon$-edge is
\begin{center}
\begin{tikzpicture}
\score{0.6}{4.4}
\draw [dotted] (0.1,0.45) -- (0.3,0.75);
\draw (0.4,0.9) -- (0.6,1.2);
\hookt{0.6}{$p_2'$};
\hookmix{0.6};
\hookbet{1.4}{$2$}{$q_1'$};
\hookmib{3}{$q_1$};
\hookx{3};
\hookmit{3.8}{$p_2$};
\draw (0.6,1.8) -- (3,0.6);
\draw (3.8,1.2) -- (4,0.9);
\draw [dotted] (4.1,0.75) -- (4.3,0.45);
\end{tikzpicture}
\end{center}
noting that if $v'' = \,\!_{q_1'}$ the two $\xi$-edges are not there and we are done, and if $v''=\,\!_{q_1'}\xi v'''$ the condition $\rho^\ast(v')\geq \rho(\eta u)$ implies that there is a $\xi$ in $u$ and $p_2 \geq p_2'$.

Let the word $u'$ satisfy $u = \,\!_{q_1}\xi u'$. We then have two parts $\B(v''')$ and $\B(u')$ and the condition $\rho^\ast(v')\geq \rho(\eta u)$ implies $\B(v''')\preccurlyeq \B(u')$. By Proposition \ref{prop:duking_one} we can remove the $\xi$-edge and get that $\B(w)\sim \B(u)\sqcup \B(v)$. Note that if $w(\B(u'))=w(\B(v'''))\xi u''$ we can still use Proposition \ref{prop:duking_one} following Construction \ref{rem:no_worries}, as there are no $\xi$-edges going into $\omega(\B(v'''))$.
\end{proof}

\begin{proposition}\label{prop:fin_BH}
Let $\cC$ be a framed flow category in almost Baues--Hennes form. Then $\cC$ is move equivalent to a framed flow category in Baues--Hennes form.
\end{proposition}

\begin{proof}
Since $\cC$ is in almost Baues--Hennes form, there exist finitely many full subcategories $\B_1,\ldots,\B_k$ with each $\B_i$ an elementary Moore flow category or a Baues--Hennes flow category for some basic, central or cyclic word $w_i$, and every object of $\cC$ is in exactly one $\B_i$. Furthermore, the only non-empty moduli spaces $\M(x,y)$ with $x$ and $y$ in different $\B_i$ are $2$-dimensional. We can assume that such a moduli space then consists of a non-trivially framed torus. Also, if $\widetilde{H}_\ast(\cC)$ contains odd torsion, it will come from a disjoint summand Moore flow category.

We prove the proposition by induction on $k$. If $k=1$ we are done unless $w_1$ is a non-special cyclic word. But a Baues--Hennes flow category of a non-special cyclic word can be written as a disjoint union of Baues--Hennes flow categories of special cyclic words.

For the induction step, pick $\B_1$ and let $\M(a,d)=\varepsilon$ be a non-empty $2$-dimensional moduli space with $|\Ob(\B_1)\cap \{a,d\}|=1$. Without loss of generality assume $a\in \Ob(\B_1)$ and $d\in \Ob(\B_l)$ with $l > 1$. If there is a $\xi$-edge going out of $a$, or an $\eta$-edge going into $d$, we can remove $\varepsilon$ from $\M(a,d)$ using Lemma \ref{lem:trick2}. If this is not possible, then $a$ and $d$ have to be the start objects of their respective Baues--Hennes flow category. It is easy to see that neither $\B_1$ nor $\B_l$ can be Baues--Hennes flow categories of a central word or a cyclic word in this case.

So the only object in $\B_1$ which may have an $\varepsilon$ moduli space is the start object $\alpha(\B_1)$, and $w_1$ is a basic word, and in that case we can combine it with $\B_l$ for some $l>1$ where the same applies. We then can combine $\B_1$ and $\B_l$ to a Baues--Hennes flow category of an $\varepsilon$-word. If this word is not special, we can use Lemma \ref{lem:spec_eps} to split it back so that there are no $\varepsilon$ moduli spaces in $\B_1$. Then $\cC$ is move equivalent to $\B_1 \sqcup \cC'$, where $\cC'$ has the subcategories $\B_2,\ldots,\B_k$ so that we can use induction on $\cC'$. 

If there are no $\varepsilon$ moduli spaces involving objects in $\B_1$, then $\B_1$ also splits off as a disjoint summand and with the argument for the $k=1$ case we get the result.
\end{proof}

We are almost in a position to prove Theorem~\ref{thm:main}. It only remains to show the following result, foreshadowed in Remark~\ref{rem:promise}.

\begin{definition} Let $n\geq 4$. When $w$ is a basic, central or $\varepsilon$-word, we write $X(w)$ for the CW complex described by Baues--Hennes in \cite[Definition~3.5]{BauHen}, whose homology is supported in degrees $n,\dots,n+3$. When $(w,A)$ is a cyclic word, we write $X(w,A)$ for the CW complex described by Baues--Hennes in \cite[Definition~3.7]{BauHen}, whose homology is supported in degrees $n,\dots,n+3$
\end{definition}

\begin{proposition}\label{prop:refereeaddition} Let $n\geq 4$. When $w$ is a special basic, central or $\varepsilon$-word, there is respectively a basic, central or $\varepsilon$-Baues--Hennes flow category $\cC(w)$ of the word $w$, with $\mathcal{X}(\cC(w))\simeq\Sigma^\infty X(w)$. When $(w,A)$ is a special cyclic word, there is a cyclic Baues--Hennes flow category $\cC(w,A)$ of the word $(w,A)$, with $\mathcal{X}(\cC(w,A))\simeq\Sigma^\infty X(w,A)$.
\end{proposition}

\begin{proof} For $w$ and $(w,A)$ as in the proposition statement, we may apply Proposition~\ref{prop:CWisflow} to the respective spaces $X(w)$ and $X(w,A)$ to obtain framed flow categories $\cC(w)$ or $\cC(w,A)$ with $\mathcal{X}(\cC(w))\simeq\Sigma^\infty X(w)$ or $\mathcal{X}(\cC(w,A))\simeq\Sigma^\infty X(w,A)$ respectively. These framed flow categories are of homological width 3, so by Theorem~\ref{thm:chang}, Proposition~\ref{prop:almost_BH} and Proposition~\ref{prop:fin_BH}, we may assume that such $\cC(w)$ or $\cC(w,A)$ are in Baues--Hennes form.

We will now argue that because $\cC(w)$ is in Baues--Hennes form such that $\mathcal{X}(\cC(w))\simeq\Sigma^\infty X(w)$, we may conclude that $\cC(w)$ is a Baues--Hennes flow category of the word $w$. We make the similar claim for the case of cyclic words.

An $(n-1)$-connected $(n+3)$-dimensional CW complex $X$ determines a collection of data called an \emph{$A^3$-cohomology system}~\cite[Definition~8.5.6]{MR1404516}. Roughly, these data are comprised of: the cohomology groups of $X$ with $\Z$, $\Z/2$, and $\Z/4$ coefficients; maps $\mu_2^4$ and $\mu_4^2$ in the Bockstein long exact cohomology sequence associated to
\[
0\to \Z/2\xrightarrow{\mu_4^2}\Z/4\xrightarrow{\mu_2^4}\Z/2\to 0;
\]maps from the universal coefficient spectral sequences of $X$; first and second Steenrod square operations on $X$; and certain secondary cohomology operations of Adem type  (see~\cite[p.~273 (4)$'$, (5)$'$]{MR1404516})
\[
\begin{array}{rcl}
H^n(X;\Z/2)\supset \ker(\Sq^2) &\xrightarrow{\phi_4^2} &H^{n+3}(X;\Z/4)/\operatorname{im}(\mu_4^2\circ \Sq^2),\\
H^n(X;\Z/4)\supset \ker(\Sq^2\circ\mu_2^4) &\xrightarrow{\phi^4_2} &H^{n+3}(X;\Z/2)/\operatorname{im}(\Sq^2).
\end{array}
\]
By work of Chang and J\"{a}schke, this data set is enough to determine $X$ up to homotopy, when $n\geq 4$; see \cite[Theorem 8.5.7]{MR1404516}.

By construction of the spaces $X(w)$ and $X(w,A)$, the cohomology and second Steenrod square operations are straightforward to determine; cf.~\cite[\textsection 10.2.14]{MR1404516}. Combining this knowledge with the fact that $\cC(w)$ and $\cC(w,A)$ are disjoint unions of Moore, Chang and Baues--Hennes flow categories of special words, the scores of $\cC(w)$ and $\cC(w,A)$ are then seen to be precisely the scores of the Baues--Hennes flow categories associated to the respective words $w$ and $(w,A)$. When $w$ is not an $\varepsilon$-word, the score inquestion is connected, and this completes the proof.

In the case that $w$ is a $\varepsilon$-word, the score of $\cC(w)$ has exactly two connected components, each of which is the graph of a basic word (or dual of a basic word; see Remark~\ref{rem:confusion}). Suppose, for a contradiction, that $\cC(w)$ is not a Baues--Hennes flow category of $w$. Then the other possibility is that it is a disjoint union of two flow categories $\cC_1$ and $\cC_2$ that are each a basic Baues--Hennes, a Moore, or a Chang flow category. As they are disjoint, there is no cohomology operation from the cohomology of $\cC_1$ to $\cC_2$, or vice versa. But for a special $\varepsilon$-word, one of the secondary cohomology operations $\phi_2^4$, $\phi^2_4$ will nontrivially relate the cohomology of $\cC_1$ to $\cC_2$, or vice versa; see \cite[\textsection 10.2.15]{MR1404516}. This contradiction shows $\cC(w)$ is a Baues--Hennes flow category of $w$.

\end{proof}

\begin{proof}[Proof of Theorem \ref{thm:main}]
	By Theorem \ref{thm:chang}, Proposition \ref{prop:almost_BH} and Proposition \ref{prop:fin_BH}, we can get both $\cC_1$ and $\cC_2$ into Baues--Hennes form. That is, each of $\cC_1$ and $\cC_2$ is a disjoint union of elementary Moore flow categories, Chang flow categories and Baues--Hennes flow categories of special words. The stable homotopy types $\mathcal{X}(\cC_1)\simeq\mathcal{X}(\cC_2)$ thus decompose into respective wedge products. By Proposition~\ref{prop:refereeaddition}, these are wedge products of Moore, Chang and Baues--Hennes CW spectra. As such a decomposition is unique \cite{BauHen}, so the decompositions into Moore, Chang and Baues--Hennes flow categories for $\cC_1$ and $\cC_2$ agree. Finally, by Lemma \ref{lem:def_BH} and Lemma \ref{lem:spec_cyc}, the flow categories $\cC_1$ and $\cC_2$ are move equivalent.
\end{proof}

\section{Algorithmic aspects and a Lipshitz-Sarkar homotopy type}
\label{sec:algorithmic}

As we noted in Remark \ref{rem:algo_Chang} there is an algorithm to turn a framed flow category in primary Smith normal form into Chang form. The proof of Proposition \ref{prop:almost_BH} gives an algorithm which turns a framed flow category in Chang form into almost Baues--Hennes form. The algorithm is in fact much easier than the proof suggests. One starts with the base partition $\mathcal{T}$, searches for $\cP_{\max}$ which is maximal with respect to $\preccurlyeq$, then searches for $\cP_{\omax}'$ which is maximal with respect to $\trianglelefteq$ among $C'$.

Depending on whether $\cP_{\max}\approx \cP_{\omax}'$ one then is either in Case 1 or 2. In Case 1 combine $\cP_{\omax}'$ and $\cP_{\max}$ to a new part by using Propositions \ref{prop:duking_one}, \ref{prop:duking_two} and \ref{prop:duking_three}. Note that these propositions consist of an obvious order of handle slides. In Case 2 form the matrix $A$ and turn it into the form $B$. This can be done algorithmically, and such that we get the order for the appropriate handle slides. The two subcases are slightly different, with the first requiring again to improve a matrix, and the second being very similar to Case 1.

We then have a partition which has fewer parts, and we can repeat the argument. From now on we may run into the Problem discussed in Remark \ref{rem:no_worries}, but since the new partition is admissible, we can perform handle slides as before. Note that we do not need to check whether the new partition is admissible, this was done in the proof of Proposition \ref{prop:almost_BH}.

After finitely many steps, our flow category is in almost Baues--Hennes form. The final step to get the flow category into Baues--Hennes form merely consists of dropping $\varepsilon$ in non-special $\varepsilon$-Baues--Hennes flow categories, as this algorithm will not produce non-special cyclic Baues--Hennes flow categories.

An important source of examples for framed flow categories was constructed by Lipshitz--Sarkar \cite{LipSarKhov} who for every link $L$ constructed a framed flow category $\cC_{Kh}(L)$ which is a disjoint union 
\[
\cC_{Kh}(L) = \coprod_{q\in \Z} \cC_{Kh}^q(L)
\]
such that $\widetilde{H}^\ast(\cC_{Kh}^q(L))$ agrees with the Khovanov cohomology of $L$ in quantum degree $q$.

The flow category $\cC_{Kh}(L)$ actually depends on a diagram for $L$, and the number of objects is exponential in the number of crossings. If the number of crossings is at least $2$, the flow category will not be in primary Smith normal form. While it is theoretically no problem to turn the flow category into primary Smith normal form, the sheer number of objects can make this a time-consuming process. Furthermore, even $1$-dimensional moduli spaces can get very large and complicated during this process. 

In \cite{ALPODSn} the authors develop an algorithm which turns a flow category into primary Smith normal form, although at the price of losing information about $2$-dimensional moduli spaces. This algorithm has been implemented by the third author into a computer program \cite{SchuetzKJ}. This program turns a flow category into Chang form, but does not turn it into almost Baues--Hennes form. It turns out that for those links where the flow category can be moved into Chang form in a reasonable amount of time, the number of objects in the Chang form is so small that it can be turned into almost Baues--Hennes form by hand with only a few moves.

In many cases losing information on $2$-dimensional moduli spaces does not prevent us from determining the Baues--Hennes type, as the presence of non-trivially framed circles can remove an $\varepsilon$ as in Example \ref{exam:remove_epsilon}. There are however simple examples where this does not work, for example in the $(5,4)$-torus knot at $q=19$ it is not known whether the Baues--Hennes form contains a $\B(\varepsilon)$ or is just a disjoint union of sphere flow categories.

\begin{example}
	Consider the knot $K=13n3663$, see Figure \ref{fig:13n3663}. According to XKnotJob \cite{SchuetzKJ}, the framed flow category $\cC^{1}_{Kh}(K)$ is move equivalent to a framed flow category with score
	\begin{center}
	\begin{tikzpicture}
	\score{0.6}{4}
	\hookbet{0.4}{$2$}{$2$};
	\hookx{0.4};
	\hookmix{1.2};
	\dotb{0.4};
	\dotc{2.8};
	\hookm{3.6}{$2$};
	\draw (1.2,1.8) -- (2.8,0.6);
	\node at (-0.8,0) {$h=-1$};
	\node at (-0.8,0.6) {$h=0$};
	\node at (-0.8,1.2) {$h=1$};
	\node at (-0.8,1.8) {$h=2$};
	\end{tikzpicture}
	\end{center}
	which is in Chang form, but not in Baues--Hennes form. However, it is easy to see that using two handle slides we can turn this into Baues--Hennes form, and thus
	\[
	\cC^1_{Kh}(13n3663) \sim \B((\xi^2\eta_2,(1)),-1) \sqcup \Sp^0\sqcup \Sp^0\sqcup \Mo(\Z/2\Z,0) \sqcup \Sp^1.
	\]
	We remark that this knot is the only prime knot with up to $13$ crossings that admits a Baues--Hennes flow category of width $3$.
\end{example}

\begin{figure}[ht]
	\includegraphics[height = 4cm]{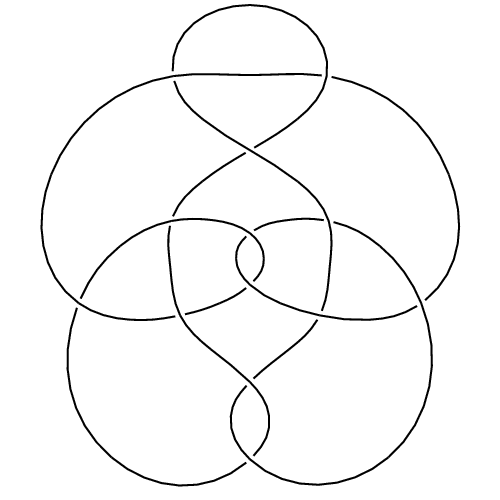}
	\caption{The knot $K=13n3663$.}
	\label{fig:13n3663}
\end{figure}

\begin{example}
	According to XKnotJob, the flow category $\cC^1_{Kh}(K)$ of the knot $K=14n8362$ is move equivalent to a flow category with score
	\begin{center}
	\begin{tikzpicture}
	\score{0.6}{4}
	\hookem{0.4}{};
	\hookm{0.4}{$2$};
	\hookx{1.2};
	\hookmix{2};
	\node at (1.1,0.9) {$2$};
	\node at (1.9,1.5) {$2$};
	\hookt{2}{};
	\hookmim{3.6}{$2$};
	\hookmix{2.8};
	\dotb{2.8};
	\node at (-0.8,0) {$h=5$};
	\node at (-0.8,0.6) {$h=6$};
	\node at (-0.8,1.2) {$h=7$};
	\node at (-0.8,1.8) {$h=8$};
	\end{tikzpicture}
	\end{center}
	which is again in Chang form, but not in Baues--Hennes form. One handle slide shows that 
	\[
	\cC^1_{Kh}(14n8362) \sim \cC(\eta2,5) \sqcup \cC(_2\eta,6) \sqcup \cC(\eta2,6) \sqcup \Mo(\Z/2\Z,6) \sqcup \Sp^7.
	\]
	There are $14$ prime knots with $14$ crossings that contain a Baues--Hennes flow category of width $3$. In fact, all of them contain $\B((\xi^2\eta_2,(1)),n)$ for some $n\in \Z$. The knot $14n8362$ was a candidate according to its Khovanov cohomology, but only contains Chang and Moore flow categories, as revealed by the computer programme and the subsequent handle slides.
\end{example}
More information on calculations can be found in \cite{ALPODSn}.


\bibliographystyle{amsalpha}
\def\MR#1{}
\bibliography{BH_references}

\end{document}